\documentclass[10pt,psamsfonts]{amsart}
\usepackage{amsthm,amsmath,amsfonts,amssymb,amsthm,graphicx,
epsf,epstopdf,color,pifont,subfigure,flafter,bm,amscd}
\usepackage{url}
\usepackage{graphicx,overpic}
\usepackage[margin=2.6cm]{geometry}

\newcommand{\ie}{{\em i.e.} }

\newtheorem {theorem} {Theorem}[section]
\numberwithin{equation}{section}
\newtheorem {prop} {Proposition}[section]
\numberwithin{equation}{section}

\numberwithin{equation}{section}
\newtheorem {lemma}  {Lemma}[section]
\numberwithin{equation}{section}

\numberwithin{equation}{section}
\newtheorem {remark}{Remark}[section]
\numberwithin{equation}{section}

\numberwithin{equation}{section}

\numberwithin{equation}{section}

\def\Blem {\begin{lem}}
\def\Elem {\end{lem}}
\def\be {\begin{equation}}
\def\ee {\end{equation}}
\def\ba {\begin{eqnarray}}
\def\ea {\end{eqnarray}}
\def\bes {\begin{equation*}}
\def\ees {\end{equation*}}
\def\bas {\begin{eqnarray*}}
\def\eas {\end{eqnarray*}}
\def\bpr {\begin{proof}}
\def\epr {\end{proof}}



\begin{document}
\allowdisplaybreaks
\title[Global Dynamics and Periodic-Annulus Bifurcations in a Reversible Quintic
Li\'enard System]
{Global Dynamics and Periodic-Annulus Bifurcations in a Reversible Quintic
Li\'enard System}

\begin{abstract}
We study the global dynamics and bifurcations of the reversible quintic
Li\'enard system
\[
 \dot{x}=y,\qquad
 \dot{y}=a_0x+a_1x^3+a_2x^5+y(b_0x+b_1x^3), 
\]
where $a_2\neq0$. Working directly in the original five-parameter space, we treat both signs
of the leading restoring coefficient in a unified framework, classify the
local dynamics, and determine all possible global
phase-portrait types on the Poincar\'e disc for configurations with one,
three, or five finite equilibria.  The resulting bifurcation
analysis includes generic and generalized
pitchfork, saddle-node,
Bogdanov--Takens, and Hopf bifurcations; nilpotent and Hamiltonian
degeneracies; bifurcations at infinity; and homoclinic, heteroclinic,
limit-cycle-breaking, and periodic-annulus-breaking transitions. 

In particular, for a parameterized class of generalized
polynomial Li\'enard systems in which the origin is the unique finite
equilibrium and is nilpotent with exactly one elliptic and one hyperbolic sector, we prove
the existence of a unique critical value of the higher-order damping
parameter in the intermediate-degree range, the nonexistence of a large-scale
periodic annulus (i.e., a periodic annulus whose closure meets the boundary circle of the Poincar\'e disc) in the critical-degree case, and its persistence when the
restoring degree dominates. In the degree range relevant to the quintic
system, the threshold is marked by a unique separatrix connection to
infinity, whereas above it infinitely many such connecting branches occur.
This transition is the large-scale
periodic-annulus--homoclinic-orbit-family bifurcation and is
established here for the one-equilibrium nilpotent configuration.
The resulting global phase-portrait classification also
reveals distinct periodic-annulus-breaking
bifurcations in configurations with multiple finite equilibria, caused by
finite separatrix connections or by separatrices reaching equilibria at
infinity.  Combining these results with Poincar\'e compactification, local
normal-form analysis, rotated-vector-field theory, and separatrix arguments,
we describe the bifurcation sets algebraically or through uniquely
characterized critical functions and obtain a complete
classification of the possible global phase-portrait types. The
bifurcation diagrams yield complete fixed-coefficient parameter
partitions outside $\Omega_{14}$; within $\Omega_{14}$, we determine the
admissible critical-surface configurations and their corresponding global
phase portraits.
\end{abstract}

\author{Chen Cao}
\address{School of Mathematical Sciences, CMA-Shanghai, Shanghai Jiao Tong
University, Shanghai 200240, P.~R. China}
\email{aceaccce@sjtu.edu.cn}

\author{Nasrin Sadri}
\address{School of Mathematics, Statistics and Computer Science, College of
Science, University of Tehran, Tehran, Iran; and School of Mathematics,
Institute for Research in Fundamental Sciences (IPM), Tehran, Iran}
\email{n.sadri@ut.ac.ir}

\author{Yilei Tang${}^{*}$}
\thanks{${}^{*}$Corresponding author.}
\address{School of Mathematical Sciences, CMA-Shanghai, Shanghai Jiao Tong
University, Shanghai 200240, P.~R. China}
\email{mathtyl@sjtu.edu.cn}

\author{Zhaoxia Wang}
\address{School of Mathematical Sciences, University of Electronic Science
and Technology of China, Chengdu, Sichuan 611731, P.~R. China}
\email{zxwang@uestc.edu.cn}


\maketitle


\section{Introduction and main results}

The investigation of limit cycles is one of the central and most difficult
themes in the qualitative theory of planar differential systems.  In the late
1920s, the work of van der Pol, Li\'enard, and Andronov connected Poincar\'e's
geometric theory with
self-sustained oscillations in vacuum-tube circuits and established limit
cycles as the mathematical description of stable oscillations
\cite{Andronov1929,Lienard1928,vanderPol1926}.  Since then, limit cycles have
become fundamental in models from mechanics and electronics to chemical and
biological kinetics.  Their number and relative configuration are also the
subject of the second part of Hilbert's sixteenth problem, which remains open
even for quadratic polynomial vector fields \cite{Hilbert1902,Smale1998}.

A particularly important subclass is formed by the Li\'enard equations
\begin{equation}\label{eq:lienard-second-order}
 \ddot{x}+f(x)\dot{x}+g(x)=0.
\end{equation}
Setting $y=\dot{x}$ gives the equivalent system
\begin{equation}\label{eq:lienard-first-form}
 \dot{x}=y,\qquad \dot{y}=-g(x)-f(x)y,
\end{equation}
whereas the change $y=\dot{x}+F(x)$, with
$F(x):=\int_0^x f(s)\,ds$, gives the Li\'enard-plane form
\begin{equation}\label{eq:lienard-second-form}
 \dot{x}=y-F(x),\qquad \dot{y}=-g(x).
\end{equation}
Motivated by the difficulty of Hilbert's problem, Smale proposed restricting
the study of limit cycles to the classical polynomial Li\'enard family
$g(x)=x$ as a more accessible, but still highly nontrivial, weakened problem
\cite{Smale1998}.

For the classical family, in which $g(x)=x$ and $F(x)$ is a polynomial of degree $m$, Lins, de Melo, and Pugh \cite{LinsMeloPugh1977} proved that
there are no limit cycles for $m\leq2$ and at most one for $m=3$, and
conjectured that the sharp upper bound for $m\geq4$ is
$\lfloor(m-1)/2\rfloor$.  The bound was subsequently
proved in \cite{LiLlibre2012} for $m=4$, and the complete bifurcation diagram and
global phase portraits of that quartic family were later obtained in
\cite{ChenChen2020}.  The conjecture is false in higher degrees: additional
relaxation cycles were constructed for $m\geq7$ and then for $m\geq6$, while
the lower bound was further raised to $m-2$ for $m\geq6$
(\cite{DeMaesschalckDumortier2011,DeMaesschalckHuzak2015,DumortierPanazzoloRoussarie2007}).
The exceptional quintic case $m=5$, however, has not been settled.  Even
though substantial information is available for its slow--fast cycles
\cite{LiLu2014}, the maximum number of limit cycles of a general classical
quintic Li\'enard equation is still unknown; see the survey
\cite{LlibreZhang2017}.

When both $F$ and $g$ are polynomials, it is customary to classify
\eqref{eq:lienard-second-form} by
$(m,n)=(\deg F, \deg g)$.  The case $m=1$ has no limit cycles by the Dulac
criterion.  Type $(2,2)$ is the Bogdanov--Takens family: Coppel proved the
uniqueness of its limit cycle, Perko obtained its global bifurcation diagram,
and the conjectured analytic properties of its homoclinic orbit were later
proved in \cite{Coppel1989,GasullEtAl2013,Perko1992}.  For type $(2,3)$,
Dumortier and Rousseau determined the bifurcation diagram and global phase
portraits up to a remaining uniqueness conjecture, which Dumortier and Li
subsequently proved \cite{DumortierLi1996,DumortierRousseau1990}.  Type
$(3,2)$ was completely analyzed by Dumortier and Li
\cite{DumortierLi1997}, and type $(4,1)$ is the quartic classical family
mentioned above.  In contrast, once $m+n\geq6$, no sharp general upper bound
is known for any unrestricted fixed type, although lower bounds and complete
global classifications have been obtained for several structured families
\cite{ChenTangWang2025,LlibreMereuTeixeira2010}.  More recently, the global
bifurcation diagrams and phase portraits of polynomial Li\'enard and van der
Pol--Duffing families with arbitrary exponents have attracted increasing
attention.  Starting from the local bifurcation analysis in \cite{Feng2014},
complete global classifications were successively obtained for the different
parity configurations of the two exponents
\cite{ChenDaiMengWang2024,ChenFengLiWang2024,ChenJiaZhangZhang2024,
ChenJinWangZhang2022,ChenLiXiaoYang2025}.  This program has recently been
extended to families containing four nonlinear monomials of arbitrary
degrees, for which global bifurcation diagrams and the corresponding phase
portraits in the Poincar\'e disc were established
\cite{WangZhouZou2026,ZhouZou2025}.  These developments show that the global
organization of arbitrary-degree Li\'enard systems depends delicately on the
relative degrees and parities of the restoring and damping terms.  The
behavior at infinity, which is indispensable for such global classifications,
was systematically organized according to the relative degrees and leading
coefficients in \cite{DumortierHerssens1999}.  A recent complementary study
classifies the local phase portraits and monodromy of polynomial generalized
Li\'enard systems at both the origin and infinity \cite{Zhang2026}.

Applications of Li\'enard theory extend well beyond abstract mathematical questions.
Li\'enard equations and their generalizations provide a common tool for
nonlinear oscillations arising in mechanics, physics, electronics, and
chemical and biological kinetics.  Applications to concrete models and the
resulting global dynamical analyses can be found, for example, in
\cite{ArtesLlibreValls2018,CaoEtAl2025,ChenChenJiaTang2023, 
ChenChenJiaTang2026,ChenLlibreTang2018,ChenTang2019,
ChenZhangZhang2023,Yu2024} and the
references therein.  Thus, existence, uniqueness, stability, and bifurcation
results for Li\'enard systems have direct consequences for sustained
oscillations in applied models.
There are also important directions beyond limit-cycle counting.  The
Darboux and Liouvillian integrability of polynomial Li\'enard systems has been
characterized in terms of invariant algebraic curves, exponential factors,
and integrating factors \cite{LlibreValls2010}.  The possible form, degree,
and number of invariant algebraic curves have generated a separate line of
research closely related to the Poincar\'e problem
\cite{Demina2018,Gine2023,QianWangYang2025,QianWangYang2026}.  

A less systematically studied parity class, and the one central to this
paper, is obtained when both $g$ and $f$ are odd, or equivalently when $g$ is
odd and $F$ is even.  In this case, the vector field of \eqref{eq:lienard-first-form} is reversible
with respect to the $y$-axis: it is invariant under
$(x,y,t)\mapsto(-x,y,-t)$.
This differs sharply from the
more familiar parity assumption in classical Li\'enard theory, in which $g$
and $F$ are odd (equivalently, $f$ is even).  Despite its natural reversible
structure, the odd--odd class has rarely been isolated and studied as a
subject in its own right.  Instead, it has appeared mainly as a special
reversible or center-producing subclass within broader investigations of
symmetries, centers and global centers, nilpotent centers, and periodic annuli
\cite{LlibreValls2022,GarciaLlibre2023,ChenLiZhang2025,
HanKongYu2026,ChenDaiFeng2024}.  Consequently, a systematic global
bifurcation classification of the quintic odd--odd class has remained largely
unavailable.  The present paper addresses this problem by
carrying out a global phase-portrait classification and bifurcation analysis
of the quintic family introduced below.  This analysis also reveals
distinctive global transitions.  Among the new
phenomena established here are periodic-annulus-breaking bifurcations, in
which an entire annulus of periodic orbits is destroyed when its separatrix
boundary changes its global connection or becomes unbounded.  Several
different mechanisms of this type occur in the family studied here.

Consider the quintic Li\'enard system
\begin{align}\label{main}
\dot{x}=&y,\\\nonumber
\dot{y}=&a_0 x+a_1 x^3+a_2 x^5+y \left( b_0 x+b_1 x^3\right),
\end{align}
where $a_2\neq 0.$
The system \eqref{main} is reversible with respect to the $y$-axis, \ie it is invariant under the transformation $(x, y, t)\rightarrow (-x, y,-t).$
Therefore, it suffices to consider the origin and the equilibria on the positive $x$-axis. The equilibria on the negative $x$-axis have the same topological type as their corresponding points on the positive $x$-axis (though the direction of the orbits is reversed).

System \eqref{main} is invariant under the transformation $(x, y, b_0, b_1) \to (-x, -y, -b_0, -b_1).$ This means that if we keep all $a_i$ fixed and let $(b_0, b_1) \to (-b_0, -b_1),$ the phase portrait of the system rotates by $180^\circ$ (central symmetry). Therefore, without loss of generality, we only consider the case where $b_1 < 0$, or the case where $b_0 \le 0$ and $b_1 = 0$.

For simplicity, we let $\lambda:=(a_0,a_1,a_2,b_0,b_1)$, $\Delta:=a_1^2-4a_0a_2$ and
 $$\Omega:=\left\{\lambda \in \mathbb{R}^5 \,\middle|\, 
a_2 \neq 0 \land \left(b_1 < 0 \lor (b_0 \leq 0 \land b_1 = 0)\right)
\right\}.
$$
The parameter regions
$\Omega_1$--$\Omega_{15}$ forming the resulting partition of
$\Omega$ are defined as follows:
\begin{align*}
\Omega_1 &:= \{\lambda\in\Omega\mid a_0=0, a_1\geq 0, a_2>0\};
& \Omega_2 &:= \{\lambda\in\Omega\mid a_0=0, a_1=0, a_2<0\};
\\
\Omega_3 &:= \{\lambda\in\Omega\mid a_0=0, a_1<0, a_2<0\};
& \Omega_4 &:= \{\lambda\in\Omega\mid a_0 \neq 0, \Delta<0\};
\\
\Omega_5 &:= \{\lambda\in\Omega\mid \Delta=0, a_0a_1>0\};
& \Omega_6 &:= \{\lambda\in\Omega\mid \Delta>0, a_0>0, a_1>0, a_2>0\};
\\
\Omega_7 &:= \{\lambda\in\Omega\mid \Delta>0, a_0<0, a_1<0, a_2<0\};
& \Omega_8 &:= \{\lambda\in\Omega\mid a_0=0, a_2<0<a_1\};
\\
\Omega_9 &:=\{\lambda\in\Omega\mid a_0=0,\ a_2>0>a_1\};
& \Omega_{10} &:= \{\lambda\in\Omega\mid \Delta=0, a_0>0, a_1<0\};
\\
\Omega_{11} &:= \{\lambda\in\Omega\mid \Delta=0, a_0<0, a_1>0\};
& \Omega_{12} &:= \{\lambda\in\Omega\mid \Delta>0, a_2<0<a_0\};
\\
\Omega_{13} &:= \{\lambda\in\Omega\mid \Delta>0, a_0<0<a_2\};
& \Omega_{14} &:= \{\lambda\in\Omega\mid \Delta>0, a_1<0, a_0>0, a_2>0\};
\\
\Omega_{15} &:= \{\lambda\in\Omega\mid \Delta>0, a_1>0, a_0<0, a_2<0\}.
\end{align*}
 
In the notation above, system~\eqref{main} has
\[
 g(x)=-(a_0x+a_1x^3+a_2x^5),\qquad
 f(x)=-(b_0x+b_1x^3),
\]
and hence
$F(x)=-b_0x^2/2-b_1x^4/4$.  It is therefore a highly nontrivial polynomial
Li\'enard system of type $(\deg F,\deg g)=(4,5)$, with as many as five finite
equilibria and with a reversing symmetry about the $y$-axis.  This type lies
well beyond the low-degree cases described above and should not be confused
with the still-open classical quintic problem, which has type $(5,1)$.
The following is one of our main results.

\begin{theorem}[Local and global bifurcation classification]\label{thm-bifurcation}
System~\eqref{main} undergoes the following local and global
bifurcations. Figure~\ref{fig:bifurcation-diagrams} gives
the complete nontrivial bifurcation diagrams outside $\Omega_{14}$.  In
$\Omega_{14}$, its three panels collectively display all admissible
relative arrangements of the critical curves and exhaust the corresponding
global phase-portrait types.
\begin{itemize}
\item[(i)] The generic pitchfork bifurcation at the origin $O=(0,0)$
occurs on
\[
P:=\{\lambda\in\Omega\mid a_0=0,\ a_1\neq0\}.
\]
Locally near $O$, the number of equilibria changes between one and
three.

\item[(ii)] The generalized pitchfork bifurcation at the origin occurs
on
\[
GP:=\{\lambda\in\Omega\mid a_0=a_1=0\}.
\]
This is the intersection of the closures in $\Omega$ of the two generic
pitchfork branches characterized by $a_1a_2>0$ and $a_1a_2<0$,
respectively.

\item[(iii)] A pair of symmetry-related saddle-node bifurcations occurs
at $E_2^\pm=\left(\pm\sqrt{-2a_0/a_1},0\right)$ on
\[
SN:=\{\lambda\in\Omega\mid \Delta=0,\ a_0a_1<0,\
a_1b_0\neq2a_0b_1\}.
\]

\item[(iv)] The Hopf bifurcation set is $H:=H_1\cup H_2\cup H_3\cup H_4$,
where
\begin{align*}
H_1&:=\left\{\lambda\in\Omega_8\ \middle|\ b_1<0,\
b_0=\frac{a_1b_1}{a_2}\right\},
&
H_2&:=\left\{\lambda\in\Omega_{12}\ \middle|\ b_1<0,\
b_0=\frac{b_1}{2a_2}(a_1+\sqrt{\Delta})\right\},
\\
H_3&:=\left\{\lambda\in\Omega_{14}\ \middle|\ b_1<0,\
b_0=\frac{b_1}{2a_2}(a_1+\sqrt{\Delta})\right\},
&
H_4&:=\left\{\lambda\in\Omega_{15}\ \middle|\ b_1<0,\
b_0=\frac{b_1}{2a_2}(a_1+\sqrt{\Delta})\right\}.
\end{align*}
The finite heteroclinic-bifurcation surfaces $OE_4$ and $E_4O$ intersect
$H_3$.  Their intersections divide $H_3$  into the following
three portions used in the global classification:
\begin{align*}
H_3^+&:=\left\{\lambda\in H_3\ \middle|\
b_0^{OE_4}<\frac{b_1}{2a_2}(a_1+\sqrt\Delta)\right\},
&
H_3^-&:=\left\{\lambda\in H_3\ \middle|\
b_0^{E_4O}<\frac{b_1}{2a_2}(a_1+\sqrt\Delta)\right\},
\\
H_3^0&:=\left\{\lambda\in H_3\ \middle|\
\frac{b_1}{2a_2}(a_1+\sqrt\Delta)
<\min\{b_0^{OE_4},b_0^{E_4O}\}\right\}.
\end{align*}
The Hopf bifurcation along $H_1$ occurs at
$E_1^\pm=(\pm\sqrt{-a_1/a_2},0)$, whereas those along $H_2$, $H_3$,
and $H_4$ occur at
$E_3^\pm=(\pm\sqrt{(-a_1-\sqrt{\Delta})/(2a_2)},0)$.

\item[(v)] A pair of symmetry-related generic Bogdanov--Takens
bifurcations occurs at $E_2^\pm$ on $BT:=(C_1\cup C_2)\cap\{b_1<0\}$,
where
\begin{align*}
C_1&:=\left\{\lambda\in\Omega_{10}\ \middle|\
b_0=\frac{2a_0b_1}{a_1}\right\},
&
C_2&:=\left\{\lambda\in\Omega_{11}\ \middle|\ b_1<0,\
b_0=\frac{2a_0b_1}{a_1}\right\}.
\end{align*}
Equivalently, $\Delta=0$, $a_0a_1<0$, and
$a_1b_0=2a_0b_1$; the Hamiltonian case $b_0=b_1=0$ is excluded.

\item[(vi)] The nilpotent center--sector degeneracy at the origin occurs
on
\[
N:=\{\lambda\in\Omega\mid a_0=0,\ a_1<0,\ b_0^2+8a_1=0\}.
\]
The periodic-annulus-breaking sets $\mathrm{PA}_2$ and $\mathrm{PA}_4$
intersect $N$ at $B^-$ and $B^+$, respectively, and divide the
corresponding components of $N$ into $\mathrm{NE}_1^-$ and
$\mathrm{NE}_2^-$, and into $\mathrm{NE}_1^+$ and $\mathrm{NE}_2^+$,
respectively.  The bifurcation set $INF$ at infinity meets the
$\Omega_3$ component of $N$ at its endpoint $\mathrm{NE}_0^-$.  The
portions of $N$ used in the global classification and these distinguished
points are defined by
\[
\begin{array}{r@{}l@{\quad}r@{}l}
B^+&\multicolumn{3}{@{}l}{\displaystyle:=
\left\{\lambda\in\Omega_9\ \middle|\
b_0=\sqrt{-8a_1},\ b_1=-\frac{\sqrt8a_2}{\sqrt{-a_1}}\right\},}
$~~ B^-:=\displaystyle\left\{\lambda\in\Omega_3\ \middle|\
b_0=-\sqrt{-8a_1},\ b_1=a_2\sqrt{-\frac2{a_1}}\right\},$
\\[2pt]
\mathrm{NE}^{+}&:=\displaystyle\{\lambda\in\Omega_3\mid
b_0=\sqrt{-8a_1}\},
\qquad\qquad\qquad\qquad  
 \mathrm{NE}^{-} :=\displaystyle\{\lambda\in\Omega_9\mid
b_0=-\sqrt{-8a_1}\},
\\[2pt]
\mathrm{NE}_1^-&\multicolumn{3}{@{}l}{\displaystyle:=
\left\{\lambda\in\Omega_3\ \middle|\
b_0=-\sqrt{-8a_1},\ b_1\in\left(a_2\sqrt{-\frac2{a_1}},0\right)\right\},}
\\[2pt]
\mathrm{NE}_2^-&\multicolumn{3}{@{}l}{\displaystyle:=
\left\{\lambda\in\Omega_3\ \middle|\
b_0=-\sqrt{-8a_1},\ b_1<a_2\sqrt{-\frac2{a_1}}\right\},}
\\[2pt]
\mathrm{NE}_0^-&:=\displaystyle\{\lambda\in\Omega_3\mid
b_0=-\sqrt{-8a_1},\ b_1=0\},
\\[2pt]
\mathrm{NE}_1^+&\multicolumn{3}{@{}l}{\displaystyle:=
\left\{\lambda\in\Omega_9\ \middle|\
b_0=\sqrt{-8a_1},\ -\frac{\sqrt8a_2}{\sqrt{-a_1}}<b_1<0\right\},}
\\[2pt]
\mathrm{NE}_2^+&\multicolumn{3}{@{}l}{\displaystyle:=
\left\{\lambda\in\Omega_9\ \middle|\
b_0=\sqrt{-8a_1},\ b_1<-\frac{\sqrt8a_2}{\sqrt{-a_1}}\right\}.}
\end{array}
\]
Across $N$, the origin changes between a nilpotent center and a
nilpotent equilibrium with one elliptic and one hyperbolic sector.

\item[(vii)] The Hamiltonian-center bifurcation set is
\[
HAM:=\{\lambda\in\Omega\mid b_0=b_1=0\}.
\]
On $HAM$ the system is Hamiltonian, and every finite antisaddle (that is, every non-saddle equilibrium) is a
center.  Its intersections with the closures of the Hopf sets are
Hamiltonian-center degeneracies rather than generic Hopf points.  Its
intersection with $\Delta=0$ and $a_0a_1<0$ is the degenerate
Hamiltonian cusp at $E_2^\pm$.

\item[(viii)] The local bifurcation at infinity occurs on
\[
INF:=\{\lambda\in\Omega\mid a_2<0,\ b_1=0\}.
\]
The local phase portraits on the two sides and in the more degenerate
Hamiltonian subcase are shown in Figures~\ref{NearInf} and~\ref{NearInf1}.

\item[(ix)] The periodic-annulus-breaking bifurcation set is
$PA:=\bigcup_{j=1}^7\mathrm{PA}_j$, where
\begin{align*}
\mathrm{PA}_1&:=\left\{\lambda\in\Omega_2\ \middle|\
b_0=-2\frac{a_2}{b_1},\ b_1<0\right\},
\\
\mathrm{PA}_2&:=\left\{\lambda\in\Omega_3\ \middle|\
b_0=\frac{a_1b_1}{a_2}-2\frac{a_2}{b_1},\
a_2\sqrt{-\frac2{a_1}}<b_1<0\right\},
\\
\mathrm{PA}_3&:=\left\{\lambda\in\Omega_8\ \middle|\
b_1<0,\ b_0=\frac{a_1b_1}{a_2}-2\frac{a_2}{b_1}\right\},
\\
\mathrm{PA}_4&:=\left\{\lambda\in\Omega_9\ \middle|\
b_1<-\frac{\sqrt8a_2}{\sqrt{-a_1}},\
b_0=\frac{a_1b_1}{2a_2}-\frac{4a_2}{b_1}\right\},
\\
\mathrm{PA}_5&:=\left\{\lambda\in\Omega_{11}\ \middle|\
b_1<0,\ b_0=\frac{2a_0b_1}{a_1}-2\frac{a_2}{b_1}\right\},
\\
\mathrm{PA}_6&:=\left\{\lambda\in\Omega_{12}\ \middle|\
b_1<0,\ b_0=b_0^{PA_6}(a_0,a_1,a_2,b_1)\right\},
\\
\mathrm{PA}_7&:=\left\{\lambda\in\Omega_{15}\ \middle|\
b_1<0,\ b_0=\frac{b_1}{2a_2}(a_1+\sqrt\Delta)
-2\frac{a_2}{b_1}\right\}.
\end{align*}
Here $b_0^{PA_6}$ is the unique parameter value at which the
fourth-quadrant stable separatrix of $O$ changes from intersecting the
positive $x$-axis to connecting with the negative-$y$ equilibrium at
infinity.  The notation $PA$ is an umbrella term: the annulus need not
be large-scale, and its critical boundary need not always be a
homoclinic-orbit family.  On $\mathrm{PA}_1$ the elliptic sector at the
nilpotent origin changes from bounded to unbounded, and the accompanying
large-scale periodic annulus breaks; its homoclinic-orbit family persists,
and the normalized critical value is $K_1^*=2$.  On $\mathrm{PA}_3$ the
large-scale homoclinic boundary enclosing the two nonzero equilibria and
its exterior periodic annulus disappear simultaneously.  On
$\mathrm{PA}_4$ the bounded periodic annulus is destroyed by the
connections between $O$ and $E_1^\pm$.  On $\mathrm{PA}_2$,
$\mathrm{PA}_5$, $\mathrm{PA}_6$, and $\mathrm{PA}_7$, the relevant
boundary separatrices of $O$, $E_2^+$, $O$, and $E_4^+$, respectively,
reach the equilibrium at infinity in the negative $y$-direction, where
$E_4^\pm=(\pm\sqrt{(-a_1+\sqrt\Delta)/(2a_2)},0)$.
The one-sided alternatives and the number of critical connecting branches
are specified in the corresponding propositions below.

\item[(x)] The finite homoclinic-bifurcation set is
$HL:=\mathrm{HL}_1\cup\mathrm{HL}_2\cup\mathrm{HL}_3\cup\mathrm{HL}_4$, where
\begin{align*}
\mathrm{HL}_1&:=\{\lambda\in\Omega_8\mid b_1<0,\ b_0=\varphi_1(a_0,a_1,a_2,b_1)\},
&
\mathrm{HL}_2&:=\{\lambda\in\Omega_{12}\mid b_1<0,\
b_0=\varphi_2(a_0,a_1,a_2,b_1)\},
\\
 \mathrm{HL}_3&:=\{\lambda\in\Omega_{14}\mid b_1<0,\ b_0=\varphi_3(a_0,a_1,a_2,b_1)\},
 &
 \mathrm{HL}_4&:=\{\lambda\in\Omega_{15}\mid b_1<0,\
b_0=\varphi_4(a_0,a_1,a_2,b_1)\}.
\end{align*}
Here, $\varphi_1$, $\varphi_2$, $\varphi_3$, and $\varphi_4$
denote the unique critical functions established in the corresponding
global-bifurcation results.  On $\mathrm{HL}_1$ and $\mathrm{HL}_2$,
the two symmetric homoclinic loops based at $O$ form a figure-eight
loop.  On $\mathrm{HL}_4$, the two symmetric homoclinic loops are based
separately at $E_4^+$ and $E_4^-$.

$\mathrm{HL}_3$ is also the unique limit-cycle-breaking bifurcation set for the
stable limit cycle. The heteroclinic surfaces $OE_4$ and $E_4O$ lie on opposite sides of
$\mathrm{HL}_3$ except when all three bifurcation surfaces meet.  Their common
intersection $\mathrm{HL}_3^0$ separates $\mathrm{HL}_3$ into the two
portions $\mathrm{HL}_3^+$ and $\mathrm{HL}_3^-$, defined together with
$\mathrm{HL}_3^0$ by
\begin{align*}
\mathrm{HL}_3^+&:=\{\lambda\in\mathrm{HL}_3\mid
\varphi_3<b_0^{E_4O}\},
\qquad
\mathrm{HL}_3^- :=\{\lambda\in\mathrm{HL}_3\mid
\varphi_3<b_0^{OE_4}\},
\\
\mathrm{HL}_3^0&:=\{\lambda\in\Omega_{14}\mid b_1<0,\
b_0=\varphi_3=b_0^{OE_4}=b_0^{E_4O}\}.
\end{align*}
On $\mathrm{HL}_3^-$, the limit cycle becomes two homoclinic loops
based at $O$, forming a figure-eight loop.  On $\mathrm{HL}_3^+$, it
becomes two homoclinic loops based separately at $E_4^+$ and $E_4^-$.
On $\mathrm{HL}_3^0$, the critical boundary is the heteroclinic loop
formed by the connections in both directions between $O$ and $E_4^\pm$.

\item[(xi)] The finite heteroclinic connection  bifurcation surfaces are defined by
\begin{align*}
E_2O&:=\{\lambda\in\Omega_{10}\mid b_1<0,\
b_0=b_0^{E_2O}(a_0,a_1,b_1)\},
&
OE_2&:=\{\lambda\in\Omega_{10}\mid
b_0=b_0^{OE_2}(a_0,a_1,b_1)\},
\\
E_4O&:=\left\{\lambda\in\Omega_{14}\ \middle|\
b_0=b_0^{E_4O},\
b_1<0\ \lor\ (b_1=0,\ b_0<0)\right\},
\\
OE_4&:=\left\{\lambda\in\Omega_{14}\ \middle|\
b_0=b_0^{OE_4},\
b_1<0\ \lor\ (b_1=0,\ b_0<0)\right\}.
\end{align*}
\begin{samepage}
The intersections of $OE_4$ with the Hopf surface $H_3$ and the
limit-cycle-breaking surface $\mathrm{HL}_3$ divide $OE_4$ into
$OE_4^+$, $OE_4^0$, and $OE_4^-$.  Similarly, the intersections of
$E_4O$ with $H_3$ and $\mathrm{HL}_3$ divide $E_4O$ into $E_4O^-$,
$E_4O^0$ and $E_4O^+$.  The intersection with $H_3$ is assigned to
$OE_4^+$ for $OE_4$ and to $E_4O^-$ for $E_4O$; the same convention
assigns a possible endpoint on the negative $b_0$-axis.  A common
endpoint at $(b_1,b_0)=(0,0)$ is excluded from these portions and is
classified separately below as $HHR=I_{57}$.  These portions used in
the global classification are defined by
\begin{align*}
OE_4^+&:=\left\{\lambda\in OE_4\ \middle|\
b_0^{OE_4}\leq\frac{b_1}{2a_2}(a_1+\sqrt\Delta)\right\},
&
OE_4^-&:=\{\lambda\in OE_4\mid \varphi_3<b_0^{OE_4}\},
\\
OE_4^0&:=\left\{\lambda\in OE_4\ \middle|\
\frac{b_1}{2a_2}(a_1+\sqrt\Delta)<b_0^{OE_4}<\varphi_3\right\},
&E_4O^0&:=\left\{\lambda\in E_4O\ \middle|\
\frac{b_1}{2a_2}(a_1+\sqrt\Delta)<b_0^{E_4O}<\varphi_3\right\},
\\
E_4O^+&:=\{\lambda\in E_4O\mid \varphi_3<b_0^{E_4O}\},
&
E_4O^-&:=\left\{\lambda\in E_4O\ \middle|\
b_0^{E_4O}\leq\frac{b_1}{2a_2}(a_1+\sqrt\Delta)\right\}.
\end{align*}
\end{samepage}
The functions $b_0^{E_2O}$, $b_0^{OE_2}$, $b_0^{E_4O}$, and
$b_0^{OE_4}$ are the unique parameter values yielding, respectively,
$E_2^+\to O$, $O\to E_2^+$, $E_4^+\to O$, and $O\to E_4^+$.
The Hamiltonian heteroclinic-reconnection set and the complete finite
heteroclinic-bifurcation set are
\begin{align*}
HHR&:=\{\lambda\in\Omega_{14}\mid a_1+2\sqrt\Delta=0,\ b_0=b_1=0\},
&
Het&:=E_2O\cup OE_2\cup E_4O\cup OE_4\cup HHR.
\end{align*}
The reversible counterparts of the four directed connections above are,
respectively, $O\to E_2^-$, $E_2^-\to O$, $O\to E_4^-$, and
$E_4^-\to O$.  On $HHR$, both directed connections occur in each
half-plane.

\end{itemize}
\end{theorem}

\begin{theorem}[Global dynamics and phase-portrait classification]\label{thm-phase-portraits}
The remaining parameter sets used to label the phase portraits are the
following.  The arguments of the critical functions introduced in
Theorem~\ref{thm-bifurcation} are omitted when no ambiguity can arise.
\\
In $\Omega_2$,
\begin{align*}
I_1&:=\{\lambda\in\Omega_2\mid b_0>0\},
& I_2&:=\{\lambda\in\Omega_2\mid b_0=0,\ b_1<0\},
\\
I_3&:=\left\{\lambda\in\Omega_2\ \middle|\
-2\frac{a_2}{b_1}<b_0<0,\ b_1<0\right\},
& I_4&:=\left\{\lambda\in\Omega_2\ \middle|\
b_0<-2\frac{a_2}{b_1},\ b_1<0\right\},
\\
I_5&:=\{\lambda\in\Omega_2\mid b_0=b_1=0\},
& I_6&:=\{\lambda\in\Omega_2\mid b_0<0,\ b_1=0\}.
\end{align*}
In $\Omega_3$,
\[
\begin{array}{@{}r@{}l@{\qquad}r@{}l@{}}
I_7&:=\{\lambda\in\Omega_3\mid b_0>\sqrt{-8a_1}\},
&I_8&:=\{\lambda\in\Omega_3\mid b_0^2+8a_1<0,\ b_1<0\},
\\[1ex]
I_9&\multicolumn{3}{l}{:=\left\{\lambda\in\Omega_3\ \middle|\
\frac{a_1b_1}{a_2}-2\frac{a_2}{b_1}<b_0<-\sqrt{-8a_1},\
b_1\in\left(a_2\sqrt{-2/a_1},0\right)\right\},}
\\[1ex]
I_{10}&\multicolumn{3}{l}{:=\left\{\lambda\in\Omega_3\ \middle|\
b_0<\frac{a_1b_1}{a_2}-2\frac{a_2}{b_1},\
b_1\in\left(a_2\sqrt{-2/a_1},0\right)\right\}}
\\[1ex]
&\multicolumn{3}{l}{\quad\cup\left\{\lambda\in\Omega_3\ \middle|\
b_0<-\sqrt{-8a_1},\ b_1\leq a_2\sqrt{-2/a_1}\right\},}
\\[1ex]
I_{11}&\multicolumn{3}{l}{:=\{\lambda\in\Omega_3\mid
-\sqrt{-8a_1}<b_0\leq0,\ b_1=0\},}
\\[1ex]
I_{12}&\multicolumn{3}{l}{:=\{\lambda\in\Omega_3\mid
b_0<-\sqrt{-8a_1},\ b_1=0\}.}
\end{array}
\]
In $\Omega_4$, $\Omega_5$  and $\Omega_7$,
\begin{align*}
I_{13}&:=\{\lambda\in\Omega_4\mid a_2>0\},
& I_{14}&:=\{\lambda\in\Omega_4\mid a_2<0,\ b_1<0\},
\\
I_{15}&:=\{\lambda\in\Omega_4\mid a_2<0,\ b_1=0\},
& I_{16}&:=\{\lambda\in\Omega_5\mid a_2>0\},
\\
I_{17}&:=\{\lambda\in\Omega_5\mid a_2<0,\ b_1<0\},
& I_{18}&:=\{\lambda\in\Omega_5\mid a_2<0,\ b_1=0\},
\\
I_{19}&:=\{\lambda\in\Omega_7\mid b_1<0\},
& I_{20}&:=\{\lambda\in\Omega_7\mid b_1=0\}.
\end{align*}
In $\Omega_8$,
\begin{align*}
I_{21}&:=\left\{\lambda\in\Omega_8\ \middle|\ b_1<0,\
b_0>\frac{a_1b_1}{a_2}\right\},
& I_{22}&:=\left\{\lambda\in\Omega_8\ \middle|\ b_1<0,\
\varphi_1<b_0<\frac{a_1b_1}{a_2}\right\},
\\
I_{23}&:=\left\{\lambda\in\Omega_8\ \middle|\ b_1<0,\
\frac{a_1b_1}{a_2}-2\frac{a_2}{b_1}<b_0<\varphi_1\right\},
& I_{24}&:=\left\{\lambda\in\Omega_8\ \middle|\ b_1<0,\
b_0<\frac{a_1b_1}{a_2}-2\frac{a_2}{b_1}\right\},
\\
I_{25}&:=\{\lambda\in\Omega_8\mid b_0=b_1=0\},
& I_{26}&:=\{\lambda\in\Omega_8\mid b_1=0,\ b_0<0\}.
\end{align*}
In $\Omega_9$,
\begin{align*}
I_{27}&:=\left\{\lambda\in\Omega_9\ \middle|\
b_0>\sqrt{-8a_1},\
\left(b_1>-\frac{\sqrt8a_2}{\sqrt{-a_1}}\ \text{or}\
b_0>\frac{a_1b_1}{2a_2}-\frac{4a_2}{b_1}\right)\right\},
\\
I_{28}&:=\left\{\lambda\in\Omega_9\ \middle|\
b_1<-\frac{\sqrt8a_2}{\sqrt{-a_1}},\
\sqrt{-8a_1}<b_0<\frac{a_1b_1}{2a_2}-\frac{4a_2}{b_1}\right\},
\\
I_{29}&:=\{\lambda\in\Omega_9\mid b_0^2+8a_1<0\},
\qquad I_{30} :=\{\lambda\in\Omega_9\mid b_0<-\sqrt{-8a_1}\}.
\end{align*}
In $\Omega_{10}$,
\begin{align*}
I_{31}&:=\{\lambda\in\Omega_{10}\mid b_1<0,\
b_0>b_0^{E_2O}(a_0,a_1,b_1)\},
\\
I_{32}&:=\left\{\lambda\in\Omega_{10}\ \middle|\ b_1<0,\
\frac{2a_0b_1}{a_1}<b_0<b_0^{E_2O}(a_0,a_1,b_1)\right\},
\\
I_{33}&:=\left\{\lambda\in\Omega_{10}\ \middle|\
b_0^{OE_2}(a_0,a_1,b_1)<b_0<\frac{2a_0b_1}{a_1}\right\},
\\
I_{34}&:=\{\lambda\in\Omega_{10}\mid
b_0<b_0^{OE_2}(a_0,a_1,b_1)\}.
\end{align*}
In $\Omega_{11}$,
\begin{align*}
I_{35}&:=\left\{\lambda\in\Omega_{11}\ \middle|\ b_1<0,\
b_0>\frac{2a_0b_1}{a_1}\right\},
& I_{36}&:=\left\{\lambda\in\Omega_{11}\ \middle|\ b_1<0,\
\frac{2a_0b_1}{a_1}-2\frac{a_2}{b_1}<b_0
<\frac{2a_0b_1}{a_1}\right\},
\\
I_{37}&:=\left\{\lambda\in\Omega_{11}\ \middle|\ b_1<0,\
b_0<\frac{2a_0b_1}{a_1}-2\frac{a_2}{b_1}\right\},
& I_{38}&:=\{\lambda\in\Omega_{11}\mid b_0=b_1=0\},
\\
I_{39}&:=\{\lambda\in\Omega_{11}\mid b_1=0,\ b_0<0\}.
\end{align*}
In $\Omega_{12}$,
\begin{align*}
I_{40}&:=\left\{\lambda\in\Omega_{12}\ \middle|\ b_1<0,\
b_0>\frac{b_1}{2a_2}(a_1+\sqrt\Delta)\right\},
&
I_{41}&:=\left\{\lambda\in\Omega_{12}\ \middle|\
b_1<0,\ \varphi_2<b_0
<\frac{b_1}{2a_2}(a_1+\sqrt\Delta)\right\},
\\
I_{42}&:=\left\{\lambda\in\Omega_{12}\ \middle|\
b_1<0,\ b_0^{PA_6}<b_0<\varphi_2\right\},
&
I_{43}&:=\{\lambda\in\Omega_{12}\mid b_1<0,\
b_0<b_0^{PA_6}\},
\\
I_{44}&:=\{\lambda\in\Omega_{12}\mid b_0=b_1=0\},
& I_{45}&:=\{\lambda\in\Omega_{12}\mid b_1=0,\ b_0<0\}.
\end{align*}
In $\Omega_{15}$,
\begin{align*}
I_{58}&:=\left\{\lambda\in\Omega_{15}\ \middle|\ b_1<0,\
b_0>\frac{b_1}{2a_2}(a_1+\sqrt\Delta)\right\},
& I_{62}&:=\{\lambda\in\Omega_{15}\mid b_0=b_1=0\},
\\
I_{59}&:=\left\{\lambda\in\Omega_{15}\ \middle|\ b_1<0,\
\varphi_4<b_0<\frac{b_1}{2a_2}(a_1+\sqrt\Delta)\right\},
& I_{63}&:=\{\lambda\in\Omega_{15}\mid b_1=0,\ b_0<0\},
\\
I_{60}&:=\left\{\lambda\in\Omega_{15}\ \middle|\
b_1<0,\ \frac{b_1}{2a_2}(a_1+\sqrt\Delta)-2\frac{a_2}{b_1}
<b_0<\varphi_4\right\},
\\
I_{61}&:=\left\{\lambda\in\Omega_{15}\ \middle|\ b_1<0,\
b_0<\frac{b_1}{2a_2}(a_1+\sqrt\Delta)-2\frac{a_2}{b_1}\right\}.
\end{align*}
In $\Omega_{14}$,
\[
\begin{array}{r@{}l@{\quad}r@{}l}
I_{46}&:=\displaystyle\{\lambda\in\Omega_{14}\mid b_1<0,\
b_0>\max\{b_0^{OE_4},b_0^{E_4O}\}\},
& I_{47}&:=\displaystyle\{\lambda\in\Omega_{14}\mid b_1<0,\
\varphi_3<b_0<b_0^{E_4O}\},
\\[2pt]
I_{48}&\multicolumn{3}{@{}l}{\displaystyle:=\left\{\lambda\in\Omega_{14}\ \middle|\ b_1<0,\
\max\left\{\frac{b_1}{2a_2}(a_1+\sqrt\Delta),b_0^{OE_4}\right\}
<b_0<\varphi_3\right\},}
\\[2pt]
I_{49}&:=\displaystyle\left\{\lambda\in\Omega_{14}\ \middle|\ b_1<0,\
b_0^{OE_4}<b_0<\frac{b_1}{2a_2}(a_1+\sqrt\Delta)\right\},
& I_{50}&:=\displaystyle\{\lambda\in\Omega_{14}\mid b_1<0,\
\varphi_3<b_0<b_0^{OE_4}\},
\\[2pt]
I_{51}&\multicolumn{3}{@{}l}{\displaystyle:=\left\{\lambda\in\Omega_{14}\ \middle|\ b_1<0,\
\max\left\{\frac{b_1}{2a_2}(a_1+\sqrt\Delta),b_0^{E_4O}\right\}
<b_0<\varphi_3\right\},}
\\[2pt]
I_{52}&\multicolumn{3}{@{}l}{\displaystyle:=\left\{\lambda\in\Omega_{14}\ \middle|\ b_1<0,\
b_0^{E_4O}<b_0<\frac{b_1}{2a_2}(a_1+\sqrt\Delta)\right\},}
\\[2pt]
I_{53}&\multicolumn{3}{@{}l}{\displaystyle:=\left\{\lambda\in\Omega_{14}\ \middle|\ b_1<0,\
\frac{b_1}{2a_2}(a_1+\sqrt\Delta)<b_0
<\min\{b_0^{OE_4},b_0^{E_4O}\}\right\},}
\\[2pt]
I_{54}&\multicolumn{3}{@{}l}{\displaystyle:=\left\{\lambda\in\Omega_{14}\ \middle|\ b_1<0,\
b_0<\min\left\{\frac{b_1}{2a_2}(a_1+\sqrt\Delta),
b_0^{OE_4},b_0^{E_4O}\right\}\right\},}
\\[2pt]
I_{55}&:=\displaystyle\{\lambda\in\Omega_{14}\mid
a_1+2\sqrt\Delta>0,\ b_0=b_1=0\},
& I_{56}&:=\displaystyle\{\lambda\in\Omega_{14}\mid
a_1+2\sqrt\Delta<0,\ b_0=b_1=0\},
\\[2pt]
I_{57}&\multicolumn{3}{@{}l}{\displaystyle:=\{\lambda\in\Omega_{14}\mid
a_1+2\sqrt\Delta=0,\ b_0=b_1=0\}.}
\end{array}
\]

All possible global phase-portrait types of
system~\eqref{main} on the Poincar\'e disc are classified in
Figures~\ref{phasea0n0-one}, \ref{phasea0n0-three},
and~\ref{phasea0n0-five}.  In $\Omega_{14}$, the listed sets describe
these types through the relative positions of the critical surfaces and
collectively exhaust all the corresponding global phase-portrait types.
\end{theorem}

\begin{remark}
For fixed $(a_0,a_1,a_2)$, Figure~\ref{fig:bifurcation-diagrams} displays
only the nontrivial bifurcation diagrams in the $(b_1,b_0)$-parameter
plane, with $b_1$ on the horizontal axis and $b_0$ on the vertical axis.
The regions $\Omega_1$, $\Omega_6$, and $\Omega_{13}$ are omitted in Figure~\ref{fig:bifurcation-diagrams}, since their phase portraits do not change for arbitrary values of $b_0$ and $b_1$; these global phase portraits are directly presented in Figures~\ref{phasea0n0-one}--\ref{phasea0n0-three}.  The regions
$\Omega_4$, $\Omega_5$, and $\Omega_7$ have no finite or global
bifurcations; when $a_2<0$, their only transition is the elementary local
bifurcation at infinity $INF$ on $b_1=0$, which is not displayed.
The displayed diagrams give complete fixed-coefficient
partitions outside $\Omega_{14}$; the three $\Omega_{14}$ panels
collectively display all admissible critical-surface arrangements and the
corresponding global phase-portrait types.
\end{remark}

\begin{figure}[t]
\begin{center}
\subfigure[$\Omega_2$]
{\includegraphics[width=.22\columnwidth,height=.22\columnwidth]{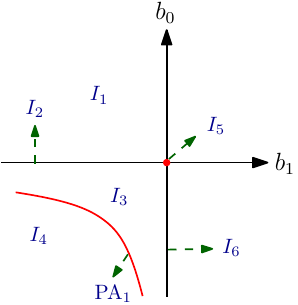}}
\quad
\subfigure[$\Omega_3$]
{\includegraphics[width=.22\columnwidth,height=.22\columnwidth]{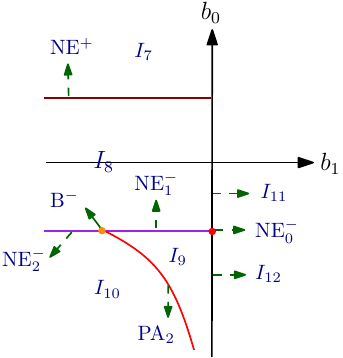}}
\quad
\subfigure[$\Omega_8$]
{\includegraphics[width=.23\columnwidth,height=.22\columnwidth]{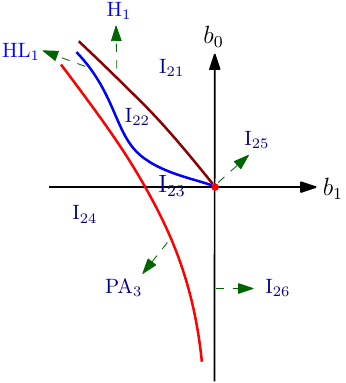}}
\quad
\subfigure[$\Omega_9$]
{\includegraphics[width=.22\columnwidth,height=.22\columnwidth]{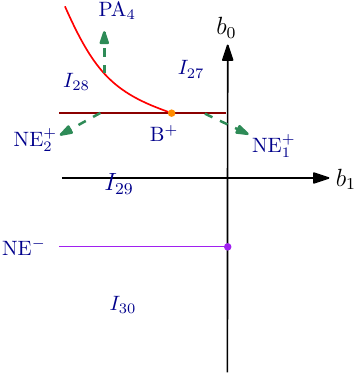}}
\\
\subfigure[$\Omega_{10}$]
{\includegraphics[width=.22\columnwidth,height=.22\columnwidth]{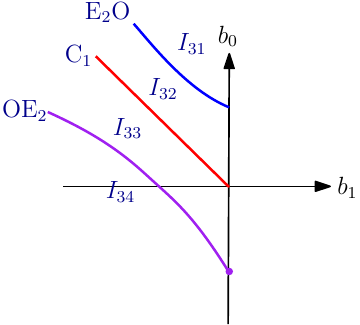}}
\quad
\subfigure[$\Omega_{11}$]
{\includegraphics[width=.22\columnwidth,height=.22\columnwidth]{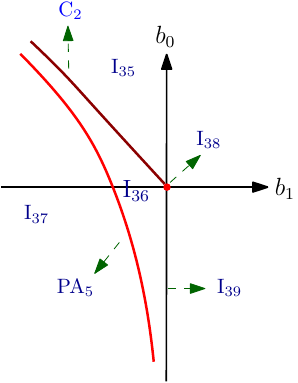}}
\quad
\subfigure[$\Omega_{12}$]
{\includegraphics[width=.22\columnwidth,height=.22\columnwidth]{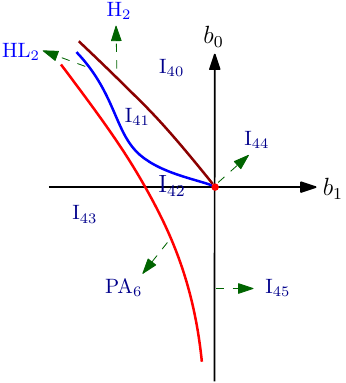}}
\quad
\subfigure[$\Omega_{15}$]
{\includegraphics[width=.22\columnwidth,height=.22\columnwidth]{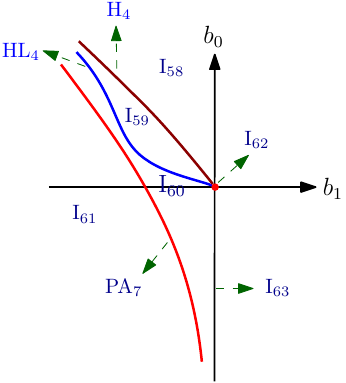}}
\\
\subfigure[$\Omega_{14} \cap \left\{ a_1+2\sqrt{\Delta}>0\right\}$]
{\includegraphics[width=.24\columnwidth,height=.24\columnwidth]{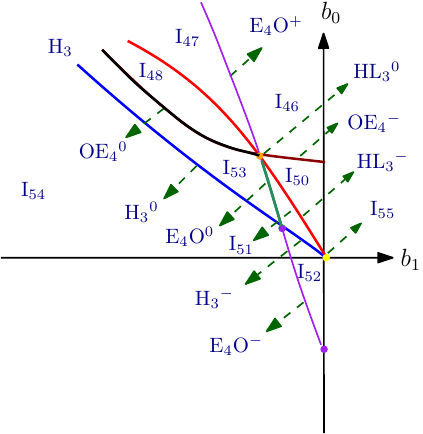}}
\quad
\subfigure[$\Omega_{14} \cap \left\{ a_1+2\sqrt{\Delta}<0\right\}$]
{\includegraphics[width=.24\columnwidth,height=.24\columnwidth]{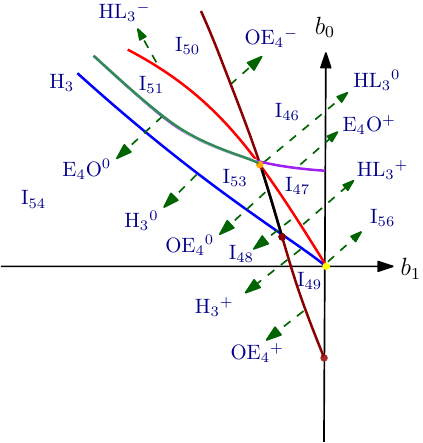}}
\quad
\subfigure[$\Omega_{14} \cap \left\{ a_1+2\sqrt{\Delta}=0\right\}$]
{\includegraphics[width=.24\columnwidth,height=.24\columnwidth]{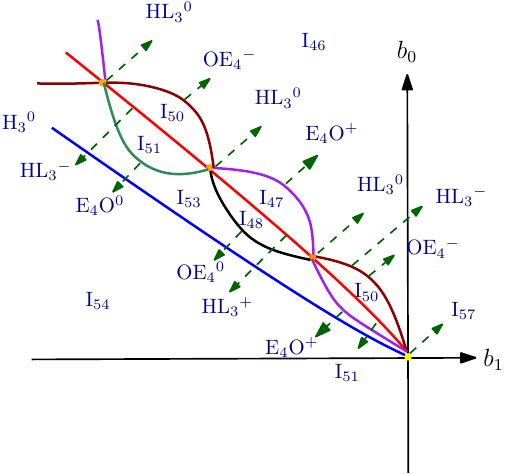}}
\end{center}
\caption{The bifurcation diagrams.
The three $\Omega_{14}$ panels collectively display all 
admissible critical-surface arrangements and the corresponding global phase-portrait types.
\label{fig:bifurcation-diagrams}}
\end{figure}

\begin{figure}
\centering
\subfigure[$\Omega_1 \cup \Omega_6 \cup I_{13} \cup I_{16}$]
{\includegraphics[width=.20\columnwidth,height=.20\columnwidth]{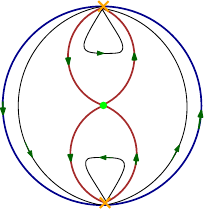}}
\quad
\subfigure[$I_1 \cup I_3 \cup I_7\cup I_9 \cup \mbox{NE}^+\cup \mbox{NE}_1^-$]
{\includegraphics[width=.20\columnwidth,height=.20\columnwidth]{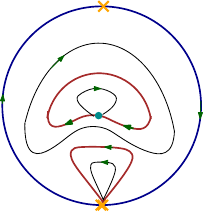}}
\quad
\subfigure[$I_2 \cup I_8 \cup I_{14} \cup I_{17} \cup I_{19}$ ]
{\includegraphics[width=.20\columnwidth,height=.20\columnwidth]{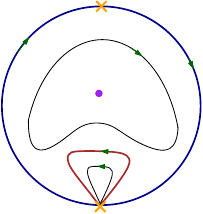}}
\quad
\subfigure[ $I_4 \cup I_{10} \cup \mbox{NE}_2^-$ ]
{\includegraphics[width=.20\columnwidth,height=.20\columnwidth]{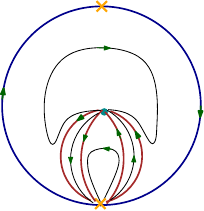}}
\\[1em]
\quad
\subfigure[$I_5 \cup I_{11} \cup I_{15} \cup I_{18} \cup I_{20}$ ]
{\includegraphics[width=.20\columnwidth,height=.20\columnwidth]{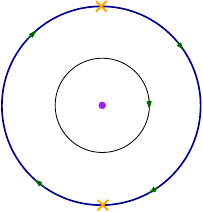}}
\quad
\subfigure[$I_6 \cup I_{12}\cup \mbox{NE}_0^-$  ]
{\includegraphics[width=.20\columnwidth,height=.20\columnwidth]{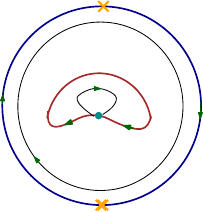}}
\quad
\subfigure[$\mbox{PA}_1 \cup \mbox{PA}_2 \cup B^-$ ]
{\includegraphics[width=.20\columnwidth,height=.20\columnwidth]{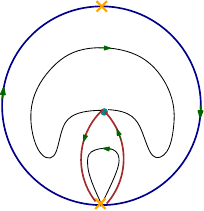}}
\caption{\label{phasea0n0-one} Phase portraits in the Poincar\'e disc for differential system \eqref{main} with one finite equilibrium.}
\end{figure}

\begin{figure}[!t]
\centering
\newcommand{\threeportrait}[2]{%
  \subfigure[#1]{\includegraphics[width=.18\columnwidth,height=.18\columnwidth]{#2}}}
\begin{tabular}{@{}c@{\hspace{.02\columnwidth}}c@{\hspace{.02\columnwidth}}c@{\hspace{.02\columnwidth}}c@{\hspace{.02\columnwidth}}c@{}}
\threeportrait{$I_{21}\cup I_{23}\cup I_{40}\cup I_{42}\cup H_1\cup H_2$}{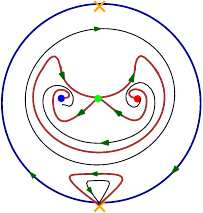} &
\threeportrait{$I_{22} \cup I_{41}$}{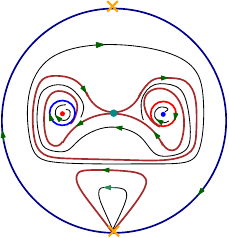} &
\threeportrait{$I_{24} \cup \mbox{PA}_3 \cup I_{43} \cup \mbox{PA}_6$}{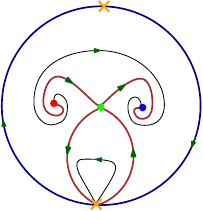} &
\threeportrait{$I_{25} \cup I_{44}$}{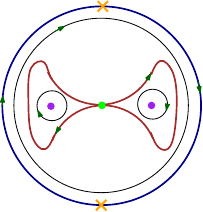} &
\threeportrait{$I_{26} \cup I_{45}$}{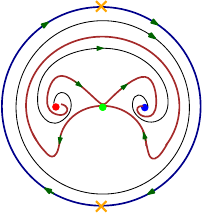}
\\
\threeportrait{$I_{27} \cup \mbox{NE}_1^+ \cup I_{30} \cup \mbox{NE}^{-}$}{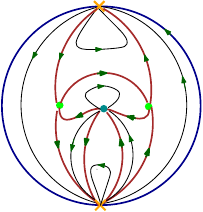} &
\threeportrait{$I_{28} \cup \mbox{NE}_2^+$}{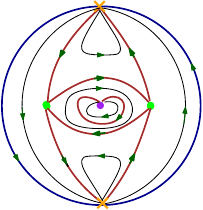} &
\threeportrait{$I_{29} \cup \Omega_{13}$}{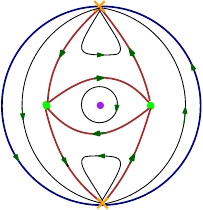} &
\threeportrait{$I_{31} \cup I_{34}$}{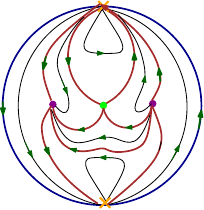} &
\threeportrait{$I_{32} \cup I_{33}$}{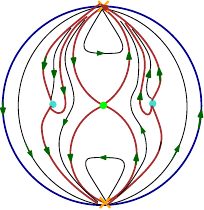}
\\
\threeportrait{$I_{35}\cup I_{36}$}{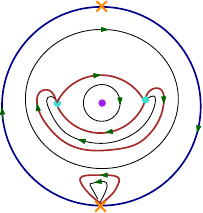} &
\threeportrait{$I_{37}$}{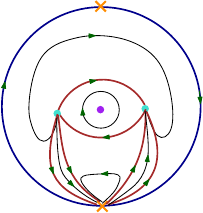} &
\threeportrait{$I_{38}$}{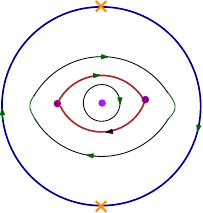} &
\threeportrait{$I_{39}$}{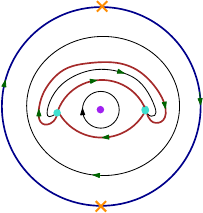} &
\threeportrait{$C_1$}{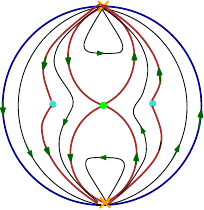}
\\
\threeportrait{$C_2$}{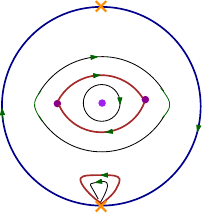} &
\threeportrait{$\mbox{PA}_4 \cup B^+$}{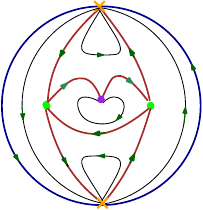} &
\threeportrait{$\mbox{PA}_5$}{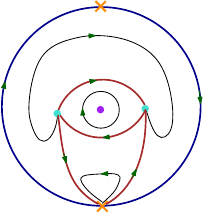} &
\threeportrait{$\mbox{HL}_1 \cup \mbox{HL}_2$}{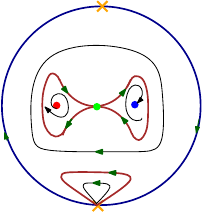} &
\threeportrait{\mbox{E}$_2$\mbox{O} $\cup$ \mbox{OE}$_2$}{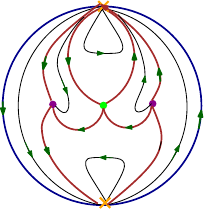}
\end{tabular}
\caption{\label{phasea0n0-three} Phase portraits in the Poincar\'e disc for differential system \eqref{main} with three finite equilibria. The blue, red, green, and purple dots denote stable, unstable, saddle, and center equilibria, respectively.}
\end{figure}

\begin{figure}[!t]
\centering
\newcommand{\fiveportrait}[2]{%
  \subfigure[#1]{\includegraphics[width=.18\columnwidth,height=.18\columnwidth]{#2}}}
\begin{tabular}{@{}c@{\hspace{.02\columnwidth}}c@{\hspace{.02\columnwidth}}c@{\hspace{.02\columnwidth}}c@{\hspace{.02\columnwidth}}c@{}}
\fiveportrait{$I_{46} \cup \mbox{H}_3^0 \cup I_{54}$}{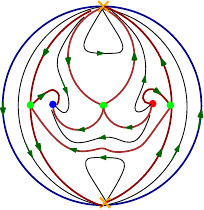} &
\fiveportrait{$I_{47} \cup \mbox{H}_3^+ \cup I_{49}$}{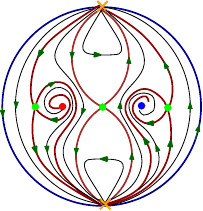} &
\fiveportrait{$I_{48}$}{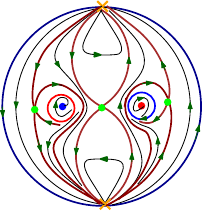} &
\fiveportrait{$I_{50} \cup I_{52}\cup \mbox{H}_3^- $}{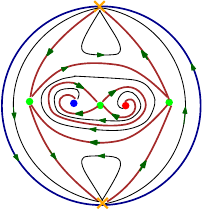} &
\fiveportrait{$I_{51}$}{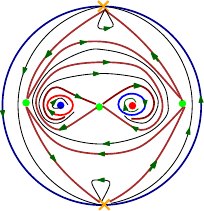}
\\
\fiveportrait{$I_{53}$}{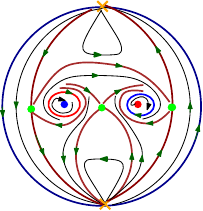} &
\fiveportrait{$I_{55}$}{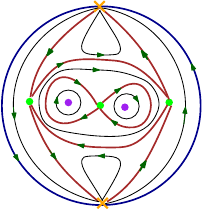} &
\fiveportrait{$I_{56}$}{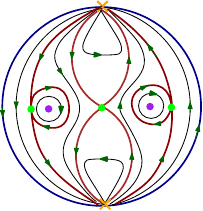} &
\fiveportrait{$I_{57}$}{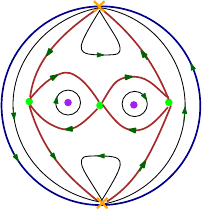} &
\fiveportrait{$\mbox{HL}_3^+$}{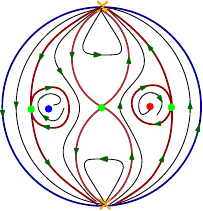}
\\
\fiveportrait{$\mbox{HL}_3^0$}{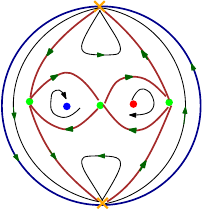} &
\fiveportrait{$\mbox{HL}_3^-$}{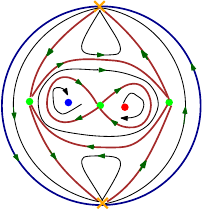} &
\fiveportrait{$\mbox{OE}_4^- \cup \mbox{E}_4\mbox{O}^-$}{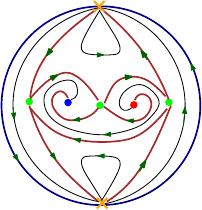} &
\fiveportrait{$\mbox{OE}_4^+ \cup \mbox{E}_4\mbox{O}^+$}{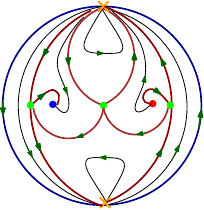} &
\fiveportrait{$\mbox{OE}_4^0 $}{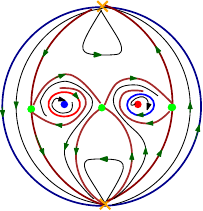}
\\
\fiveportrait{$\mbox{E}_4\mbox{O}^0$}{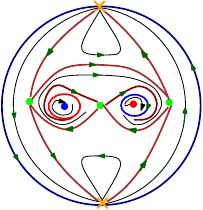} &
\fiveportrait{$\mbox{HL}_4$}{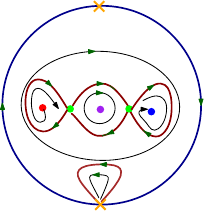} &
\fiveportrait{$I_{58}\cup H_4\cup I_{60}$}{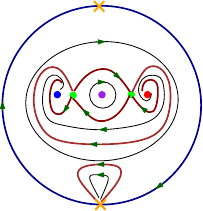} &
\fiveportrait{$I_{59}$}{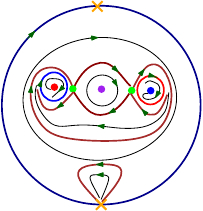} &
\fiveportrait{$I_{61} \cup \mbox{PA}_7$}{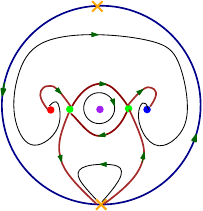}
\\
\fiveportrait{$I_{62}$}{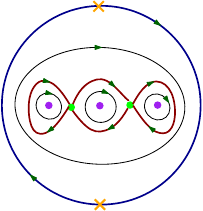} &
\fiveportrait{$I_{63}$}{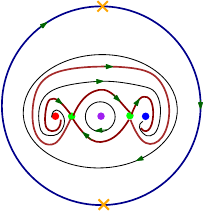} & &
\end{tabular}
\caption{\label{phasea0n0-five} Phase portraits in the Poincar\'e disc for differential system \eqref{main} with five finite equilibria.}
\end{figure}

The main purpose of this paper is twofold. First, we identify new periodic-annulus-breaking phenomena and provide a unified way to characterize them through the critical separatrix transition that prevents continuation of the annulus. In particular, the classification isolates a new large-scale periodic-annulus--homoclinic-orbit-family bifurcation and distinguishes it from the usual destruction of an isolated limit cycle at a homoclinic or heteroclinic boundary. The relevant bifurcation sets are described either algebraically or by uniquely characterized critical functions, which makes their role in the bifurcation analysis explicit. Second, by retaining the original parameters, we obtain the global dynamics for both signs of the leading coefficient of the restoring polynomial $g$, covering all transitions between one-, three-, and five-equilibrium configurations within a single parameter space, avoiding duplication of the global analysis, and yielding a complete classification of the possible global phase-portrait types, together with complete bifurcation diagrams outside $\Omega_{14}$ and all admissible critical-surface arrangements and corresponding global phase-portrait types within $\Omega_{14}$.

An outline of this paper is as follows.  Section~\ref{sec:preliminary-results} develops the general
periodic-annulus theory used in the classification.  It recalls the relevant
results for generalized polynomial Li\'enard systems, introduces the notions
of periodic-annulus-breaking and large-scale
periodic-annulus--homoclinic-orbit-family bifurcations, and states the
degree-dependent existence, nonexistence, and persistence results for a
large-scale periodic annulus, together with the critical-connection theorem
at infinity.  Section~\ref{sec:general-theorem-proofs} proves
Theorems~\ref{thm_bifurcation}--\ref{thm:critical-connections} and
Proposition~\ref{prop:n-equals-s}, stated in
Section~\ref{sec:preliminary-results} by means of
the Conti--Filippov transformation, a geometric separatrix-intersection criterion,
strict monotonicity with respect to the damping parameter, and the dynamics
at infinity, with the critical-connection theorem treated in the final
subsection.
Section~\ref{sec:local-analysis} classifies all finite equilibria and equilibria at
infinity of system~\eqref{main}, thereby establishing the local bifurcations
and degeneracies in Theorem~\ref{thm-bifurcation}.
Section~\ref{s5} determines the
existence, nonexistence, uniqueness, and stability of limit cycles in the
relevant parameter regions and constructs the critical homoclinic and
limit-cycle-breaking functions and their associated bifurcation sets.
Section~\ref{s7} treats the regions
$\Omega_1$--$\Omega_{15}$ case by case, determines all remaining separatrix
connections and periodic annuli, and completes the global
phase-portrait classification and the bifurcation classification outside
$\Omega_{14}$; within $\Omega_{14}$, it establishes the admissible
critical-surface orderings and corresponding phase portraits.  The precise definitions
of the bifurcation sets and of the remaining phase-portrait regions are
collected in Theorems~\ref{thm-bifurcation}
and~\ref{thm-phase-portraits}, respectively.

\section{Preliminary results}\label{sec:preliminary-results}

Recently, considerable attention has been paid to the generalized polynomial Li\'enard system:
\begin{equation} \label{sys_cdf}
\begin{aligned}
\dot{x} &= y, \\
\dot{y} &= -g(x) - f(x)y,
\end{aligned}
\end{equation}
where $g(x) = \sum_{i=k}^m a_i x^i$ and $f(x) = \sum_{i=s}^n b_i x^i$ are polynomials with $a_k a_m b_s b_n \neq 0$. Without loss of generality, it is assumed that the leading coefficient of the damping polynomial satisfies $b_n > 0$. Otherwise, a linear transformation together with a time rescaling can be applied.

Chen, Dai and Feng \cite{ChenDaiFeng2024} systematically investigated the global phase portraits of system \eqref{sys_cdf} under the specific configuration where the origin is the unique equilibrium and the system possesses a periodic annulus. Under suitable assumptions, they established necessary and sufficient conditions for the existence of a periodic annulus in the cases where the origin is either a linear center or a nilpotent center. Subsequently, they focused on more degenerate configurations of the origin. To analyze the existence of a {\it large-scale periodic annulus} (i.e., a periodic annulus whose closure meets the boundary circle of the Poincar\'e disc), especially when the origin is a degenerate nilpotent equilibrium whose neighborhood consists of exactly one elliptic sector and one hyperbolic sector, the following conditions are considered:
\begin{itemize}
    \item[\textbf{(i)}] $x g(x) > 0$ for all $x \neq 0$;
    \item[\textbf{(ii)}] $k$ and $s$ are odd with $k > 2s+1$ and $a_k > 0$, or $s$ is odd with $k = 2s+1 \ge 3$ and $b_s^2 - 2(k+1)a_k \ge 0$;
    \item[\textbf{(iii)}] $F(x_1) = F(x_2)$ if $G(x_1) = G(x_2)$ for $x_1 < 0 < x_2$, where $F(x) := \int_0^x f(\xi) d\xi$ and $G(x) := \int_0^x g(\xi) d\xi$;
    \item[\textbf{(iv)}] The equation $F(x) = 0$ has nonzero real roots.
\end{itemize}

Here, conditions (i) and (ii) guarantee that the origin is the unique equilibrium and exhibits the specified nilpotent topological structure with one elliptic sector and one hyperbolic sector. Condition (iii) serves as a necessary condition for the existence of a periodic annulus. 
Under these assumptions, the major results from \cite{ChenDaiFeng2024} concerning the periodic annulus can be formulated directly using these conditions as follows:

\begin{prop}[\cite{ChenDaiFeng2024}] \label{prop_cdf1}
Under conditions (i), (ii), and (iii), if $b_s < 0$, then system \eqref{sys_cdf} always possesses a large-scale periodic annulus.
\end{prop}

\begin{remark}
It is worth noting that when $b_s < 0$ and $b_n > 0$, condition (iv) is automatically guaranteed to hold, because the leading coefficient and the coefficient of the lowest-degree term of the polynomial $F(x)$ have opposite signs.
\end{remark}

\begin{prop}[\cite{ChenDaiFeng2024}] \label{prop_cdf2}
Under conditions (i), (ii), (iii), and (iv), if $b_s > 0$, then system \eqref{sys_cdf} possesses a large-scale periodic annulus.
\end{prop}

However, for the remaining complementary case where $b_s > 0$ and condition (iv) fails (i.e., the equation $F(x) = 0$ has no nonzero real roots), no corresponding qualitative results were provided in \cite{ChenDaiFeng2024}, leaving the precise global dynamics and the existence of the periodic annulus completely open. 

Motivated by this unresolved problem, the present work aims to
systematically bridge this gap by exploring the boundary of large-scale
periodic annuli in this open parameter region.  We call a transition a
\textit{periodic-annulus-breaking bifurcation} if a periodic annulus exists
on one side of a critical parameter value but cannot be continued across it
because its boundary undergoes a separatrix transition.  The annulus may be
bounded or large-scale, and the critical boundary may be homoclinic,
heteroclinic, or involve an equilibrium at infinity.

A \textit{large-scale periodic-annulus--homoclinic-orbit-family bifurcation}
is the particular case in which the homoclinic-orbit family filling an
elliptic sector persists, while the sector changes from bounded to unbounded
at the critical value.  Accordingly, the accompanying large-scale periodic
annulus exists only on the bounded side.

To systematically explore the boundary of the large-scale periodic annulus in this remaining open parameter region (where $b_s > 0$ and condition (iv) fails), we parameterize the original damping polynomial $f(x) = \sum_{i=s}^n b_i x^i$ by isolating its local trailing term and introducing a bifurcation parameter $K \ge 0$ for the higher-order dissipative terms. Specifically, we consider the generalized polynomial Li\'enard system:
\begin{equation} \label{sys_gen}
\begin{aligned}
\dot{x} &= y, \\
\dot{y} &= -g(x) - f(x; K)y,
\end{aligned}
\end{equation}
where $g(x)=\sum_{i=k}^m a_i x^i$ and $f(x;K)=b_sx^s+Kh(x)$. If $deg(f)=n>s$, we set $h(x)=\sum_{i=s+1}^n d_i x^i$ with $d_n>0$, so that $b_i=Kd_i$ for $i=s+1,\ldots,n$. If $n=s$, the sum is empty and we set $h\equiv0$; the parameter $K$ is then inoperative and $f(x;K)\equiv b_sx^s$. Here $m,n,k,s\in\mathbb N$ and the leading and trailing coefficients of every nonzero polynomial are nonzero. In the genuinely parameterized case $n>s$, the normalization of $K$ remains flexible; for example, one may take $K=b_n$ and $d_n=1$.

For the genuinely parameterized case $n>s$, we preserve conditions (i)--(iii) for every $K\geq0$ and impose the following condition on the higher-order part:
\begin{itemize}
    \item[\textbf{(v)}] $h(x) > 0$ for all $x > 0$.
\end{itemize}

Under these structural assumptions, we classify the global dynamics of this parameterized family. Since conditions (i) and (ii) govern the local topology of the origin, the possible local configurations at infinity are governed by the polynomial degrees $m$ and $n$. We summarize the resulting topological alternatives in the following four theorems and one proposition.

\begin{theorem}[Large-scale periodic-annulus-breaking bifurcation]\label{thm_bifurcation}
Under conditions (i)--(iii) and (v), assume that $b_s>0$, condition (iv) fails and $n>s$. If $2s+1<m\leq 2n+1$, then there exists a unique critical bifurcation value $K^*>0$, depending on the invariant parameters, such that system \eqref{sys_gen} undergoes a periodic-annulus-breaking bifurcation, more precisely, a large-scale periodic-annulus--homoclinic-orbit-family bifurcation:
\begin{itemize}
    \item[(a)] {For $0 \le K < K^*$, the elliptic sector at the origin is bounded and system \eqref{sys_gen} possesses a large-scale periodic annulus.}
    \item[(b)] {For $K\geq K^*$, the elliptic sector at the origin is unbounded, system \eqref{sys_gen} possesses no large-scale periodic annulus, and the boundary separatrix entering the origin from the fourth quadrant tends to infinity when it is followed backward in time.}
\end{itemize}
{In both cases the elliptic sector and its homoclinic-orbit
family persist; the transition at $K=K^*$ concerns their boundedness and
the existence of the accompanying large-scale periodic annulus.}

\end{theorem}

\begin{theorem}[Nonexistence of a large-scale periodic annulus]\label{thm_collapse}
Under conditions (i)--(iii) and (v), assume that $b_s>0$, condition (iv) fails and $n>s$. If $m=k=2s+1$, then, for every $K\geq0$, system \eqref{sys_gen} possesses no large-scale periodic annulus.
\end{theorem}

\begin{theorem}[Existence of a large-scale periodic annulus]\label{thm_survival}
Under conditions (i)--(iii) and (v), assume that $b_s>0$, condition (iv) fails and $n>s$. If $m>2n+1$, then, for every $K\geq0$, system \eqref{sys_gen} possesses a large-scale periodic annulus.
\end{theorem}

\begin{prop}\label{prop:n-equals-s}
When $n=s$, 
under conditions (i)--(iii), with $b_s>0$ and condition (iv) failing, system \eqref{sys_gen} possesses no large-scale periodic annulus when $m=k=2s+1$, whereas it possesses a large-scale periodic annulus when $m>2s+1$.
\end{prop}

\begin{theorem}[Connections from the origin to infinity]\label{thm:critical-connections}
Under the assumptions of Theorem~\ref{thm_bifurcation}, suppose further
that $n+1<m<2n+1$.  At $K=K^*$, the boundary orbit of
system~\eqref{sys_gen} separating the elliptic and hyperbolic sectors in
the fourth (resp. third) quadrant is the unique orbit branch connecting the origin with
the equilibrium at infinity in the negative $y$-direction.  For every
$K>K^*$, system~\eqref{sys_gen} has infinitely many orbit branches in
the fourth quadrant connecting the origin with this equilibrium at infinity.
\end{theorem}

\begin{remark}\label{rem:critical-connections}
The relation between $m$ and $n+1$ affects the local configuration at
infinity.  Under the assumptions of Theorem~\ref{thm_bifurcation}, both
$m$ and $n$ are odd; hence, the equality $m=n+1$ cannot occur.
When $m<n+1$, the positive $x$-direction at infinity contains a
saddle-node; see \cite{ChenDaiFeng2024} or
\cite{DumortierHerssens1999}.  The associated global connection problem
is not needed in the subsequent analysis and will not be pursued here.

When $m>n+1$, the equilibrium in the positive $x$-direction no longer
exists.  The applications of Theorem~\ref{thm:critical-connections} below
all satisfy the sharper range $n+1<m<2n+1$.  The endpoint $m=2n+1$
requires a separate three-term dominant balance at infinity and is not
covered by this theorem.  The theorem shows that the number of orbit
branches connecting the origin with the equilibrium at infinity in the
negative $y$-direction changes from one at $K=K^*$ to infinitely many for
$K>K^*$.  Consequently, these two parameter values correspond to different
global phase portraits.
\end{remark}

\begin{remark}\label{rem:multiple-finite-equilibria}
Theorems~\ref{thm_bifurcation}--\ref{thm:critical-connections} are stated for the case where the origin is the unique finite equilibrium. In fact, the same idea can also be applied to systems with multiple finite equilibria. 
However, since the finite equilibria other than the origin may exhibit a wide variety of local dynamical configurations, we do not formulate these extensions as a single theorem for system~\eqref{sys_cdf}.  Some representative cases are treated in Section~\ref{s7}.
\end{remark}

\begin{remark}
As with many global bifurcation thresholds, the critical value $K^*$ in Theorem~\ref{thm_bifurcation} cannot in general be expressed explicitly in terms of the coefficients.  In certain special cases, however, its exact expression can be obtained; see Section~\ref{s7} and, in particular,
Proposition~\ref{prop:separatrix-infinity}.
\end{remark}


\section{Proofs of Theorems~\ref{thm_bifurcation}--\ref{thm:critical-connections} and Proposition~\ref{prop:n-equals-s}}
\label{sec:general-theorem-proofs}

In this section, we present the qualitative proofs of Theorems
\ref{thm_bifurcation}--\ref{thm:critical-connections} and Proposition~\ref{prop:n-equals-s}.  Before treating the individual cases, we
establish the geometric criterion for the large-scale periodic annulus and
prove three common local and parameter-dependent properties of the boundary
separatrix.  The proof of Theorem~\ref{thm:critical-connections} is given in
the final subsection.

Specifically, condition (i) ensures that system \eqref{sys_gen} possesses the origin as its unique equilibrium point. Meanwhile, condition (ii) guarantees that the origin is a nilpotent equilibrium whose local neighborhood consists of exactly one elliptic sector and one hyperbolic sector.
For system \eqref{sys_gen}, write $F_K(x):=\int_0^x f(\xi;K)\,d\xi$. The generalized symmetry supplied by condition (iii) and the resulting geometric criterion will be used repeatedly, so we record them explicitly. 

\begin{lemma}[Geometric criterion for the periodic annulus]\label{lem:geometric-criterion}
Assume conditions (i)--(iii), $b_s>0$, and that condition (iv) fails. Let $\Gamma_s(K)$ be the boundary orbit of system \eqref{sys_gen} entering the origin through the fourth quadrant. Then system \eqref{sys_gen} possesses a large-scale periodic annulus if and only if $\Gamma_s(K)$ intersects the positive $x$-axis at a finite point. Every such intersection is transversal.
\end{lemma}

\begin{proof}
Apply the Conti--Filippov transformation
$z=\operatorname{sgn}(x)\sqrt{2G(x)}$,
and set $Y=y+F_K(x)$. Away from the origin, make the positive time change $d\tau/dt=g(x)/z$. Since condition (i) gives $g(x)/z>0$, this change preserves orbit orientation. If $x=x(z)$ denotes the inverse of the Conti--Filippov transformation, the  system \eqref{sys_gen} becomes
\begin{equation}\label{eq:filippov-transformed-system}
\begin{aligned}
\frac{dz}{d\tau}&=Y-\widetilde F_K(z),
\qquad &
\frac{dY}{d\tau}&=-z,
\end{aligned}
\end{equation}
where $\widetilde F_K(z):=F_K(x(z))$. Points $z$ and $-z$ correspond to points $x_2>0$ and $x_1<0$ satisfying $G(x_1)=G(x_2)$. Hence condition (iii) says precisely that $\widetilde F_K$ is even. System \eqref{eq:filippov-transformed-system} is therefore invariant under $(z,Y,\tau)\mapsto(-z,Y,-\tau)$.

Suppose that the orbit $\Gamma_s(K)$ of system \eqref{sys_gen} meets the positive $x$-axis at $(X,0)$, $X>0$. Since $\dot y(X,0)=-g(X)<0$, the intersection is transversal. When the orbit is followed backward from this point, it enters the first quadrant; there $x$ decreases and the orbit cannot cross back through $y=0$, because the backward vector field satisfies $dy/d(-t)=g(x)>0$ on that axis. On the strip $0\leq x\leq X$, the vector field has at most linear growth in $y$, so finite-time escape is impossible; strict monotonicity of $x$ also excludes a recurrent limit set. Since the origin is the only finite equilibrium, this branch therefore either returns to the origin or reaches the fixed line $x=0$ at a point with $y>0$. In the latter case reversibility supplies the reflected branch returning to the origin. Thus the boundary orbit closes the elliptic sector. Conversely, if the elliptic sector is bounded, its fourth-quadrant boundary must leave that quadrant through $y=0$ at a finite positive value of $x$, and hence $\Gamma_s(K)$ intersects the positive $x$-axis. 

Finally, for the compactified phase portraits allowed by conditions (i)--(iii), boundedness of this elliptic sector is equivalent to the existence of the adjacent large-scale periodic annulus; see \cite{ChenDaiFeng2024}. Indeed, the closed sector boundary and the boundary circle of the Poincar\'e disc enclose an invariant annular component with no equilibrium.  Its two boundary components are preserved by the reverser in \eqref{eq:filippov-transformed-system}; the compactified phase portrait forces every intervening orbit to meet the fixed line $z=0$ twice.  The orbit segment between two consecutive intersections is carried by the reverser onto its complementary segment and hence closes.  Continuous dependence and the Poincar\'e--Bendixson theorem then fill the component with periodic orbits.  If the elliptic sector is unbounded, its boundary reaches the boundary circle and cuts this annular component, so no adjacent large-scale periodic annulus remains. This proves the equivalence.
\end{proof}

In the fourth quadrant ($x > 0, y < 0$), the trajectories are governed by the differential equation:
\begin{equation} \label{orbit_eq}
\frac{dy}{dx} = \frac{-g(x) - f(x; K)y}{y} = -\frac{g(x)}{y} - b_s x^s - K h(x).
\end{equation}
For convenience, denote the slope field in \eqref{orbit_eq} by
\[
P(x,y,K):=-\frac{g(x)}{y}-b_sx^s-Kh(x).
\]

\begin{lemma}[Local structure of the boundary separatrix]\label{lem:local}
Let
\begin{equation}\label{eq:gamma-definition}
\gamma:=
\begin{cases}
\dfrac{b_s}{s+1},& k>2s+1,\\[3mm]
\dfrac{b_s+\sqrt{b_s^2-4(s+1)a_k}}{2(s+1)},& k=2s+1.
\end{cases}
\end{equation}
For every $K\geq0$, the boundary separatrix in the fourth quadrant admits a unique local graph $y=y_s(x;K)$ satisfying
\begin{equation}
 y_s(x;K)=-\gamma x^{s+1}+o(x^{s+1})
\qquad\text{as }x\to0^+.   
\label{y_s(x;K)}
\end{equation}

Moreover, for each $K_0\geq0$, its intersection with a sufficiently small transverse section $x=\varepsilon>0$ depends smoothly on $K$ near $K_0$.
\end{lemma}

\begin{proof}
To resolve the singularity of equation \eqref{sys_gen} at the nilpotent equilibrium, use the directional blow-up $y=x^{s+1}v$. After the positive time rescaling $d\tau=x^sdt$, system \eqref{sys_gen} becomes, in the chart $x\geq0$,
\begin{equation}\label{eq:blown-origin}
\begin{aligned}
\frac{dx}{d\tau}&=xv,\\
\frac{dv}{d\tau}&=-\frac{g(x)}{x^{2s+1}}-b_sv-(s+1)v^2
-K\frac{h(x)}{x^s}v.
\end{aligned}
\end{equation}
Both quotients in \eqref{eq:blown-origin} extend smoothly to $x=0$, because $k\geq2s+1$ and every monomial of $h$ has degree at least $s+1$. The fourth quadrant of the original $(x,y)$-plane corresponds to $\{x>0,v<0\}$ in the blow-up chart. Moreover, the restriction of the blown-up system to the exceptional divisor is independent of $K$, since $h(x)/x^s\to0$ as $x\to0^+$. Since
\[
\lim_{x\to0^+}\frac{g(x)}{x^{2s+1}}
=
\begin{cases}
0,& k>2s+1,\\
a_k,& k=2s+1
\end{cases}
,\]
the exceptional divisor contains the equilibria $(0,0)$ and $(0,-b_s/(s+1))$ when $k>2s+1$. When $k=2s+1$, its equilibria are $(0,v_+)$ and $(0,v_-)$, where
\[
v_\pm=\frac{-b_s\pm\sqrt{b_s^2-4(s+1)a_k}}{2(s+1)}.
\]
These two equilibria are distinct when $b_s^2-4(s+1)a_k>0$ and coalesce when $b_s^2-4(s+1)a_k=0$. In either case, the equilibrium corresponding to the boundary separatrix in the fourth quadrant is $(0,-\gamma)$, with $\gamma$ given by \eqref{eq:gamma-definition}.

At $(0,-\gamma)$, the eigenvalue in the $x$-direction is $-\gamma<0$. The other eigenvalue is $b_s>0$ when $k>2s+1$, and $\sqrt{b_s^2-4(s+1)a_k}\geq0$ when $k=2s+1$. Hence the stable-manifold theorem in the hyperbolic cases, and the strong stable-manifold theorem in the zero-eigenvalue case, give a unique invariant graph $v=v_s(x;K)$ with $v_s(0;K)=-\gamma$. These invariant manifolds, and therefore their intersections with every sufficiently small section $x=\varepsilon>0$, depend smoothly on $K$. Blowing down gives $y_s(x;K)=x^{s+1}v_s(x;K)=-\gamma x^{s+1}+o(x^{s+1})$.
\end{proof}

Condition (v) gives the family \eqref{sys_gen}, parametrized by $K$, a strict rotated-vector-field structure for $x>0$ away from $y=0$. Combined with Lemma~\ref{lem:local}, this yields the following strict monotonicity of the boundary separatrix  $y_s(x;K)$ (see \eqref{y_s(x;K)}).

\begin{lemma}[Strict monotonicity with respect to $K$]\label{lem:monotonicity}
Let $0\leq K_1<K_2$. At every common value $x>0$ for which the two boundary separatrices remain in the fourth quadrant, $y_s(x;K_2)<y_s(x;K_1)<0$.
\end{lemma}

\begin{proof}
For $x>0$ and $v<0$, the determinant of the vector field of \eqref{eq:blown-origin} and its derivative in $K$ is
\[
\begin{vmatrix}
\frac{dx}{d\tau} & \frac{dv}{d\tau}\\[3mm]
\frac{\partial}{\partial K}\left(\frac{dx}{d\tau}\right) &
\frac{\partial}{\partial K}\left(\frac{dv}{d\tau}\right)
\end{vmatrix}
=-x\frac{h(x)}{x^s}v^2<0.
\]
Thus, the blown-up family is a strict generalized rotated vector field with respect to $K$ away from the exceptional divisor (see \cite[Chapter 4.3]{ZhangEtAl1992} or \cite{Perko1975}).

It remains to determine the ordering of the stable manifolds as they leave $(0,-\gamma)$. Their graph equation is
\begin{equation}\label{eq:blown-graph}
x\frac{dv_s}{dx}
=-\frac{g(x)}{x^{2s+1}v_s}-b_s-(s+1)v_s
-K\frac{h(x)}{x^s}.
\end{equation}
By the smooth parameter dependence in Lemma \ref{lem:local}, we may differentiate \eqref{eq:blown-graph} with respect to $K$. Writing $w=\frac{\partial v_s}{\partial K}$ gives
\begin{equation}\label{eq:variation-K}
x\frac{dw}{dx}
=\left(\frac{g(x)}{x^{2s+1}v_s^2}-(s+1)\right)w
-\frac{h(x)}{x^s},
\qquad w(0)=0.
\end{equation}
Let $A_0$ denote the limit at $x=0$ of the coefficient multiplying $w$ in \eqref{eq:variation-K}. Then
\[
A_0=
\begin{cases}
-(s+1),& k>2s+1,\\[1mm]
-\dfrac{\sqrt{b_s^2-4(s+1)a_k}}{\gamma},& k=2s+1.
\end{cases}
\]
Let $d_\ell x^\ell$ be the lowest-degree nonzero term of $h(x)$. Condition (v) implies $d_\ell>0$, and $\ell\geq s+1$. A leading-order balance in \eqref{eq:variation-K} gives
\[
w(x)=-\frac{d_\ell}{\ell-s-A_0}x^{\ell-s}
+o\!\left(x^{\ell-s}\right)
\qquad\text{as }x\to0^+.
\]
Since $\ell-s-A_0>0$, it follows that $w(x)<0$ for all sufficiently small $x>0$. Consequently, for any $K_1<K_2$, smooth dependence on $K$ and compactness of $[K_1,K_2]$ give
$v_s(x;K_2)<v_s(x;K_1)$
for all sufficiently small $x>0$.

This ordering persists throughout the common domain of the two separatrices in the fourth quadrant. Indeed, suppose that they have a first intersection at $x=x_*>0$. At this point, $v_s(x_*;K_2)=v_s(x_*;K_1)$, and subtracting their graph equations \eqref{eq:blown-graph} yields
\[
x_*\left(v_s^{\prime}(x_*;K_2)-v_s^{\prime}(x_*;K_1)\right)
=-(K_2-K_1)\frac{h(x_*)}{x_*^s}<0.
\]
On the other hand, at a first intersection the difference $v_s(x;K_2)-v_s(x;K_1)$ reaches zero from below, so its derivative at $x_*$ must be non-negative, which is a contradiction. Therefore $v_s(x;K_2)<v_s(x;K_1)<0$ throughout their common domain. Multiplying by $x^{s+1}>0$ gives $y_s(x;K_2)<y_s(x;K_1)<0$.
\end{proof}

\begin{lemma}[Openness of the hitting set]\label{lem:hitting-open}
Let $\mathcal A$ be the set of all $K\geq0$ for which the boundary separatrix $y_s(x;K)$ intersects $y=0$ at some finite $x>0$. Then $\mathcal A$ is relatively open in $[0,+\infty)$.
\end{lemma}

\begin{proof}
Suppose that $K_0\in\mathcal A$ and that the first intersection is $(X_0,0)$ with $X_0>0$. In the original system, $\dot y(X_0,0)=-g(X_0)<0$, so the intersection with the positive $x$-axis is transversal. Choose a sufficiently small transverse section $x=\varepsilon>0$ near the blown-up equilibrium. By Lemma \ref{lem:local}, the stable, or strong stable, manifold of $(0,-\gamma)$ intersects this section at a point depending continuously on $K$. Starting from this regular section, the ordinary continuous-dependence theorem applies to the original system. Transversality and the implicit-function theorem then imply that, for all $K$ sufficiently close to $K_0$, the corresponding separatrix also intersects $y=0$ near $(X_0,0)$. Hence $\mathcal A$ is relatively open.
\end{proof}

From this point onward we assume $n>s$; the exceptional case $n=s$ has already been stated separately above.

\subsection{Proof of Theorem \ref{thm_bifurcation}}

We organize the proof in three steps. First we show that the boundary separatrix intersects the positive $x$-axis when $K=0$, see Lemma \ref{lem:lower}. We then construct, from a regular transverse section, a global upper barrier which prevents such an intersection for all sufficiently large $K$, see Lemma \ref{lem:upper}. Finally, strict parameter monotonicity and openness of the hitting set yield the unique threshold.

\begin{lemma}[Intersection for $K=0$]\label{lem:lower}
For $K=0$, the boundary separatrix $y_s(x;0)$ intersects the positive $x$-axis transversally at some finite point $(X_0,0)$. Consequently, there exists $\delta>0$ such that the same conclusion holds for every $K\in[0,\delta)$.
\end{lemma}
\begin{proof}
Suppose, for contradiction, that $y_s(x;0)$ never intersects the positive $x$-axis. For $K=0$, the higher-order damping term $Kh(x)$ vanishes and $f(x;0)=b_sx^s$. Hence, as long as the separatrix remains in the fourth quadrant, equation \eqref{orbit_eq} reduces to
\[
y_s^{\prime}(x;0)=\frac{g(x)}{|y_s(x;0)|}-b_sx^s>-b_sx^s. 
\]
Since $y_s(x;0)\to0$ as $x\to0^+$, integration yields
$y_s(x;0)>-\frac{b_s}{s+1}x^{s+1}$.
This estimate prevents escape to $y=-\infty$ at a finite value of $x$. Under the contradiction hypothesis, ordinary continuation in the original system therefore shows that the graph is defined for every $x>0$.

The horizontal isocline for $K=0$ is $y_{\rm iso}(x)=-g(x)/(b_sx^s)$. As $x\to+\infty$, we have $y_{\rm iso}(x)\sim-(a_m/b_s)x^{m-s}$ and $y_{\rm iso}^{\prime}(x)\sim-((m-s)a_m/b_s)x^{m-s-1}$; in particular, $y_{\rm iso}^{\prime}(x)<0$ for all sufficiently large $x$. Moreover, $m>2s+1$ gives $m-s>s+1$. Comparing these exponents with the preceding lower estimate shows that $y_s(x;0)>y_{\rm iso}(x)$ for all sufficiently large $x$. Choose $R>0$ such that both $y_s(x;0)>y_{\rm iso}(x)$ and $y_{\rm iso}^{\prime}(x)<0$ hold for $x\geq R$.

The separatrix cannot subsequently cross the isocline from above to below. Indeed, at any contact point one has $y_s^{\prime}=0$, and hence $(y_s-y_{\rm iso})^{\prime}=-y_{\rm iso}^{\prime}>0$. Thus $y_s(x;0)>y_{\rm iso}(x)$ and $y_s^{\prime}(x;0)>0$ for every $x\geq R$. If the separatrix never reaches $y=0$, then it is negative and increasing on $[R,+\infty)$. Setting $M:=-y_s(R;0)>0$, we obtain
\[
y_s^{\prime}(x;0)
=\frac{g(x)}{|y_s(x;0)|}-b_sx^s
\geq\frac{g(x)}{M}-b_sx^s\longrightarrow+\infty.
\]
In particular, $y_s^{\prime}(x;0)>1$ for all sufficiently large $x$, which forces $y_s$ to reach zero at a finite value of $x$, a contradiction. Hence the intersection $(X_0,0)$ exists. It is transversal because $\dot y(X_0,0)=-g(X_0)<0$. Lemma \ref{lem:hitting-open} now gives $\delta>0$ such that $[0,\delta)\subset\mathcal A$.
\end{proof}

\begin{lemma}[Non-intersection for large $K$]\label{lem:upper}
There exists $K_{\max}>0$ such that, for every $K\geq K_{\max}$, the boundary separatrix does not intersect the positive $x$-axis.
\end{lemma}

\begin{proof}
Let $X_0$ be given by Lemma \ref{lem:lower}, and fix $0<\varepsilon<X_0$. Set $r:=\max\{s+1,m-n\}$ and $\chi(x):=x^{s+1}+x^r$. Since $n>s$ and $m\leq2n+1$, we have
$s+1\leq r\leq n+1$.
Choose $c>0$ sufficiently small that $y_s(\varepsilon;0)<-c\chi(\varepsilon)<0$, and put $y_U(x):=-c\chi(x)$ for $x\geq\varepsilon$. Since $\chi(x)>0$ on $[\varepsilon,+\infty)$, we have $y_U(x)<0$ throughout this interval, while at the initial section the choice of $c$ gives $y_s(\varepsilon;0)<y_U(\varepsilon)<0$.

Define
\[
\xi_c(x):=
\frac{\dfrac{g(x)}{c\chi(x)}+c\chi^{\prime}(x)-b_sx^s}{h(x)},
\qquad x\geq\varepsilon.
\]
The function $\xi_c$ is bounded above on $[\varepsilon,+\infty)$. Indeed, it is continuous there by condition (v), while at infinity its three contributions have orders $x^{m-r-n}$, $x^{r-1-n}$ and $x^{s-n}$, respectively. These exponents are non-positive because $r\geq m-n$, $r\leq n+1$ and $n>s$. We may therefore choose
\begin{equation}\label{eq:Kmax-bifurcation}
K_{\max}>
\max\left\{0,\,
\sup_{x\geq\varepsilon}\xi_c(x)
\right\}.
\end{equation}

Fix $K\geq K_{\max}$. Since $\varepsilon<X_0$, the graph $y_s(x;0)$ is defined and negative for $0<x\leq\varepsilon$. At every such $x$ for which the $K$-separatrix remains in the fourth quadrant, Lemma \ref{lem:monotonicity} gives $y_s(x;K)<y_s(x;0)<0$. Hence the $K$-separatrix cannot intersect the positive $x$-axis before reaching $x=\varepsilon$. If its graph terminates before that section, continuation in the original system and the preceding non-intersection imply that it escapes to infinity, so the desired conclusion already holds. Otherwise, it reaches $x=\varepsilon$, and
\[
y_s(\varepsilon;K)<y_s(\varepsilon;0)<y_U(\varepsilon).
\]
On the comparison curve $y=y_U(x)$, equations \eqref{orbit_eq} and \eqref{eq:Kmax-bifurcation} give
\[
P(x,y_U(x),K)-y_U^{\prime}(x)
=\frac{g(x)}{c\chi(x)}-b_sx^s-Kh(x)+c\chi^{\prime}(x)<0
\]
for every $x\geq\varepsilon$. Thus $y_U$ is a strict upper barrier, and $y_s(x;K)<y_U(x)<0$ throughout the remaining common domain. The separatrix therefore cannot intersect the positive $x$-axis. This proves the lemma.
\end{proof}

\begin{proof}[Proof of Theorem \ref{thm_bifurcation}]
Let $\mathcal A$ be the hitting set defined in Lemma \ref{lem:hitting-open}.
Lemma \ref{lem:lower} gives $[0,\delta)\subset\mathcal A$, whereas Lemma \ref{lem:upper} shows that $\mathcal A$ is bounded above. Moreover, Lemma \ref{lem:monotonicity} implies that $\mathcal A$ is downward closed: if $K_0\in\mathcal A$ and $0\leq K<K_0$, then the $K$-separatrix lies above the $K_0$-separatrix in their common fourth-quadrant domain and must intersect the positive $x$-axis no later than the latter one.

Define $K^*:=\sup\mathcal A$. Then $0<K^*<+\infty$. Since $\mathcal A$ is relatively open by Lemma \ref{lem:hitting-open}, one has $K^*\notin\mathcal A$. Downward closedness now yields
$\mathcal A=[0,K^*)$.
Consequently, Lemma \ref{lem:geometric-criterion} gives a large-scale periodic annulus for $0\leq K<K^*$ and no such annulus for $K\geq K^*$.

It remains only to record the behavior of the non-intersecting separatrix. Let $(0,X_K)$ be its maximal interval as a graph in the fourth quadrant. If $X_K<+\infty$, continuation in the original system and the absence of an intersection with $y=0$ imply $y_s(x;K)\to-\infty$ as $x\to X_K^-$. If $X_K=+\infty$, then the graph itself has $x\to+\infty$. Hence, in either case, the boundary separatrix leaves every compact subset of the phase plane when followed backward in time, which is precisely the asserted tendency to infinity.
\end{proof}

\begin{remark}\label{rem:Kstar-Kinfty}
Consider the endpoint case $m=2n+1$. For $K>0$, the leading
coefficient of the damping polynomial is $Kd_n$, and the normalized
parameter used in~\cite{DumortierHerssens1999} is
$\epsilon(K)=\frac{a_m}{K^2d_n^2}$.
Define the critical value at which the phase portrait at infinity on
the Poincar\'e disc changes by
$K_{\infty}:=\frac{2\sqrt{(n+1)a_m}}{d_n}$.
Then the inequality $\epsilon(K)>1/(4(n+1))$ is equivalent to
$0<K<K_{\infty}$.

We claim that
$K^*\geq K_{\infty}$.
Indeed, fix $0<K<K_{\infty}$. Under the present assumptions,
$m$ and $n$ are odd. The phase portrait on the Poincar\'e disc in
\cite[Section~3, Figure~8]{DumortierHerssens1999}, together with the local analysis
at infinity in Subsection~\ref{subsec:equilibria-at-infinity}, show that, when
$\epsilon(K)>1/(4(n+1))$, the dynamics near the boundary circle of the Poincar\'e disc gives
rise to a center-focus problem, namely the problem of deciding whether a one-sided punctured neighborhood of {this boundary circle} is filled with periodic orbits. For the present system, condition~(iii) makes the
system~\eqref{eq:filippov-transformed-system} obtained via the
Conti--Filippov transformation
reversible under $(z,Y,\tau)\mapsto(-z,Y,-\tau)$. Combining this
symmetry with the local dynamics at infinity shows that a one-sided
neighborhood of the {boundary circle of the Poincar\'e disc} is filled
with periodic orbits.

If the fourth-quadrant boundary separatrix did not intersect the
positive $x$-axis, the continuation argument in the proof of
Theorem~\ref{thm_bifurcation} would force it to leave every compact
subset of the phase plane and hence to enter this periodic
neighborhood of the boundary circle of the Poincar\'e disc, which is impossible. Therefore the
separatrix intersects the positive $x$-axis for every
$0<K<K_{\infty}$. Since the hitting set is
$\mathcal A=[0,K^*)$, it follows that
$K^*\geq K_{\infty}$.
\end{remark}

\subsection{Proof of Theorem \ref{thm_collapse}}
\begin{proof} 

Here $m=k=2s+1$, and hence $g(x)=a_kx^{2s+1}$. For $K=0$, the curve $y_0(x)=-\gamma x^{s+1}$ is invariant. Indeed, substitution into the orbital equation \eqref{orbit_eq} gives precisely $(s+1)\gamma^2-b_s\gamma+a_k=0$. In the blown-up system \eqref{eq:blown-origin}, this curve is the constant graph $v=-\gamma$. By Lemma \ref{lem:local}, it is the stable manifold of $(0,-\gamma)$ when $b_s^2-4(s+1)a_k>0$, and the strong stable manifold when $b_s^2-4(s+1)a_k=0$. Consequently, it is exactly the boundary separatrix: $y_s(x;0)\equiv-\gamma x^{s+1}<0$ for all $x>0$.
Thus the boundary separatrix does not intersect the positive $x$-axis when $K=0$.

Now let $K>0$. Lemma \ref{lem:monotonicity} gives $y_s(x;K)<y_s(x;0)=-\gamma x^{s+1}<0$ throughout the common domain in the fourth quadrant. Equivalently, the invariant curve for $K=0$ is a strict upper barrier for the perturbed separatrix, since $P(x,y_0(x),K)-y_0^{\prime}(x)=-Kh(x)<0$ on this curve.
Suppose, for contradiction, that $y_s(X;K)=0$ for some first $X>0$. On $(0,X)$, both $y_s(x;K)$ and $y_s(x;0)$ remain in the fourth quadrant, so Lemma \ref{lem:monotonicity} applies. Letting $x\to X^-$ in the preceding inequality yields
\[
0=y_s(X;K)\leq y_s(X;0)=-\gamma X^{s+1}<0,
\]
a contradiction. Thus the boundary separatrix never intersects the positive $x$-axis for any $K\geq0$. Lemma \ref{lem:geometric-criterion} therefore implies that system \eqref{sys_gen} possesses no large-scale periodic annulus.
\end{proof}

\subsection{Proof of Proposition \ref{prop:n-equals-s}}
\begin{proof} 
When $n=s$, the convention above gives $h\equiv0$, so the system is independent of $K$. If $m=k=2s+1$, the invariant curve $y=-\gamma x^{s+1}$ is the boundary separatrix, exactly as in the $K=0$ part of the proof of Theorem \ref{thm_collapse}; it never intersects the positive $x$-axis. If $m>2s+1$, the argument of Lemma \ref{lem:lower} applies verbatim, because it uses only $f(x)=b_sx^s$ and the inequality $m>2s+1$: the boundary separatrix must intersect the positive $x$-axis transversally. The two conclusions now follow from Lemma \ref{lem:geometric-criterion}.
\end{proof}

\subsection{Proof of Theorem \ref{thm_survival}}
\begin{proof} 

Fix an arbitrary $K\geq0$ and set $\eta(x):=x^{s+1}+x^{n+1}$. Since $n>s$, the continuous quotient $(b_sx^s+Kh(x))/\eta^{\prime}(x)$ has finite limits $b_s/(s+1)$ at the origin and $Kd_n/(n+1)$ at infinity. It is therefore bounded on $(0,+\infty)$. Choose a constant $C>0$ such that
\begin{equation}\label{eq:C-survival}
C>\gamma,
\qquad
C>\sup_{x>0}\frac{b_sx^s+Kh(x)}{\eta^{\prime}(x)},
\end{equation}
and define the comparison curve $y_L(x):=-C\eta(x)$. By Lemma \ref{lem:local} and the first inequality in \eqref{eq:C-survival}, $y_s(x;K)>y_L(x)$ for all sufficiently small $x>0$.

On the curve $y=y_L(x)$, the difference between the slope field and the tangent slope is
\begin{align*}
P(x,y_L(x),K)-y_L^{\prime}(x)
&=\frac{g(x)}{C\eta(x)}-(b_sx^s+Kh(x))+C\eta^{\prime}(x)>0,
\end{align*}
where condition (i) gives $g(x)>0$ for $x>0$, and the second inequality in \eqref{eq:C-survival} gives $C\eta^{\prime}(x)>b_sx^s+Kh(x)$. Hence $y_L$ is a strict global lower barrier. Consequently,
\begin{equation}\label{eq:survival-lower-bound}
y_s(x;K)>y_L(x)
\end{equation}
throughout the part of the separatrix lying in the fourth quadrant. Notice that this single choice of $C$ replaces the separate estimates near the origin, on a compact interval, and at infinity.

If the separatrix does not intersect the positive $x$-axis, estimate \eqref{eq:survival-lower-bound} prevents it from escaping to infinity at a finite value of $x$. Hence its graph is defined for every $x>0$ under this assumption.

The horizontal isocline in the fourth quadrant is $y_{\rm iso}(x)=-g(x)/(b_sx^s+Kh(x))$. If $K>0$, then $y_{\rm iso}(x)\sim-(a_m/(Kd_n))x^{m-n}$ as $x\to+\infty$, whereas $y_L(x)\sim-Cx^{n+1}$. Since $m>2n+1$, we have $m-n>n+1$, and therefore $y_L(x)>y_{\rm iso}(x)$ for all sufficiently large $x$. If $K=0$, then $y_{\rm iso}(x)\sim-(a_m/b_s)x^{m-s}$, and $m-s>n+1$ gives the same conclusion.

Moreover, $y_{\rm iso}^{\prime}(x)<0$ for all sufficiently large $x$. Choose $R>0$ such that $y_L(x)>y_{\rm iso}(x)$ and $y_{\rm iso}^{\prime}(x)<0$ hold for $x\geq R$. If the separatrix has not already intersected the positive $x$-axis, \eqref{eq:survival-lower-bound} gives $y_s(R;K)>y_{\rm iso}(R)$. It cannot subsequently cross the isocline from above to below: at any contact point one has $y_s^{\prime}=0$, and hence $(y_s-y_{\rm iso})^{\prime}=-y_{\rm iso}^{\prime}>0$.
Thus $y_s(x;K)>y_{\rm iso}(x)$ and $y_s^{\prime}(x;K)>0$ for every $x\geq R$ as long as the separatrix remains in the fourth quadrant.

Suppose, for contradiction, that it never reaches $y=0$. Then it is negative and increasing on $[R,+\infty)$. With $M:=-y_s(R;K)>0$, we have $|y_s(x;K)|\leq M$, and therefore
\[
y_s^{\prime}(x;K)
=\frac{g(x)}{|y_s(x;K)|}-\bigl(b_sx^s+Kh(x)\bigr)
\geq\frac{g(x)}{M}-\bigl(b_sx^s+Kh(x)\bigr).
\]
Because $m>n$, the right-hand side tends to $+\infty$. In particular, it is larger than $1$ for all sufficiently large $x$, which forces $y_s$ to reach $0$ at a finite value of $x$, a contradiction. Hence the boundary separatrix intersects the positive $x$-axis transversally. Lemma \ref{lem:geometric-criterion} now gives a large-scale periodic annulus.

The proof uses the two possible relations between $k$ and $2s+1$ only through the common coefficient $\gamma$ in Lemma \ref{lem:local}; no separate global argument is required.

\end{proof}

\begin{remark}
Theorem~\ref{thm_survival} also admits an interpretation in terms of
the dynamics at infinity. For $K>0$, the assumptions imply that the
effective degrees $m$ and $n$ are odd. Since $m>2n+1$ and the
normalized leading restoring coefficient is positive, the phase
portrait on the Poincar\'e disc in
\cite[Section~3, Figure~9]{DumortierHerssens1999}, together with the local analysis
at infinity in Subsection~\ref{subsec:equilibria-at-infinity}, gives rise to a center-focus problem near the
boundary circle of the Poincar\'e disc. For the present system, condition~(iii) makes
system~\eqref{eq:filippov-transformed-system}, obtained via the
Conti--Filippov transformation, reversible under
$(z,Y,\tau)\mapsto(-z,Y,-\tau)$. Combining this symmetry with the
local dynamics at infinity shows that a one-sided neighborhood of the
boundary circle of the Poincar\'e disc is filled with periodic orbits.
If the fourth-quadrant boundary separatrix did not intersect the positive
$x$-axis, the continuation argument used above would force it to tend
to infinity, which is impossible in such a periodic neighborhood.
The geometric criterion therefore gives the large-scale periodic
annulus. When $K=0$, the same argument applies with the effective
damping degree $s$, since $m>2n+1>2s+1$.
\end{remark}

\subsection{Proof of Theorem~\ref{thm:critical-connections}}

\begin{proof}
Let $Y(x):=y_s(x;K^*)$ and $p:=m-n>1$.  The additional hypothesis
$m<2n+1$ implies $m\leq2n$ and is used in both the
compactification and
uniqueness arguments below.  The proof of Theorem~\ref{thm_bifurcation}
shows that $Y$ does not intersect the positive $x$-axis at any finite
point.  Since $m>n+1$, the positive $x$-direction contains no equilibrium
at infinity.  The local phase portrait at infinity therefore forces $Y$
to approach the equilibrium at infinity in the negative $y$-direction.

We first determine the asymptotic behavior at this equilibrium without
assuming a power law.  Put $\delta:=2n+1-m>0$, $u=1/x$,
$z=-yu^{n+1}$, $\bar g(u)=u^mg(1/u)$, and
$\bar h(u)=u^nh(1/u)$.  After the positive desingularization
$d/d\tau=u^n d/dt$, the chart at infinity is
\[
 u'=uz,\qquad
 z'=u^\delta\bar g(u)
 -(b_su^{n-s}+K\bar h(u))z+(n+1)z^2.
\]
On $u=0$ it has the hyperbolic equilibrium
$z=Kd_n/(n+1)$, whose two eigenvalues are positive; its open sector
therefore consists precisely of orbits satisfying
\[
y(x)\sim-\frac{Kd_n}{n+1}x^{n+1}.
\]
The remaining point $z=0$ is resolved by $z=u^\delta v$, which gives
\[
 u'=u^{\delta+1}v,\qquad
 v'=\bar g(u)-(b_su^{n-s}+K\bar h(u))v
 +(m-n)u^\delta v^2.
\]
Its exceptional divisor has the single equilibrium
$v=a_m/(Kd_n)$, with transverse eigenvalue $-Kd_n\ne0$; the one-sided
center separatrix satisfies
$y(x)\sim-a_mx^{m-n}/(Kd_n)$.  Thus these are the two local
asymptotic classes.  If the critical orbit $Y$ belonged to the open,
faster class, sector openness and continuous dependence on $K$ would
put $y_s(x;K)$ in the same sector for $K<K^*$ sufficiently close to
$K^*$, contrary to Theorem~\ref{thm_bifurcation}(a).  Consequently,
\begin{equation}\label{eq:critical-separatrix-asymptotic}
Y(x)\sim-\frac{a_m}{K^*d_n}x^{m-n}
\qquad\text{as }x\to+\infty.
\end{equation}

We next prove uniqueness at $K=K^*$.  Let $y(x)$ be any other orbit
branch entering the origin through the fourth quadrant.  Since $Y$ is
the lower boundary of the elliptic sector, $y(x)>Y(x)$ near the origin.
If this orbit never intersects the positive $x$-axis, uniqueness of
solutions gives $Y(x)<y(x)<0$ for all $x>0$.  Setting $w:=y-Y>0$ and
subtracting the two orbital equations yields
\[
\frac{w^{\prime}}{w}=\frac{g(x)}{y(x)Y(x)}>
\frac{g(x)}{Y(x)^2}.
\]
By \eqref{eq:critical-separatrix-asymptotic}, the last expression is
bounded from below by $cx^{2n-m}$ for all sufficiently large $x$ and
some $c>0$.  Since $m<2n+1$, and hence $2n-m\geq0$, integration forces
$w$ to grow at least as $\exp(c_1x^{2n-m+1})$ for some $c_1>0$.  This
contradicts $0<w(x)<-Y(x)=\mathcal{O}(x^{m-n})$.  Thus every such orbit
other than $Y$ intersects the positive $x$-axis, proving that $Y$ is the
unique orbit branch in the fourth quadrant connecting the origin with the
equilibrium at infinity in the negative $y$-direction.

Finally, fix $K>K^*$.  Along the critical curve $Y$, the slope field for
the parameter $K$ satisfies
\[
\left.\frac{dy}{dx}\right|_{(y,K)=(Y(x),K^*)}^{(y,K)=(Y(x),K)}
=\left(-\frac{g(x)}{Y(x)}-b_sx^s-Kh(x)\right)-Y^{\prime}(x)
=-(K-K^*)h(x)<0.
\]
Thus, for the fixed parameter $K>K^*$, $Y$ is a strict upper barrier for
system~\eqref{sys_gen}, while Lemma~\ref{lem:monotonicity} gives
$y_s(x;K)<Y(x)$ in their common domain.  On a sufficiently small
transversal $x=\varepsilon>0$, the part of the local elliptic sector
between $y_s(\varepsilon;K)$ and $Y(\varepsilon)$ contains an interval
of initial points and hence infinitely many orbit branches entering the
origin. For any one of them, let
$[\varepsilon,X_y)$ be its maximal fourth-quadrant graph.  Uniqueness
and the upper barrier give $y_s(x;K)<y(x)<Y(x)<0$ on this interval.
If $X_y<+\infty$, the orbital equation could be continued unless
$y(x)\to-\infty$ as $x\to X_y^-$; hence the branch tends to the
negative-$y$ equilibrium at infinity.  If $X_y=+\infty$, then
\eqref{eq:critical-separatrix-asymptotic} and $m-n>1$ give
$y(x)<Y(x)\to-\infty$ and $y(x)/x\to-\infty$, with the same endpoint.
Thus all these infinitely many branches connect the origin with that
equilibrium at infinity.
\end{proof}

\section{Local analysis of equilibria}\label{sec:local-analysis}

Throughout this section, we assume that $\lambda\in\Omega$.
Consequently, only the cases $b_1<0$ and $b_1=0$ are considered.

\subsection{Finite equilibria}
\begin{lemma}\label{prop1}
Let  $\Delta:={a_1}^2-4a_0 a_2.$ The following statements hold:
\begin{itemize}
\item[(a)] If $a_0=0,$ then
\begin{itemize}
\item[(i)]
There is only one equilibrium at the origin if $a_1a_2\ge 0.$
\item[(ii)]
There are three equilibria: $O$ and $E^{\pm}_1:\, \left(\pm \sqrt{-\frac{a_1}{a_2}}, 0\right)$, if
$a_1a_2<0.$ 
\end{itemize}%
\item[(b)] If $a_0 \neq 0,$ then
\begin{itemize}
\item[(i)]
If $\Delta<0,$ there is only one equilibrium at the origin.
\item[(ii)]
If $\Delta=0,$ then
\begin{itemize}
\item[($\alpha$)]
If $a_0a_1>0,$ there is only one equilibrium at the origin.
\item[($\beta$)]
If $a_0a_1<0,$ there are three equilibria: $O$ and $E^{\pm}_2:\, \left(\pm \sqrt{-\frac{2a_0}{a_1}}, 0\right).$
\end{itemize}
\item[(iii)]
If $\Delta>0,$ then
\begin{itemize}
\item[($\alpha$)]
If $a_0,a_1,a_2>0$ or $a_0,a_1,a_2<0$, there is only one equilibrium at the origin.
\item[($\beta$)]
If $a_2<0<a_0,$ there are three equilibria: $O$ and
$E^{\pm}_3:\, \left(\pm \sqrt{\frac{-a_1-\sqrt{\Delta}}{2a_2}}, 0\right).$
\item[($\gamma$)]
If $a_0<0<a_2,$ there are three equilibria: $O$ and
$E^{\pm}_4:\, \left(\pm \sqrt{\frac{-a_1+\sqrt{\Delta}}{2a_2}}, 0\right).$
\item[($\delta$)]
If $a_1<0<a_0$ and $a_2>0$, or if $a_0<0$, $a_2<0$, and $a_1>0$, there are five equilibria: $O$, $E^{\pm}_3$, and $E^{\pm}_4.$
\end{itemize}
\end{itemize}

\end{itemize}%
\end{lemma}

\begin{lemma}[Origin] Consider the differential system \eqref{main}. The qualitative behaviour of the origin is determined as follows:

\textbf{I. Case $a_0 = 0.$}
\begin{itemize}
\item[(a)]
The origin is a saddle if either $a_1>0$ or $a_1=0$ and $a_2>0$.
\item[(b)]
The origin is a center  when  
(i) $a_1<0$ and ${b_0}^2+8a_1<0;$ or 
 (ii)
$a_1=b_0=0$ and $a_2<0.$
\item[(c)]
The phase portrait of the origin consists of one hyperbolic and one elliptic sector as in Figure \ref{elihyp}, if
(i)
$a_1<0$ and ${b_0}^2+8a_1\geq 0;$ or
(ii)
$a_1=0,$ $b_0\neq 0$ and $a_2<0.$
\end{itemize}
\textbf{II. Case $a_0 \neq 0.$}
The origin is a saddle if $a_0>0$ and 
 is a center for $a_0<0$.
\end{lemma}

\begin{proof}
Since the differential system \eqref{main} is invariant under the transformation $(x,t)\to(-x,-t)$, any focus at the origin is excluded. Therefore, whenever the origin is either a center or a focus, it must in fact be a center.

Let $a_0=0.$ Then the origin is a nilpotent equilibrium of system \eqref{main}. 
The statements in \textbf{Case I} follow directly from \cite[Theorem 3.5]{DumortierLlibreArtes2006}.
Now suppose $a_0\neq 0.$ The eigenvalues of the Jacobian matrix at the origin are $\pm\sqrt{a_0}.$ 
Hence, the conclusions in \textbf{Case II} follow.
\end{proof}

\begin{figure}[!t]
\centering
\subfigure[]
{\makebox[.25\columnwidth][c]{\includegraphics[width=.25\columnwidth,height=.16\columnwidth,keepaspectratio]{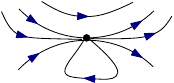}}\label{elihyp}}
\hspace{.04\columnwidth}
\subfigure[Saddle-node  ]
{\makebox[.25\columnwidth][c]{\includegraphics[width=.25\columnwidth,height=.16\columnwidth,keepaspectratio]{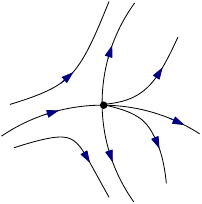}}\label{ehh}}
\hspace{.04\columnwidth}
\subfigure[Cusp ]
{\makebox[.25\columnwidth][c]{\includegraphics[width=.25\columnwidth,height=.16\columnwidth,keepaspectratio]{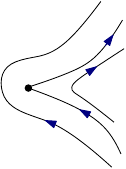}}\label{ehhhh}}
\caption{Phase portraits of equilibria. The nilpotent equilibria are shown in (1) and (3), and the semi-hyperbolic equilibrium in (2).}
\end{figure}

\begin{lemma}[$E^{+}_1$]\label{propE1}
The qualitative behaviour of the equilibrium point $E_1^{+}$ is described as follows:
\begin{itemize}
\item[(a)] If $a_2 > 0$, then it is a saddle.
\item[(b)] If $a_2 < 0$ and $b_0 \neq \frac{a_1 b_1}{a_2},$ then 
\begin{itemize}
\item[(i)] it is an unstable focus (resp.\ unstable node) if
\[
\frac{a_1 b_1}{a_2} < b_0 < \frac{a_1 b_1}{a_2} + 2\sqrt{2a_1}
\quad \left(\text{resp.\ } b_0 \geq \frac{a_1 b_1}{a_2} + 2\sqrt{2a_1}\right).
\]
\item[(ii)] it is a stable focus (resp.\ stable node) if
\[
\frac{a_1 b_1}{a_2} - 2\sqrt{2a_1} < b_0 < \frac{a_1 b_1}{a_2}
\quad \left(\text{resp.\ } b_0 \leq \frac{a_1 b_1}{a_2} - 2\sqrt{2a_1}\right).
\]
\end{itemize}

\item[(c)] If $a_2 < 0$ and $b_0 = \frac{a_1 b_1}{a_2},$ then 
 it is an unstable weak focus of order 1 for $b_1<0$ and 
  a center for $b_1=0$.
\end{itemize}
\end{lemma}

\begin{proof}
The determinant and  trace of the Jacobian matrix at $E^{+}_1$ are $-\frac{2a_1^2}{a_2}$ and  $ b_0 \sqrt{-\frac{a_1}{a_2}}+b_1 \left(-\frac{a_1}{a_2} \right)^{3/2},$ respectively. Hence, the conclusions in parts (a) and (b) follow immediately.

Assume $b_0 = \frac{a_1 b_1}{a_2}$ and $b_1<0.$ Shift the equilibrium point $E_1^+$ to the origin, and we have
\begin{align}\label{main2}
\dot{x}=&y,\\\nonumber
\dot{y}=&a_2{x}^{5}+5a_2\sqrt {-{\frac {a_1}{a_2}}}{x}^{4}+b_1{x}^{3}y-9a_1{x}^{3}
+3b_1\sqrt {-{\frac {a_1}{a_2}}}{x}^{2}y-7a_1\sqrt {-{\frac {a_1}{a_2}}}{x}^{2}-2
{\frac {b_1a_1xy}{a_2}}+2{\frac {{a_1}^{2}x}{a_2}}.
\end{align}
By using the linear transformation
$(x,y)\rightarrow\left(u,\sqrt{-\frac{2}{a_2}}\,a_1v\right)$
and the time rescaling $\tau =-\sqrt{-\frac{2}{a_2}} a_1\, t,$ system \eqref{main2} turns into the Jordan form 
\begin{equation}\label{Jordan1}
\begin{array}{l}
\dot{u}=-v,\\
\dot{v}=u-\sqrt{-\frac{2}{a_2}}\, b_1\, u v
-\frac{7\sqrt{-\frac{a_1}{a_2}}\, a_2 u^2}{2 a_1}
+\frac{3\sqrt{-\frac{2}{a_2}}\sqrt{-\frac{a_1}{a_2}}\, a_2 b_1 u^2 v}{2 a_1}
-\frac{9 a_2 u^3}{2 a_1}+\frac{\sqrt{-\frac{2}{a_2}}\, a_2 b_1  u^3 v}{2 a_1}
+\frac{5\sqrt{-\frac{a_1}{a_2}}\, a_2^2 u^4}{2 a_1^2}+\frac{a_2^2 u^5}{2 a_1^2}.
\end{array}
\end{equation}
Using the algorithm provided in \cite{ZhangHouZeng2000}, the first Lyapunov coefficient is given by
\begin{align*}
L^{1+}:=\frac{\sqrt{2a_1}{b_1}\pi}{2 a_1}.
\end{align*}
Since $b_1<0$, one has $L^{1+}<0$. Hence, the origin of
system \eqref{Jordan1} is a stable weak focus of order 1, and consequently
the equilibrium point $E_1^+$ of system \eqref{main} is an unstable weak
focus of order 1.

Now let $b_0 = \frac{a_1 b_1}{a_2}$. If $b_1 = 0$, then $b_0 = 0$, and $E_1^+$ has a pair of purely imaginary eigenvalues. So, it is either a center or a focus. In this case, system \eqref{main} is Hamiltonian with Hamiltonian function $H(x,y)=\frac{1}{2}y^2-\frac{1}{4} a_1 x^4-\frac{1}{6} a_2 x^6.$ Therefore, the equilibrium point $E_1^{+}$ is a center.
\end{proof}

\begin{lemma}[$E^{+}_2$]\label{propE2}
Assume $\Delta=0.$ Then equilibrium $E^{+}_2$
is a saddle-node as in Figure \ref{ehh} if $a_1b_0\neq 2 b_1 a_0$ 
and a cusp as in Figure \ref{ehhhh} if $a_1 b_0=2 b_1 a_0.$
\end{lemma}

\begin{proof}
The eigenvalues of the Jacobian matrix at $ E_2^{+} $ are $0$ and 
$ b_0\sqrt{-\frac{2a_0}{a_1}}+b_1 \sqrt{\left(-\frac{2a_0}{a_1}\right)^3}$. Therefore, the equilibrium point $E^{+}_2$ is semi-hyperbolic when $a_1b_0\neq 2 b_1 a_0$.  Using \cite[Theorem 2.19]{DumortierLlibreArtes2006}, it is a saddle-node.
Further, it becomes a nilpotent equilibrium when $a_1 b_0=2 b_1 a_0.$ By using  \cite[Theorem 3.5]{DumortierLlibreArtes2006}, it is a cusp.
\end{proof}

\begin{lemma}[$E^{+}_3$]\label{propE3}
Let $b_0^{\ast}:=\frac{b_1}{2a_2}\left(a_1+\sqrt{\Delta}\right).$ Then equilibrium  $E_3^{+}$ satisfies the following:
\begin{itemize}
\item[(a)] If $b_0 < b_0^{\ast}$, then it is a stable equilibrium point.
\item[(b)] If $b_0> b_0^{\ast}$, then it is an unstable equilibrium point.
\item[(c)]  Let $b_0 = b_0^{\ast}.$
\begin{itemize}
\item[(i)]
If $a_2<0$ and $b_1<0$, then it is an unstable weak focus of order 1.
\item[(ii)]
If $a_2>0$ and $b_1<0$, then it is a stable weak focus of order 1.
 \end{itemize}
\item[(d)] If $b_0 = b_0^{\ast}$ and $b_1 = 0$, then it is a center.    
\end{itemize}
\end{lemma}

\begin{proof}
The determinant of the Jacobian matrix at $E^{+}_3$ is given by $D_3=4a_0-\frac{{a_1}^2}{a_2}-\frac{a_1 \sqrt{\Delta}}{a_2}.$ Therefore,
\begin{itemize}
\item[(A)]
Let $ a_2 < 0 < a_0 $. Since $ a_2 < 0 $, we have $ -a_1 - \sqrt{\Delta} < 0 $. Hence $ a_2 D_3 < 0 $, and consequently $ D_3 > 0 $.
\item[(B)]
Let $a_1<0<a_0$ and $a_2>0$. Since $a_2>0$, it follows that $-a_1-\sqrt{\Delta}>0$. Thus $a_2D_3>0$, and hence $D_3>0$.
\item[(C)]
Let $a_0<0$, $a_2<0$, and $a_1>0$. Then $a_2D_3<0$, which implies $D_3>0$.
\end{itemize}
Furthermore, the trace of the Jacobian matrix at $E^{+}_3$ is given by
$T= \frac{1}{2}b_0\sqrt{-{\frac{2a_1+2\sqrt{\Delta}}{a_2}}}
+\frac{1}{8} b_1 \left(-{\frac{2a_1+2\sqrt{\Delta}}{a_2}}\right)^{3/2}.$ Therefore, $T$ is positive (resp. negative) when $b_0$ is greater (resp. less) than ${b_0}^{\ast}.$ Thus, the results stated in parts (a) and (b) follow directly.

Assume $b_0 = b_0^{\ast}$ and $b_1<0$. Repeating the computation in
Lemma~\ref{propE1} and applying the linear transformation
$\left(x,y\right)\rightarrow\left(u,
\frac{\sqrt{-a_2 \left({a_1}^{2}-4a_2a_0+a_1\sqrt{{a_1}^2-4a_2a_0}\right)}}{a_2}\,v\right)$
together with the time rescaling
$\tau=-\frac{\sqrt{-a_2 \left({a_1}^{2}-4a_2a_0+a_1\sqrt{{a_1}^2-4a_2a_0}\right)}}{a_2}t,$
we obtain the Jordan form
\begin{align}\label{Jordan2}
\dot{u}&=-v,\\\nonumber
\dot{v}&=-\frac{a_0a_2 \left(\alpha+u \right) +a_1a_2 \left(\alpha+u \right)^{3}
+{a_2}^{2} \left(\alpha+u \right)^{5}
-\sqrt{{a_2}^{2} \left( \alpha^{2}a_1+4a_0 \right) } \left( b_1\alpha \left( \alpha+u \right) 
-b_1 \left(\alpha+u \right)^{3}\right) v}{2 a_2 \left(\alpha^{2}a_1+2a_0 \right)},
\end{align}
where $\alpha=\frac{1}{2} \sqrt{\frac{-2(a_1+\sqrt{{a_1}^2-4a_0a_2})}{a_2}}.$  
The first  Lyapunov coefficient corresponding to  $E_3^{+}$ is
\begin{small}
\begin{eqnarray*}
&L^{3+}=-\frac {32\sqrt{-{\frac {2a_1+2\sqrt {\Delta}}{a_2}}}
\sqrt{-a_2 \left( {a_1}^{2}-4a_2a_0+a_1\sqrt {\Delta} \right) }{ b_1}\pi }
{\left( {a_1}^{2}-4a_2a_0+a_1\sqrt {\Delta} \right)^{8}}
\Big(a_1{\Delta}^{7/2} \left({a_1}^{2}-2a_2a_0\right)
\left({a_1}^{4}-4{a_1}^{2}a_2a_0+2a_2^{2}{a_0}^{2} \right)
\qquad\qquad\qquad\qquad&\\\nonumber
&\qquad\qquad\qquad\qquad\qquad\qquad\qquad\qquad\qquad\qquad
-{\Delta}^{3} \left(2a_2^{4}{a_0}^{4}-16{a_1}^{2}a_2^{3}{a_0}^{3}
+20{a_1}^{4}{a_2}^{2}{a_0}^{2}-8{a_1}^{6}a_2a_0
+{a_1}^{8} \right) \Big).&
\end{eqnarray*}
\end{small}
Since $b_1<0$, one has $L^{3+}>0$. Hence, the origin of
the Jordan system \eqref{Jordan2} is an unstable weak focus of order 1.
Consequently, the equilibrium point $E_3^+$ of system \eqref{main} is an
unstable weak focus of order 1 when $a_2<0$, and a stable weak focus of
order 1 when $a_2>0$.

Now let $b_0 = b_0^{\ast}$. If $b_1 = 0$, then $b_0 = 0$, and $E_3^{+}$ has a pair of purely imaginary eigenvalues. Then, it is either a center or a focus. In this case, system \eqref{main} is Hamiltonian with Hamiltonian function
$H(x,y)=\frac{1}{2}y^2-\frac{1}{2}a_0x^2
-\frac{1}{4} a_1 x^4-\frac{1}{6} a_2 x^6.$ Consequently, the equilibrium point $E_3^{+}$ is a center.
\end{proof}

\begin{lemma}[$E^{+}_4$]\label{propE4}
The equilibrium $E^{+}_4$ is always a saddle.
\end{lemma}
\begin{proof}
The determinant of the Jacobian matrix at $E^{\pm}_4$ is given by $D_4=4a_0-\frac{{a_1}^2}{a_2}+\frac{a_1 \sqrt{\Delta}}{a_2}.$ Therefore,

Let $a_0<0<a_2$. Since $a_2>0$, we have $-a_1+\sqrt{\Delta}>0$. Hence $a_2D_4<0$ and $D_4<0$.

Let $a_1<0<\min\{a_0, a_2\}$. Then $a_2D_4\leq a_1\sqrt{\Delta}$. Since $a_1<0$, we have $a_2D_4<0$.
Hence $D_4<0$.

Let $\max\{a_0, a_2\}<0<a_1$. Since $\left|a_1\sqrt{\Delta}\right|>\left|\Delta\right|$, we have $a_2D_4>0$. Hence, $D_4<0$.
\end{proof}

\subsection{Equilibria at infinity}\label{subsec:equilibria-at-infinity}

\begin{lemma}\label{propIn}
The dynamics of system \eqref{main} near infinity are topologically equivalent to one
of the phase portraits of Figure \ref{NearInf}.
\end{lemma}
\begin{figure}[t!]
\begin{center}
\subfigure[$b_1 <0$  and $a_2<0$ ]
{\includegraphics[width=.2\columnwidth,height=.21\columnwidth]{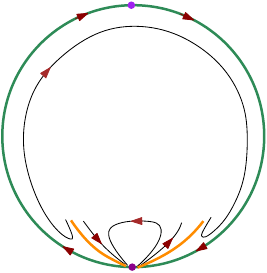}\label{b1n0a2n}}
\quad
\subfigure[  $b_1 = 0 $  and $a_2<0$ ]
{\includegraphics[width=.2\columnwidth,height=.21\columnwidth]{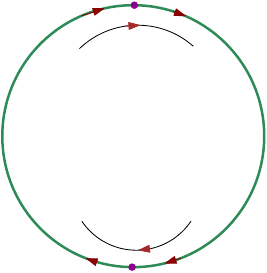}\label{b10a2n}}
\quad
\subfigure[$a_2>0$ and $b_1\leq0$]
{\includegraphics[width=.2\columnwidth,height=.21\columnwidth]{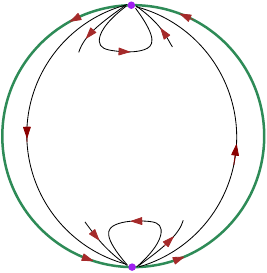}\label{b10a2p}}
\caption{\label{NearInf} Dynamics near infinity in the Poincar\'e disc.}
\end{center}
\end{figure}
\begin{figure}[t!]
\begin{center}
\subfigure[Phase portrait on the blow-up circle  for $b_0=b_1=0$ and $a_2>0$ ]
{\includegraphics[width=.2\columnwidth,height=.2\columnwidth]{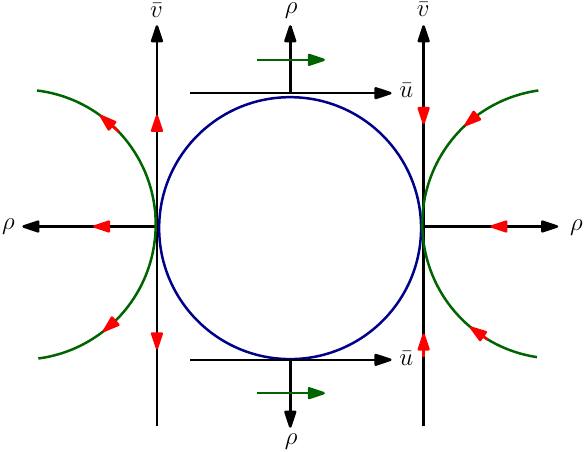}\label{Ca2p}}
\quad
\subfigure[Blowing down of the origin of the local chart $U_2$ for $b_0=b_1=0$ and $a_2>0$  ]
{\includegraphics[width=.2\columnwidth,height=.2\columnwidth]{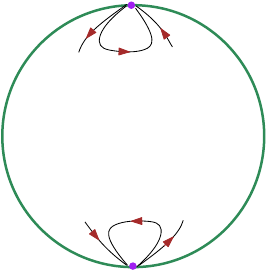}\label{Ba2p}}
\quad
\subfigure[Phase portrait on the blow-up circle  for $b_0=b_1=0$ and  $a_2<0$ ]
{\includegraphics[width=.2\columnwidth,height=.2\columnwidth]{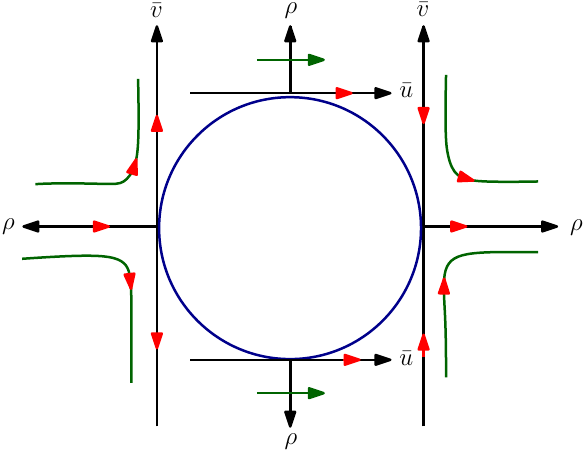}\label{Ca2n}}
\subfigure[Blowing down of the origin of the local chart $U_2$ for $b_0=b_1=0$ and $a_2<0$]
{\includegraphics[width=.2\columnwidth,height=.2\columnwidth]{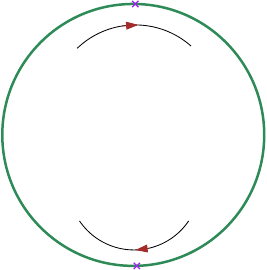}\label{Ba2n}}
\caption{\label{NearInf1} Local phase portraits of system $X_1^{\pm}$ and $Y_1^{\pm}$ at the origin.}
\end{center}
\end{figure}
\begin{proof}
The differential system \eqref{main} in the local chart $U_1$ is expressed by
\begin{align}\label{11}
\dot{u}=& a_2 + a_1 v^2+b_1 u v+a_0 v^4+b_0 u v^3-u^2 v^4,\\\nonumber
\dot{v}=& -u v^5,
\end{align}
while it has no equilibria on $v=0.$ Further,  in the local chart $U_2,$ it is expressed by
\begin{align}\label{12}
\dot{u}=&v^4-b_0 u^2 v^3-b_1 u^4 v-a_0 u^2 v^4-a_1 u^4 v^2-a_2 u^6,\\\nonumber
\dot{v}=&-u v\left(b_0 v^3+ b_1 u^2 v+a_2 u^4+a_1 u^2 v^2+a_0 v^4\right),
\end{align}
and the origin is a non-hyperbolic equilibrium. When $b_0^2+b_1^2\neq0$, according to \cite{DumortierHerssens1999}, dynamics near infinity in the Poincar\'e disc for differential system \eqref{main} are shown in Figure \ref{NearInf}.

Let $b_0=b_1=0.$ The origin is a  non-hyperbolic equilibrium for system \eqref{12}. Based on the Newton diagram of system \eqref{12}, the quasi-homogeneous blow-up is chosen with weights $(5,3)$:
\begin{equation}\label{qu1e}
(u, v) =(\bar{u} \rho^{5}, \bar{v} \rho^3). 
\end{equation}
The quasi-homogeneous blow-up transformation \eqref{qu1e} with $\bar{u} = \pm 1$
(resp. $\bar{v} = \pm1$) converts system \eqref{12} to the associated vector field $X_1^{\pm}$ (resp. $Y_1^{\pm}$). These vector fields are defined as follows:
\begin{align}
X_1^{\pm}:~~
\begin{cases}
\dfrac{\mathrm{d}\rho}{\mathrm{d}\tilde{\lambda}} =
\mp \dfrac{1}{5}\rho \left(a_2\rho^{18}+a_1\rho^{14}\bar{v}^2+a_0\rho^{10}\bar{v}^4-\bar{v}^4\right) ,\\\nonumber
\dfrac{\mathrm{d}\bar{v}}{\mathrm{d}\tilde{\lambda}} =\mp \dfrac{1}{5} \bar{v}(2a_2\rho^{18}+2a_1\rho^{14}\bar{v}^2+2a_0\rho^{10}\bar{v}^4+3\bar{v}^4),
\end{cases}
~~~ Y_1^{\pm}:&
\begin{cases}
\dfrac{\mathrm{d}\rho}{\mathrm{d}\tilde{\lambda}} = \mathcal{O}(|(\rho, \bar{u})|), \\\\
\dfrac{\mathrm{d}\bar{u}}{\mathrm{d}\tilde{\lambda}} = 1 + \mathcal{O}(|(\rho, \bar{u})|),
\end{cases}
\end{align}
where $\mathrm{d}\tilde{\lambda} = \rho^{7} d\tau.$  Dynamics near infinity in the Poincar\'e disc for differential system \eqref{main} when $b_0=b_1=0$ are shown in Figure \ref{NearInf1}.
\end{proof}

Combining the preceding local classifications, we obtain the existence of
all the local bifurcations and degeneracies listed in
items~(i)--(viii) of Theorem~\ref{thm-bifurcation}.  More precisely,
Lemma~\ref{prop1} together with the local classification of the origin
yields the pitchfork set $P$, the generalized pitchfork set $GP$, and the
nilpotent center--sector degeneracy set $N$; Lemma~\ref{propE2} yields the
saddle-node set $SN$ and, when its nonzero eigenvalue also vanishes, the
Bogdanov--Takens set $BT$ at $E_2^\pm$; Lemmas~\ref{propE1} and
\ref{propE3} yield the Hopf set $H$, while their Hamiltonian subcases,
together with the classification of the origin, yield the Hamiltonian-center
degeneracy set $HAM$.  Finally, Lemma~\ref{propIn} yields the bifurcation set
$INF$ at infinity.  The periodic-annulus-breaking, homoclinic,
heteroclinic and limit-cycle-breaking bifurcations in items~(ix)--(xi)
are global and require the analysis in the subsequent sections.
For normal-form methods concerning fine foci and degenerate Hopf
bifurcations, see \cite{WuXuZhang2021}.

\section{Limit cycles and homoclinic loops for $\lambda\in\Omega$}\label{s5}
System \eqref{main} is reversible with respect to the $y$-axis, i.e. it is invariant under the transformation $(x, y, t)\to (-x, y,-t)$. 
\begin{lemma}\label{lem:reversible-periodic-orbit}
Every periodic orbit of system~\eqref{main} that meets the $y$-axis
belongs to a periodic annulus.  In particular, no limit cycle meets the
$y$-axis.
\end{lemma}

\begin{proof}
Let $R(x,y)=(-x,y)$ and let a periodic orbit meet
$\operatorname{Fix}(R)=\{x=0\}$ at two consecutive points $p$ and $q$.
Both intersections are transversal, since $\dot x=y\ne0$ there.
For every $p'$ in a small interval of $\operatorname{Fix}(R)$ about
$p$, continuous dependence gives a transverse next intersection
$q'$ with $\operatorname{Fix}(R)$.  Reversibility maps the orbit segment
from $p'$ to $q'$ onto its complementary segment, so their union is a
periodic orbit.  These nearby, pairwise disjoint periodic orbits form a
periodic annulus containing the original orbit.
\end{proof}
Thus the study of limit cycles can be restricted to the region $x>0$.

\begin{lemma}\label{lem11}
For $\lambda\in\Omega$, the differential system
\eqref{main} admits no limit cycles if $b_0b_1\geq0$ and
$b_0^2+b_1^2\neq0$.
\end{lemma}

\begin{proof}
Computing the divergence of system~\eqref{main}, we obtain
\[
\mathrm{div}\,\eqref{main} =  b_0 x + b_1 x^3.
\]
If $b_0 b_1 \ge 0$ and $(b_0, b_1) \neq (0,0)$, then for $x > 0$ the divergence $\mathrm{div}\,\eqref{main}$ is nonzero and does not change sign. Hence, by Bendixson's criterion, system~\eqref{main} admits no limit cycles in the region $x > 0$.
\end{proof}

\subsection{The cases
$\lambda\in\Omega_1\cup\Omega_2\cup\Omega_3\cup\Omega_8\cup\Omega_9$}
\begin{lemma}\label{lem12}
For
$\lambda\in\Omega_1\cup\Omega_2\cup\Omega_3\cup\Omega_9$,
system~\eqref{main} has no limit cycles in the region $x>0$.
\end{lemma}

\begin{proof}
According to Lemma~\ref{prop1}, when $a_1 a_2 \ge 0$, system~\eqref{main} has no equilibria in the region $x > 0$ and therefore admits no limit cycles. 
Furthermore, by Lemma~\ref{propE1}, if $a_2 > 0$ and $a_1 < 0$, the system has a unique equilibrium $E_1^+$ in the region $x > 0$, which is a saddle with index $-1$. Consequently, system~\eqref{main} admits no limit cycles in this case as well.
\end{proof}

Consequently, among the cases with $a_0=0$, it suffices to
study limit cycles of system~\eqref{main} for $\lambda\in\Omega_8$.
After the scaling transformation 
 $(x, y, t) \mapsto \left(\sqrt{-\frac{a_1}{a_2}} x, -\frac{a_1\sqrt{a_1}}{a_2} y, \frac{\sqrt{-a_2}}{a_1} t\right)$, system~\eqref{main} reduces to
\begin{equation}\label{sys1}
\begin{aligned}
\dot{x} &= y, \\
\dot{y} &= x^3 - x^5+ y(\bar{b}_0 x + \bar{b}_1 x^3) =: -g(x) - yf(x),
\end{aligned}
\end{equation}
where $\bar{b}_0=\frac{b_0}{\sqrt{a_1}}$ and $\bar{b}_1=-\frac{b_1 \sqrt{a_1}}{a_2}$. 

\begin{remark}
Since $\lambda\in\Omega_8$, only $\bar b_1<0$, or
$\bar b_1=0$ with $\bar b_0\leq0$, is considered. The half-plane
obtained from these cases by the central parameter symmetry need not be
treated separately.
\end{remark}

\begin{lemma}\label{lem:omega8-limit-cycles}
For $\lambda\in\Omega_8$, the following assertions hold for
system~\eqref{sys1} in the region $x>0$:
\begin{itemize}
\item[(i)] there are no limit cycles if
$\bar b_0+\bar b_1=0$;
\item[(ii)] there are no limit cycles if
$\bar b_0(\bar b_0+\bar b_1)>0$;
\item[(iii)] there is at most one limit cycle if
$\bar b_0(\bar b_0+\bar b_1)<0$, and it is unstable whenever it
exists.
\end{itemize}
\end{lemma}

\begin{proof}
We prove the three assertions successively.

For (i), if $\bar{b}_0 = \bar{b}_1 = 0$, then
system~\eqref{sys1} is Hamiltonian and therefore admits no limit cycles.

Now assume that $\bar{b}_0 + \bar{b}_1 = 0$ and $(\bar{b}_0, \bar{b}_1) \neq (0,0)$. System~\eqref{sys1} can be equivalently written as
\begin{equation}\label{sys2}
\begin{aligned}
\dot{x} &= y - F(x), \\
\dot{y} &= -g(x+1),
\end{aligned}
\end{equation}
where $F(x) := \int_1^{1+x} f(s)\,\mathrm{d}s$. Then the following properties hold:
\begin{itemize}
\item[(a)] $f(1+x) = F^{\prime}(x)$ has a unique zero at
$x = 0$ in $(-1, \infty)$. When $\bar{b}_0>0$, one has
$f(1+x) < 0$ for $x \in (-1,0)$ and $f(1+x) > 0$ for
$x \in (0,\infty)$.

\item[(b)] $F(0) = 0$.

\item[(c)] $x\,g(1+x) > 0$ for all $x \in (-1,\infty)\setminus\{0\}$.
\end{itemize}

Moreover, a direct computation yields
$\frac{f(x)}{g(x)} = -\frac{\bar{b}_1 x^3 + \bar{b}_0 x}{x^5 - x^3} = \frac{\bar{b}_0}{x^2}$,
which is monotone. Consequently, the equation
$\frac{f(x_1)}{g(x_1)} = \frac{f(x_2)}{g(x_2)}$
has no solutions with $-1 < x_1 < 0 < x_2$. Therefore,
by~\cite[Corollary~10]{ChenLlibreTang2018}, system~\eqref{sys11}
admits no limit cycles in $(-1,\infty)$.

For (ii) and (iii), set
$\lambda(x):=\frac{f(x)}{g(x)}
=-\frac{\bar b_1x^2+\bar b_0}{x^4-x^2}$.
The behavior of this function controls the existence and number of limit
cycles in system~\eqref{sys1}.

To prove (ii), we first show that there exists a constant
$c \in \mathbb{R}$ such that the equation $\lambda(x)=-c$ has no
solution for $x \in (0,\infty) \setminus \{1\}$. Indeed,
$\lambda(x)=-c$ is equivalent, after setting $u=x^2$, to the following
quadratic equation:
\begin{equation}\label{c}
cu^2-(c+\bar{b}_1)u-\bar{b}_0=0.
\end{equation}
Its discriminant is given by
\[
\Delta_u(c)=c^2+2(\bar{b}_1+2\bar{b}_0)c+\bar{b}_1^2,
\]
which is a quadratic polynomial in $c$. 
The discriminant of $\Delta_u(c)$ with respect to $c$ is
\[
\Delta_c=16\bar{b}_0(\bar{b}_1+\bar{b}_0)>0.
\]
Therefore, there exists some $c \in \mathbb{R}$ such that $\Delta_u(c)<0$, which implies that equation \eqref{c} has no real roots. Consequently, the equation $\lambda(x)=-c$ has no solutions for $x \in (0,\infty) \setminus \{1\}$.
Furthermore,  we have $(x-1) g(x)>0$ for $x \in (0,\infty) \setminus \{1\}$ and $f(1)=-(\bar{b}_0+\bar{b}_1)\neq 0.$
By~\cite[Lemma~2.1]{DumortierKooijLi2000}, system~\eqref{sys1}
admits no limit cycles in the region $x>0$ under the condition in (ii).

Finally, for (iii), we show that $\lambda(x)$ is monotonic on
$(0,\infty) \setminus \{1\}$. Taking its derivative yields
$$
\lambda^{\prime}(x)=\frac{2x\phi(x)}{(x^4-x^2)^2},
$$
where $\phi(x):=\bar{b}_1x^4+2\bar{b}_0x^2-\bar{b}_0$. In the region $x \in (0,\infty) \setminus \{1\}$, the sign of $\lambda^{\prime}(x)$ depends entirely on the sign of the polynomial $\phi(x)$. We consider the discriminant of the quadratic polynomial $\phi(\sqrt{u})$ with respect to $u$, which is given by
$$
\Delta_\phi=4\bar{b}_0(\bar{b}_0+\bar{b}_1)<0.
$$
Here $\bar b_1<0$, so $\phi(\sqrt{u})<0$, which implies
that $\lambda(x)$ is monotonically decreasing on
$(0,\infty)\setminus\{1\}$.

Furthermore, $(x-1) g(x)>0$ for $x \in (0,\infty) \setminus \{1\}$ and $\left(x-\sqrt{-\frac{\bar{b}_0}{\bar{b}_1}} \right) f(x)$ has a fixed sign for $x \in (0,\infty) \setminus \left\{\sqrt{-\frac{\bar{b}_0}{\bar{b}_1}}\right\}$.
By \cite[Theorem 4.7 or 4.9, Chapter 4]{ZhangEtAl1992} or
\cite[Lemma 2.2]{DumortierKooijLi2000}, it follows that system \eqref{sys1} has at most
one limit cycle in the region $x>0$ under the condition in (iii).
Under this condition and $\bar b_1<0$, one has $\bar b_0>0$; hence
the limit cycle is unstable if it exists.
\end{proof}


\subsection{The cases
$\lambda\in\bigcup_{j=4}^{7}\Omega_j\cup
\bigcup_{j=10}^{15}\Omega_j$}

\begin{lemma}\label{lem17}
For
$\lambda\in\Omega_4\cup\Omega_5\cup\Omega_6\cup\Omega_7\cup
\Omega_{10}\cup\Omega_{11}\cup\Omega_{13}$, system~\eqref{main}
has no limit cycles in the region $x>0$.
\end{lemma}

\begin{proof}
According to Lemma~\ref{prop1}, when $\Delta < 0$, or $\Delta=0$ and $a_0a_1>0,$ or $\Delta>0$ and $a_0,a_1,a_2$ have the same sign, system~\eqref{main} has no equilibria in the region $x > 0$ and therefore admits no limit cycles. 
Furthermore, by Lemma~\ref{propE2}, if $\Delta=0$ and $a_0a_1<0$, the system has a unique equilibrium $E_2^+$ in the region $x > 0$, which is either a saddle-node or a cusp. Consequently, system~\eqref{main} admits no limit cycles in this case as well. 
Further, by Lemma~\ref{propE4}, if $\Delta>0$ and $a_0<0<a_2$, the system has a unique equilibrium $E_4^+$ in the region $x > 0$, which is a saddle. Then, system~\eqref{main} admits no limit cycles in this case.
\end{proof}

Consequently, among the cases with $a_0\neq0$, it suffices
to study limit cycles of system \eqref{main} for
$\lambda\in\Omega_{12}\cup\Omega_{14}\cup\Omega_{15}$.

If $a_2<0,$ after the scaling transformation 
$ (x, y, t) \mapsto \left(\sqrt{\alpha} x, \alpha \sqrt{\alpha} \sqrt{-a_2} y, \frac{1}{\sqrt{-a_2}\alpha} t\right),$
 system~\eqref{main} is reduced to
\begin{equation}\label{sys11}
\begin{aligned}
\dot{x} &= y, \\
\dot{y} &= x (1-x^2) (x^2+\beta)+ y \left(\tilde{b}_0x+\tilde{b}_1 x^3\right)=: -\tilde{g}(x) - y\tilde{f}(x),
\end{aligned}
\end{equation}
where $\tilde{b}_0=\frac{b_0}{\sqrt{-a_2}\sqrt{\alpha}},$ $\tilde{b}_1=\frac{b_1 \sqrt{\alpha}}{\sqrt{-a_2}},$ $\beta=\frac{\sqrt{\Delta}-a_1}{a_1+\sqrt{\Delta}}$ and 
$\alpha=\frac{-a_1-\sqrt{\Delta}}{2a_2}.$ 

If $a_2>0,$  using the scaling transformation 
$(x, y, t) \mapsto \left(\sqrt{\alpha} x, \alpha \sqrt{\alpha} \sqrt{a_2} y, \frac{1}{\sqrt{a_2}\alpha} t\right),$
system~\eqref{main} is reduced to
\begin{equation}\label{sys111}
\begin{aligned}
\dot{x} &= y, \\
\dot{y} &= x (x^2-1) (x^2+\beta)+ y \left(\hat{b}_0x+\hat{b}_1 x^3\right)=: -\hat{g}(x) - y\hat{f}(x),
\end{aligned}
\end{equation}
where $\hat{b}_0=\frac{b_0}{\sqrt{a_2}\sqrt{\alpha}},$ $\hat{b}_1=\frac{b_1 \sqrt{\alpha}}{\sqrt{a_2}},$  $\beta=\frac{\sqrt{\Delta}-a_1}{a_1+\sqrt{\Delta}}$ and 
$\alpha=\frac{-a_1-\sqrt{\Delta}}{2a_2}.$ 

We write $b_0,b_1$ for the normalized coefficients whenever no ambiguity arises.

\subsubsection{The case $\lambda\in\Omega_{12}$,
equivalently $\beta>0$,  for system \eqref{sys11}}

\leavevmode{System \eqref{sys11} has three equilibria: the origin, $e_1=(1,0)$, and $e_2=(-1,0)$.}

\begin{lemma}\label{lem:omega12-limit-cycles}
For $\lambda\in\Omega_{12}$, the following assertions hold for
system~\eqref{sys11} in the region $x>0$:
\begin{itemize}
\item[(i)] there are no limit cycles if
$\tilde b_0+\tilde b_1=0$;
\item[(ii)] there are no limit cycles if
$\tilde b_1(\tilde b_0+\tilde b_1)<0$;
\item[(iii)] there is at most one limit cycle if
$\tilde b_0(\tilde b_0+\tilde b_1)<0$, and it is unstable whenever
it exists.
\end{itemize}
\end{lemma}

\begin{proof}
For (i), if $\tilde{b}_0 = \tilde{b}_1 = 0$, then
system~\eqref{sys11} is Hamiltonian and therefore admits no limit cycles.

For the case $\tilde{b}_0 + \tilde{b}_1 = 0$ and
$(\tilde{b}_0, \tilde{b}_1) \neq (0,0)$, the proof is similar to
part~(i) of Lemma~\ref{lem:omega8-limit-cycles}, with the difference
that a direct computation yields
$\frac{\tilde{f}(x)}{\tilde{g}(x)} = \frac{\tilde{b}_0}{x^2 +\beta}$,
which is monotone. Consequently, the equation
$\frac{\tilde{f}(x_1)}{\tilde{g}(x_1)} = \frac{\tilde{f}(x_2)}{\tilde{g}(x_2)}$
has no solutions with $0<x_1<x_2$. Therefore,
by~\cite[Corollary~10]{ChenLlibreTang2018}, system~\eqref{sys11}
admits no limit cycles in $(0,\infty)$.

For (ii) and (iii), set
$\tilde\lambda(x):=\frac{\tilde f(x)}{\tilde g(x)}
=\frac{\tilde b_1x^2+\tilde b_0}{(1-x^2)(x^2+\beta)}$.
The behavior of this function controls the existence and number of limit
cycles in system~\eqref{sys11}.

To prove (ii), we first show that there exists a constant
$c \in \mathbb{R}$ such that the equation $\tilde\lambda(x)=-c$ has no
solution for $x \in (0,\infty)$. After setting $u=x^2$, this equation
becomes
\begin{equation}\label{c1}
cu^2-(c+\tilde{b}_1-c\beta)u-\tilde{b}_0-c\beta=0.
\end{equation}
Its discriminant is given by
\[
\Delta_u(c)=(\beta+1)^2 c^2+2(\tilde{b}_1+2\tilde{b}_0-\beta \tilde{b}_1)c+\tilde{b}_1^2,
\]
which is a quadratic polynomial in $c$. 
The discriminant of $\Delta_u(c)$ with respect to $c$ is
\[
\Delta_c=16 (\tilde{b}_1+\tilde{b}_0) (\tilde{b}_0-\beta \tilde{b}_1)>0.
\]
Therefore, there exists some $c \in \mathbb{R}$ such that $\Delta_u(c)<0$, which implies that equation \eqref{c1} has no real roots. Consequently, the equation $\tilde{\lambda}(x)=-c$ has no solutions for $x \in (0,\infty) \setminus \{1\}$.
Furthermore,  we have $(x-1) \tilde{g}(x)>0$ for $x \in (0,\infty)\setminus \{1\}$ and $\tilde{f}(1)=-(\tilde{b}_0+\tilde{b}_1)\neq 0.$
According to~\cite[Lemma~2.1]{DumortierKooijLi2000},
system~\eqref{sys11} admits no limit cycles in the region $x>0$ under
the condition in (ii).

Finally, for (iii), we show that $\tilde\lambda(x)$ is monotonic on
$(0,\infty) \setminus \{1\}$. Taking its derivative yields
$$
\tilde{\lambda}^{\prime}(x)=\frac{2x\psi(x)}{(1-x^2)^2 (x^2+\beta)^2},
$$
where $\psi(x):=\tilde{b}_1 x^4+2 \tilde{b}_0 x^2+\tilde{b}_1\beta+\tilde{b}_0 \beta-\tilde{b}_0$. In the region $x \in (0,\infty) \setminus \{1\}$, the sign of $\tilde{\lambda}^{\prime}(x)$ depends entirely on the sign of the polynomial $\psi(x)$. We consider the discriminant of the quadratic polynomial $\psi(\sqrt{u})$ with respect to $u$, which is given by
$$
\Delta_{\psi}=4 (\tilde{b}_1+\tilde{b}_0) (\tilde{b}_0-\tilde{b}_1\beta)<0.
$$
Here $\tilde b_1<0$, so $\psi(\sqrt{u})<0$, which
implies that $\tilde\lambda(x)$ is monotonically decreasing on
$(0,\infty)\setminus\{1\}$.

Furthermore, $(x-1) \tilde{g}(x)>0$ for $x \in (0,\infty) \setminus \{1\}$ and $\left(x-\sqrt{-\frac{\tilde{b}_0}{\tilde{b}_1}} \right) \tilde{f}(x)$ has a fixed sign for $x \in (0,\infty) \setminus \left\{\sqrt{-\frac{\tilde{b}_0}{\tilde{b}_1}}\right\}$.
By \cite[Theorem 4.7 or 4.9, Chapter 4]{ZhangEtAl1992} or
\cite[Lemma 2.2]{DumortierKooijLi2000}, it follows that system \eqref{sys11} has at most
one limit cycle in the region $x>0$ under the condition in (iii).
Since $\tilde b_1<0$, the limit cycle is unstable if it exists.
\end{proof}

\subsubsection{The case $\lambda\in\Omega_{15}$,
equivalently $\beta\in(-1,0)$, for system \eqref{sys11}}

\leavevmode{System \eqref{sys11} has five equilibria: the origin, $e_3^{\pm}=(\pm1,0)$,
and $e_4^{\pm}=(\pm\sqrt{-\beta},0)$.}

Since the origin is a center and $e_4^+$ is a saddle, any limit cycle of system \eqref{sys11} in $x>0$ must lie in the region $x\in(\sqrt{-\beta},\infty)$.


\begin{lemma}\label{lem:omega15-limit-cycles}
For $\lambda\in\Omega_{15}$, the following assertions hold for
system~\eqref{sys11} in $x\in(\sqrt{-\beta},\infty)$:
\begin{itemize}
\item[(i)] there are no limit cycles if
$\tilde b_0+\tilde b_1=0$;
\item[(ii)] there are no limit cycles if
$\tilde b_0-\tilde b_1\beta=0$;
\item[(iii)] there are no limit cycles if
$\tilde b_1(\tilde b_0+\tilde b_1)<0$ or
$\tilde b_0(\tilde b_0-\tilde b_1\beta)<0$;
\item[(iv)] there is at most one limit cycle if
$(\tilde b_0+\tilde b_1)(\tilde b_0-\tilde b_1\beta)<0$, and it is
unstable whenever it exists.
\end{itemize}
\end{lemma}

\begin{proof}
For (i), a direct computation yields
$\frac{\tilde{f}(x)}{\tilde{g}(x)} = \frac{ \tilde{b}_0}{x^2+\beta}$.
Since $x^2+\beta$ is positive in $(\sqrt{-\beta},\infty)$, the sign of
$\tilde f(x)/\tilde g(x)$ is determined solely by the sign of
$\tilde b_0$. Hence, it is monotone in this region, and
by~\cite[Corollary~10]{ChenLlibreTang2018}, system~\eqref{sys11}
admits no limit cycles there.

For (ii), a direct computation leads to
$\frac{\tilde{f}(x)}{\tilde{g}(x)} = \frac{ \tilde{b}_1 }{1-x^2}$.
Therefore, $\tilde f(x)/\tilde g(x)=0$ has no roots in the region
$x \in (\sqrt{-\beta},\infty)\setminus\{1\}$. Hence,
by~\cite[Corollary~10]{ChenLlibreTang2018}, system~\eqref{sys11}
admits no limit cycles in $(\sqrt{-\beta},\infty)$.


For (iii), we first prove that there exists a constant
$c \in \mathbb{R}$ such that the equation
$\tilde{\lambda}(x)=-c$ has no solution for $x \in (\sqrt{-\beta}, \infty) \setminus \{1\}$ where 
$\tilde{\lambda}=\frac{\tilde{b}_0+\tilde{b}_1 x^2}{(1-x^2) (x^2+\beta)}$. Indeed, 
$\tilde{\lambda}(x)=-c$ for 
$x \in (\sqrt{-\beta}, \infty) \setminus \{1\}$ is equivalent to $\tilde{\lambda}(\sqrt{u})=-c$ for $u \in (-\beta, \infty) \setminus \{1\}$. Thus, we consider the following quadratic equation with respect to $u$:
\begin{equation}\label{c303}
c u^2+\left(c \beta-\tilde{b}_1-c\right) u-c\beta-\tilde{b}_0=0.
\end{equation}
Its discriminant is given by
\[
\Delta_u(c)=(\beta+1)^2 c^2+2 \left(2 \tilde{b}_0+ \tilde{b}_1(1-\beta) \right) c+{\tilde{b}_1}^2,
\]
which is a quadratic polynomial in $c$. 
The discriminant of $\Delta_u(c)$ with respect to $c$ is
\[
\Delta_c=16 (\tilde{b}_1 + \tilde{b}_0) (\tilde{b}_0 - \tilde{b}_1 \beta)>0.
\]
Therefore, there exists some $c \in \mathbb{R}$ such that $\Delta_u(c)<0$, which implies that equation \eqref{c303} has no real roots. Consequently, the equation $\tilde{\lambda}(x)=-c$ has no solutions for $x \in (\sqrt{-\beta}, \infty) \setminus \{1\}$.
Furthermore,  we have $(x-1) \tilde{g}(x)>0$ for $x \in (\sqrt{-\beta}, \infty) \setminus \{1\}$ and $\tilde{f}(1)=-(\tilde{b}_0+\tilde{b}_1)\neq 0.$
By~\cite[Lemma~2.1]{DumortierKooijLi2000}, system~\eqref{sys11}
admits no limit cycles in $(\sqrt{-\beta}, \infty)$ under either
condition in (iii).

Finally, for (iv), we show that $\tilde\lambda(x)$ is monotonic on
$(\sqrt{-\beta}, \infty) \setminus \{1\}$. Taking its derivative yields
$$
\tilde{\lambda}^{\prime}(x)=\frac{2x\psi(x)}{(1-x^2)^2 (x^2+\beta)^2},
$$
where $\psi(x):=\tilde{b}_1 x^4+2 \tilde{b}_0 x^2+\tilde{b}_1\beta+\tilde{b}_0 \beta-\tilde{b}_0$. In the region $x \in (\sqrt{-\beta}, \infty) \setminus \{1\}$, the sign of $\tilde{\lambda}^{\prime}(x)$ depends entirely on the sign of the polynomial $\psi(x)$. We consider the discriminant of the quadratic polynomial $\psi(\sqrt{u})$ with respect to $u$, which is given by
$$
\Delta_{\psi}=4 (\tilde{b}_1+\tilde{b}_0) (\tilde{b}_0-\tilde{b}_1\beta)<0.
$$
Here $\tilde b_1<0$, so $\psi(\sqrt{u})<0$, which
implies that $\tilde\lambda(x)$ is monotonically decreasing on
$(\sqrt{-\beta},\infty)\setminus\{1\}$.

By \cite[Theorem 4.7 or 4.9, Chapter 4]{ZhangEtAl1992} or
\cite[Lemma 2.2]{DumortierKooijLi2000}, it follows that system \eqref{sys11} has at most
one limit cycle in $(\sqrt{-\beta},\infty)$ under the condition in
(iv). Since $\tilde b_1<0$, the limit cycle is unstable if it exists.
\end{proof}

\subsubsection{The case $\lambda\in\Omega_{14}$,
equivalently $\beta<-1$, for system \eqref{sys111}}

\leavevmode{System \eqref{sys111} has five equilibria: the origin, $e_3^{\pm}=(\pm1,0)$,
and $e_4^{\pm}=(\pm\sqrt{-\beta},0)$.}

Since both the origin and $e^+_4$ are saddles, any limit cycle of system \eqref{sys111} that lies at $x>0$ can exist only for $x\in(0,\sqrt{-\beta})$.

\begin{lemma}\label{lem:omega14-limit-cycles}
For $\lambda\in\Omega_{14}$, the following assertions hold for
system~\eqref{sys111} in $x\in(0,\sqrt{-\beta})$:
\begin{itemize}
\item[(i)] there are no limit cycles if
$\hat b_0+\hat b_1=0$;
\item[(ii)] there are no limit cycles if
$\hat b_0-\hat b_1\beta=0$;
\item[(iii)] there are no limit cycles if
$\hat b_1(\hat b_0+\hat b_1)>0$ or
$\hat b_0(\hat b_0-\hat b_1\beta)>0$;
\item[(iv)] there is at most one limit cycle if
$(\hat b_0+\hat b_1)(\hat b_0-\hat b_1\beta)<0$, and it is stable
whenever it exists.
\end{itemize}
\end{lemma}

\begin{proof}
For (i), a direct computation gives
$\frac{\hat{f}(x)}{\hat{g}(x)}=\frac{\hat{b}_0}{x^2+\beta}$.
For $x\in(0,\sqrt{-\beta})$, we have $x^2+\beta<0$. Hence, the sign of
$\frac{\hat{f}(x)}{\hat{g}(x)}$ is determined solely by the sign of $\hat{b}_0$.
Therefore, $\frac{\hat{f}(x)}{\hat{g}(x)}$ is monotone on $(0,\sqrt{-\beta})$.
Consequently, by \cite[Corollary~10]{ChenLlibreTang2018},
system~\eqref{sys111} admits no limit cycles in
$x\in(0,\sqrt{-\beta})$.

For (ii), a direct computation yields
$\frac{\hat{f}(x)}{\hat{g}(x)}=\frac{\hat{b}_1}{x^2-1}$.
Therefore, $\frac{\hat{f}(x)}{\hat{g}(x)}=0$ has no roots in the region
$x\in(0,\sqrt{-\beta})\setminus\{1\}.$
Hence, by \cite[Corollary~10]{ChenLlibreTang2018},
system~\eqref{sys111} admits no limit cycles in
$(0,\sqrt{-\beta})$.

For (iii), we first prove that there exists a constant
$c \in \mathbb{R}$ such that the equation
$\hat{\lambda}(x)=-c$ has no solution for $x \in (0,\sqrt{-\beta}) \setminus \{1\}$ where 
$\hat{\lambda}=\frac{\hat{b}_0+\hat{b}_1 x^2}{(x^2-1) (x^2+\beta)}$. Indeed, $\hat{\lambda}(x)=-c$ for 
$x \in (0,\sqrt{-\beta}) \setminus \{1\}$ is equivalent to $\hat{\lambda}(\sqrt{u})=-c$ for 
$u \in (0,-\beta) \setminus \{1\}$. Thus, we consider the following quadratic equation with respect to $u$:
\begin{equation}\label{c3}
c u^2+\left(c \beta+\hat{b}_1-c\right) u-c\beta+\hat{b}_0=0.
\end{equation}
Its discriminant is given by
\[
\Delta_u(c)=(\beta+1)^2 c^2-2 \left(2 \hat{b}_0+ \hat{b}_1(1-\beta) \right) c+{\hat{b}_1}^2,
\]
which is a quadratic polynomial in $c$. 
The discriminant of $\Delta_u(c)$ with respect to $c$ is
\[
\Delta_c=16 (\hat{b}_1 + \hat{b}_0) (\hat{b}_0 - \hat{b}_1 \beta)>0.
\]
Indeed, since $\beta<-1$, either of the conditions in the
statement implies
$(\hat b_0+\hat b_1)(\hat b_0-\beta\hat b_1)>0$.
Therefore, there exists some $c \in \mathbb{R}$ such that
$\Delta_u(c)<0$, which implies that equation \eqref{c3} has no real
roots. Consequently, the equation $\hat{\lambda}(x)=-c$ has no
solutions for $x \in (0,\sqrt{-\beta}) \setminus \{1\}$.
Furthermore,  we have $(x-1) \hat{g}(x)>0$ for $x \in (0,\sqrt{-\beta}) \setminus \{1\}$ and $\hat{f}(1)=-(\hat{b}_0+\hat{b}_1)\neq 0.$
By~\cite[Lemma~2.1]{DumortierKooijLi2000}, system~\eqref{sys111} admits no
limit cycles in $(0,\sqrt{-\beta})$ under either condition in (iii).

Finally, for (iv), we show that $\hat\lambda(x)$ is monotonic on
$(0,\sqrt{-\beta}) \setminus \{1\}$. Taking the derivative of
$\hat{\lambda}(x)$ yields
$$
\hat{\lambda}^{\prime}(x)=-\frac{2x\psi(x)}{(x^2-1)^2 (x^2+\beta)^2},
$$
where $\psi(x):=\hat{b}_1 x^4+2 \hat{b}_0 x^2+\hat{b}_1\beta+\hat{b}_0 \beta-\hat{b}_0$. In the region $x \in (0, \sqrt{-\beta}) \setminus \{1\}$, the sign of $\hat{\lambda}^{\prime}(x)$ depends entirely on the sign of the polynomial $\psi(x)$. We consider the discriminant of the quadratic polynomial $\psi(\sqrt{u})$ with respect to $u$, which is given by
$$
\Delta_{\psi}=4 (\hat{b}_1+\hat{b}_0) (\hat{b}_0-\hat{b}_1\beta)<0.
$$
Here $\hat b_1<0$, so $\psi(\sqrt{u})<0$, which implies
that $\hat\lambda(x)$ is monotonically increasing on
$(0,\sqrt{-\beta})\setminus\{1\}$.

By \cite[Theorem 4.7 or 4.9, Chapter 4]{ZhangEtAl1992} or
\cite[Lemma 2.2]{DumortierKooijLi2000}, it follows that system \eqref{sys111} has at most
one limit cycle in $(0,\sqrt{-\beta})$ under the condition in (iv).
Since $\hat b_1<0$, the limit cycle is stable if it exists.
\end{proof}

\begin{lemma}\label{rotate}
For $\lambda\in\Omega$ and $x>0$, the vector field of
system \eqref{main} forms a generalized rotated vector field with respect
to the parameters $b_i$, $i=0,1$. Furthermore, the amplitude of the
stable (resp. unstable) limit cycle increases as $b_i$ decreases
(resp. increases).
\end{lemma}
	
\begin{proof}
 The vector field of system \eqref{main} is a generalized rotated vector field (see \cite[Chapter 4.3]{ZhangEtAl1992} or \cite{Perko1975}) with respect to $b_i$ when $x>0$. This is verified by computing the relevant determinant:
$$
\left|
\begin{array}{cc}
\dot{x} & \dot{y} \\[3mm]
\frac{\partial \dot{x}}{\partial b_i} &
\frac{\partial \dot{y}}{\partial b_i}\\
\end{array}
\right|
=\left|
\begin{array}{cc}
y &a_0 x+a_1 x^3 +a_2 x^5+ y(b_0 x + b_1 x^3) \\[3mm]
0& x^{2i+1}y\\
\end{array}
\right|
=x^{2i+1}y^2 \ge 0.
$$
Note that the equality holds if and only if $y=0$. Thus, by \cite[Theorem 3.5, Chapter 4]{ZhangEtAl1992}, the amplitude of the limit cycle of system \eqref{main}, when the limit cycle exists, varies monotonically with respect to the parameters $b_i$. 

Specifically, due to the positive sign of the determinant, the standard theory of generalized rotated vector fields implies that as the parameter $b_i$ increases, the vector field rotates consistently in one direction. Consequently, when the limit cycle exists, the amplitude of a stable limit cycle decreases, whereas the amplitude of an unstable limit cycle increases. This directly yields that the amplitude of the stable (resp. unstable) limit cycle increases as $b_i$ decreases (resp. increases).
\end{proof}

To state the separatrix monotonicity precisely, let
\[
(b_0^\#,b_1^\#):=
\begin{cases}
(\bar b_0,\bar b_1),&\lambda\in\Omega_8,\\
(\tilde b_0,\tilde b_1),&\lambda\in\Omega_{12}\cup\Omega_{15},
\end{cases}
\qquad
\xi:=
\begin{cases}
0,&\lambda\in\Omega_8\cup\Omega_{12},\\
\sqrt{-\beta},&\lambda\in\Omega_{15}.
\end{cases}
\]
We consider, respectively, the parameter ranges
\begin{equation}\label{eq:parameter-strips-519}
\begin{aligned}
&b_0^\#(b_0^\#+b_1^\#)<0,
&&\lambda\in\Omega_8\cup\Omega_{12},\\
&(b_0^\#+b_1^\#)(b_0^\#-\beta b_1^\#)<0,
&&\lambda\in\Omega_{15}.
\end{aligned}
\end{equation}
For $x_0\in[\xi,1)$, let $\Gamma^+(x_0)$ and
$\Gamma^-(x_0)$ be the positive and negative half-orbits of the
corresponding normalized system, namely system~\eqref{sys1} in
$\Omega_8$ and system~\eqref{sys11} in
$\Omega_{12}\cup\Omega_{15}$, starting from $(x_0,0)$. At
$x_0=\xi$, these are the relevant separatrix branches of the saddle at
the origin for $\Omega_8$ and $\Omega_{12}$, and of the saddle
$e_4^+=(\sqrt{-\beta},0)$ for $\Omega_{15}$. Whenever the two
half-orbits return transversally to the positive $x$-axis, denote the
abscissas of their first return points by $x_A$ and $x_B$, respectively.


\begin{lemma} \label{lem:monotonicity_xA_xB}
In each of the three parameter ranges in
\eqref{eq:parameter-strips-519}, suppose that the transversal returns
defining $x_A$ and $x_B$ persist on the parameter interval under
consideration. If $b_{1-i}^\#$ is fixed, $i=0,1$, then $x_A$ strictly
increases and $x_B$ strictly decreases continuously as $b_i^\#$
increases.
\end{lemma}

\begin{proof}
We present the proof for the positive half-orbit $\Gamma^+$. The proof
for the negative half-orbit $\Gamma^-$ is analogous. The three
normalized systems have the same parameter-dependent damping term, so
the following comparison applies in all three parameter ranges.

Let $y_0^+(x)$ and $y_\epsilon^+(x)$ denote the trajectories of
$\Gamma^+(x_0)$ in the upper half-plane $(y>0)$ for the parameters
$b_i^\#$ and $b_i^\#+\epsilon$, where $0<\epsilon\ll1$, respectively.
To indicate the parameter dependence, write the damping function as
$f_{b_i^\#}(x)$. Both trajectories satisfy the orbital differential
equation of the corresponding normalized system:
\begin{equation} \label{eq:orbital_ODE}
\frac{dy}{dx} = \frac{-g(x) - y f_{b_i^\#}(x)}{y} = -\frac{g(x)}{y} - f_{b_i^\#}(x).
\end{equation}

We define the vertical variation between the two orbits as $z^+(x) = y_\epsilon^+(x) - y_0^+(x)$ for $x_0 \le x \le \min\{x_A(0), x_A(\epsilon)\}$. Since both orbits emanate from $(x_0, 0)$, we have $z^+(x_0) = 0$.

Using \eqref{eq:orbital_ODE}, the derivative of the variation is given by:
\begin{equation}
\label{eq:dz_dx}
\frac{dz^+}{dx}
= g(x) \frac{z^+(x)}{y_\epsilon^+(x) y_0^+(x)}
- \big( f_{b_i^\#+\epsilon}(x) - f_{b_i^\#}(x) \big).
\end{equation}

To preserve the generality of the perturbation, we do not explicitly expand $f_{b_i^\#}(x)$ initially. Instead, we introduce the following auxiliary functions:
\begin{equation}
\psi_1(x) = -\int_{x_0}^x \big( f_{b_i^\#+\epsilon}(s) - f_{b_i^\#}(s) \big) ds, \quad \text{and} \quad \psi_2(x) = \frac{g(x)}{y_\epsilon^+(x) y_0^+(x)}.
\end{equation}
Notice that $\psi_1^{\prime}(x) = - \big( f_{b_i^\#+\epsilon}(x) - f_{b_i^\#}(x) \big)$. From \eqref{eq:dz_dx}, we can rewrite the variation $z^+(x)$ in the integral form:
\begin{equation} \label{eq:z_integral_form}
z^+(x) = \psi_1(x) + \int_{x_0}^x \psi_2(\tau) z^+(\tau) d\tau.
\end{equation}

Let $\psi_3(x) = \int_{x_0}^x \psi_2(\tau) z^+(\tau) d\tau$. Then $z^+(x) = \psi_1(x) + \psi_3(x)$, and differentiating $\psi_3(x)$ yields:
\begin{equation}
\frac{d\psi_3(x)}{dx} = \psi_2(x) z^+(x) = \psi_2(x) \big( \psi_1(x) + \psi_3(x) \big).
\end{equation}
Rearranging this gives a standard first-order linear ordinary differential equation for $\psi_3(x)$:
\begin{equation} \label{eq:psi3_ODE}
\frac{d\psi_3(x)}{dx} - \psi_2(x) \psi_3(x) = \psi_1(x) \psi_2(x).
\end{equation}

Multiplying both sides by the integrating factor $\exp\left(-\int_{x_0}^x \psi_2(s) ds\right)$ and integrating from $x_0$ to $x$ (noting that $\psi_3(x_0) = 0$), we obtain:
\begin{equation}
\psi_3(x) = \int_{x_0}^x \psi_1(\tau) \psi_2(\tau) \exp\left( \int_\tau^x \psi_2(s) ds \right) d\tau.
\end{equation}

Substituting this back into $z^+(x) = \psi_1(x) + \psi_3(x)$, we apply integration by parts on the integral term by setting $u = \psi_1(\tau)$ and $dv = \psi_2(\tau) \exp\left(\int_\tau^x \psi_2(s) ds\right) d\tau$. This simplifies the expression to:
\begin{equation} \label{eq:z_final}
z^+(x) = \psi_1(x_0) \exp\left( \int_{x_0}^x \psi_2(s) ds \right) + \int_{x_0}^x \psi_1^{\prime}(\tau) \exp\left( \int_\tau^x \psi_2(s) ds \right) d\tau.
\end{equation}
Since $\psi_1(x_0) = 0$, the first term vanishes completely. Now,
recalling that $f(x)=-(b_0^\#x+b_1^\#x^3)$, the perturbation term is
exactly $f_{b_i^\#+\epsilon}(\tau)-f_{b_i^\#}(\tau)
=-\epsilon\tau^{2i+1}$. Therefore:
\begin{equation}
\psi_1^{\prime}(\tau) = - \big( f_{b_i^\#+\epsilon}(\tau) - f_{b_i^\#}(\tau) \big) = \epsilon \tau^{2i+1}.
\end{equation}

For $\epsilon > 0$ and $\tau > x_0 \ge 0$, we have $\psi_1^{\prime}(\tau) > 0$. Furthermore, the exponential factor is strictly positive. Consequently, the integral in \eqref{eq:z_final} is strictly positive, yielding:
\begin{equation}
z^+(x) = y_\epsilon^+(x) - y_0^+(x) > 0.
\end{equation}
Thus the perturbed orbit $y_\epsilon^+$ lies strictly above the unperturbed orbit $y_0^+$ throughout their common domain.  This ordering of the two transversal returns implies that $x_A(\epsilon)>x_A(0)$, so $x_A$ strictly increases as $b_i^\#$ increases.

The same argument applies to the negative half-orbit $\Gamma^-$.  The variation $z^-(x)=y_\epsilon^-(x)-y_0^-(x)$ has the same integral representation, and hence $z^-(x)>0$.  Since both orbits lie in the lower half-plane, $y_\epsilon^-(x)>y_0^-(x)$ means that $y_\epsilon^-$ lies closer to the $x$-axis and returns to it first.  Therefore, $x_B(\epsilon)<x_B(0)$, so $x_B$ strictly decreases as $b_i^\#$ increases.

The continuous dependence of $x_A$ and $x_B$ on the parameter $b_i^\#$
follows from the smooth dependence of solutions on initial conditions
and parameters, together with the transversality of the returns.
\end{proof}

To extend the return separation beyond ordinary intersections with the
positive $x$-axis, fix one of the three parameter ranges in
\eqref{eq:parameter-strips-519}.  Let $\mu=b_i^\#$ vary in an interval,
with $b_{1-i}^\#$ and the remaining coefficients fixed, and let
$p(\mu)=(\xi,0)$ be the corresponding saddle.  Denote by
$\Gamma_u(\mu)$ the relevant unstable branch followed forward and by
$\Gamma_s(\mu)$ the relevant stable branch followed backward; both belong
to the same saddle $p(\mu)$.

In the compactified half-disc, choose disjoint isolating neighborhoods of
the possible finite and infinite terminal sets and, when present, of the
unique limit cycle.  Their transverse entrance sections, together with
the transverse portions of the positive $x$-axis away from equilibria,
form an ordered transversal chain $\Sigma$.  Let $q_u(\mu)$ and
$q_s(\mu)$ be the first intersections of $\Gamma_u(\mu)$ and
$\Gamma_s(\mu)$ with $\Sigma$.  Choose an injective order coordinate
$\sigma$ on $\Sigma$, agreeing with the $x$-coordinate on the
positive-$x$-axis component, and set
\[
D(\mu):=\sigma(q_u(\mu))-\sigma(q_s(\mu)).
\]
Thus $D(\mu)=x_A(\mu)-x_B(\mu)$ whenever both branches return
transversally to the positive $x$-axis.

\begin{lemma}[Extended separatrix shooting]\label{lem:extended-shooting}
The function $D(\mu)$ is continuous and strictly increasing
in $\mu$.  Moreover, $D(\mu)=0$ if and only if $\Gamma_u(\mu)$ and
$\Gamma_s(\mu)$ form a homoclinic orbit.  
\end{lemma}

\begin{proof}
The preceding equilibrium, limit-cycle, and infinity classifications
ensure that every branch not forming the required connection meets
$\Sigma$.  Continuous dependence and transversality, together with the
matching of the order coordinate at adjacent components of $\Sigma$, show
that $D$ is continuous.  The strict rotated-field ordering in
Lemma~\ref{rotate}, equivalently the comparison in
Lemma~\ref{lem:monotonicity_xA_xB}, implies that $\sigma(q_u)$ strictly
increases and $\sigma(q_s)$ strictly decreases with $\mu$.  Hence $D$ is
strictly increasing.

Since $\sigma$ is injective and $\Sigma$ contains only regular points,
$D(\mu)=0$ if and only if $q_u(\mu)=q_s(\mu)$.  By uniqueness of
solutions, this occurs exactly when the stable and unstable branches of
$p(\mu)$ form a homoclinic orbit.  
\end{proof}

\begin{remark}
By the continuity and strict monotonicity of $D(\mu)$, together with its
continuous dependence on the remaining coefficients, opposite signs of
$D$ at two parameter values determine a unique homoclinic parameter.  This
parameter depends continuously on the remaining coefficients as long as
the setup of Lemma~\ref{lem:extended-shooting} persists.
\end{remark}

\begin{lemma}\label{lem:homoclinic-curves}
For $b_1<0$, there are three unique continuous critical functions whose
graphs define the homoclinic-bifurcation hypersurfaces
\[
\begin{aligned}
\mbox{HL}_1:&\quad b_0=\varphi_1(a_1,a_2,b_1),
&&\lambda\in\Omega_8,\\
\mbox{HL}_2:&\quad b_0=\varphi_2(a_0,a_1,a_2,b_1),
&&\lambda\in\Omega_{12},\\
\mbox{HL}_4:&\quad b_0=\varphi_4(a_0,a_1,a_2,b_1),
&&\lambda\in\Omega_{15}.
\end{aligned}
\]
On $\mbox{HL}_1$ and $\mbox{HL}_2$, the corresponding separatrix
configurations are figure-eight homoclinic loops based at the origin.
On $\mbox{HL}_4$, the right-hand
homoclinic loop is based at $e_4^+=(\sqrt{-\beta},0)$, and its
reversible image is based at $e_4^-=(-\sqrt{-\beta},0)$; hence this
configuration is not a figure-eight loop. Each of the three hypersurfaces
separates the parameter region with exactly one limit cycle from the
parameter region with no limit cycles.
\end{lemma}

\begin{proof}
For every normalized system put
\[
\mathcal H(x,y):=\frac{y^2}{2}+\int_0^x g(s)\,ds,
\qquad
\dot{\mathcal H}=x(b_0^\#+b_1^\#x^2)y^2.
\]
Fix $b_1^\#<0$.  In $\Omega_8$, the small limit cycle on the cycle side
of the Hopf line $b_0^\#=-b_1^\#$ separates the two relevant branches
and fixes one strict exit order.  At $b_0^\#=0$,
$\dot{\mathcal H}=b_1^\#x^3y^2<0$ for $x>0$; hence the forward unstable
branch lies below the saddle energy and the backward stable branch above
it until their respective exits.  Lemmas~\ref{lem11} and~\ref{propIn}
exclude a periodic or unclassified terminal set, so this is the opposite
exit order.  Lemma~\ref{lem:extended-shooting} gives the unique zero
$b_0^\#=\bar\varphi_1(b_1^\#)$.  It is the right homoclinic loop at $O$,
and reversibility supplies the left loop, giving a figure-eight.

In $\Omega_{12}$ the saddle is again $O$.  The Hopf cycle fixes the first
exit order, whereas on $b_0^\#=0$ the same strict identity
$\dot{\mathcal H}=b_1^\#x^3y^2<0$ places the two branches on opposite
sides of the saddle level.  The finite equilibria classified above,
together with Lemmas~\ref{lem11} and~\ref{propIn},
leave only the corresponding opposite exit order.  The extended shooting
lemma therefore gives the unique $\mbox{HL}_2$, and reversibility again
turns its right loop into a figure-eight.

In $\Omega_{15}$ take $\xi=\sqrt{-\beta}$ and vary $b_0^\#$ between
$-b_1^\#$ and $\beta b_1^\#$.  Near the former line the Hopf cycle fixes
one exit order.  On the latter line,
\[
\dot{\mathcal H}=b_1^\#x(x^2+\beta)y^2<0,
\qquad x>\sqrt{-\beta},
\]
so the two branches have the opposite exit order, with all other exits
again covered by Lemma~\ref{propIn}.  The unique zero furnished by
Lemma~\ref{lem:extended-shooting} is the right homoclinic loop based at
$e_4^+$; its reversible image is based at $e_4^-$, so the configuration
is not a figure-eight.  In all three cases, the limit-cycle uniqueness
proved above and strict rotation show that the zero separates the
one-cycle and no-cycle regions.  Continuous dependence in the extended
shooting lemma and the inverse scalings give the continuous functions
$\varphi_1,\varphi_2,$ and $\varphi_4$ stated above.
\end{proof}

In what follows, fix $\beta<-1$ and put $R:=\sqrt{-\beta}$.

\begin{lemma}\label{lem:limit-cycle-breaking-phi3}
There exists a unique continuous function $\hat\varphi_3$ on
$\hat b_1<0$, with
$-\hat b_1<\hat\varphi_3(\hat b_1)<\beta\hat b_1$, such that
system~\eqref{sys111} has its unique stable limit cycle exactly when
$-\hat b_1<\hat b_0<\hat\varphi_3(\hat b_1)$.  At the upper endpoint,
the limit-cycle family terminates at a closed separatrix boundary.  Its
graph $\mbox{HL}_3:\hat b_0=\hat\varphi_3(\hat b_1)$ is the normalized
limit-cycle-breaking curve; the inverse scaling gives the hypersurface
$b_0=\varphi_3(a_0,a_1,a_2,b_1)$ in the original parameter space.
\end{lemma}

\begin{proof}
The Hopf cycle exists for $\hat b_0>-\hat b_1$ close to
$-\hat b_1$.  Every such cycle has its two $x$-axis extrema in
$0<x<R$.  On either half-orbit, with $w=y^2$,
the orbital equation gives
\[
w'=2x(x^2-1)(x^2+\beta)+2x(\hat b_0+\hat b_1x^2)y.
\]
Uniformly for $\hat b_0\in[-\hat b_1,\beta\hat b_1]$, this implies
$|w'|\leq C(1+w)$.  Since $w=0$ at an $x$-axis extremum and the
$x$-interval has length at most $R$, Gronwall's inequality confines
the whole Hopf family to one fixed compact rectangle.  Thus neither
its amplitude nor its closure can escape to infinity.  Moreover,
strict rotation makes cycles at two parameter values bound a trapping
annulus for every intermediate value; the Poincar\'e--Bendixson theorem
and the one-cycle result therefore make the existence set one interval
and exclude a fold.

Let $\hat b_0^*$ be its upper endpoint.  A sequence of cycles tending to
$\hat b_0^*$ has a compact invariant limit.  A regular limiting cycle
would continue past $\hat b_0^*$, while collapse to $e_3^+$ is possible
only at $\hat b_0=-\hat b_1$.  Hence the Poincar\'e--Bendixson theorem
and the equilibrium classification make the limit a closed separatrix
boundary, homoclinic or heteroclinic, formed by the saddles $O$ and
$e_4^+$.  At
$\hat b_0=\beta\hat b_1$,
\[
\dot{\mathcal H}=\hat b_1x(x^2+\beta)y^2>0,
\qquad 0<x<R,\quad y\ne0,
\]
so no closed orbit or closed separatrix boundary can occur there.
Hence $-\hat b_1<\hat b_0^*<\beta\hat b_1$.  The interval property gives
the uniqueness of $\hat b_0^*$; the same uniform compactness and
uniqueness, together with continuous dependence, give its continuity in
the remaining coefficients.  Thus
$\hat b_0^*=\hat\varphi_3(\hat b_1)$, and
the inverse scaling gives the stated $\varphi_3$.
\end{proof}

\begin{remark}
Depending on the relative ordering of the separatrices of
the origin and of $e_4^+$, the critical configuration on
$\mbox{HL}_3$ may be homoclinic or heteroclinic. This distinction is
discussed in detail in Section~\ref{s7:omega14}.
\end{remark}

\section{Proof of the bifurcation and phase-portrait classifications in
Theorems~\ref{thm-bifurcation} and~\ref{thm-phase-portraits}}\label{s7}

The proof uses the preceding sections in four stages.
Sections~\ref{sec:preliminary-results}
and~\ref{sec:general-theorem-proofs}
establish the general periodic-annulus criteria and the critical-connection
results at infinity, which will be used especially for $\mathrm{PA}_1$ and
$\mathrm{PA}_2$.  Section~\ref{sec:local-analysis} proves items~(i)--(viii) of
Theorem~\ref{thm-bifurcation}: Lemma~\ref{prop1} and the local
classification of the origin give the pitchfork, generalized pitchfork,
and nilpotent center--sector bifurcations; Lemma~\ref{propE2} gives the
saddle-node and Bogdanov--Takens bifurcations; Lemmas~\ref{propE1} and
\ref{propE3} give the Hopf bifurcations and, in their Hamiltonian
subcases, the Hamiltonian-center degeneracies; and Lemma~\ref{propIn}
gives the bifurcation at infinity.  Section~\ref{s5} supplies the
nonexistence, uniqueness, and stability results for limit cycles, while
Lemmas~\ref{lem:homoclinic-curves} and
\ref{lem:limit-cycle-breaking-phi3} construct the critical homoclinic and
limit-cycle-breaking functions.  Finally, the present section combines
these results with separatrix shooting and comparison, the generalized
rotated-vector-field property, reversibility, and the
Poincar\'e--Bendixson and index arguments to determine every remaining
global connection.  It thereby proves the global items of
Theorem~\ref{thm-bifurcation} and, at the same time, proves
Theorem~\ref{thm-phase-portraits} by showing that the listed parameter
regions are mutually exhaustive and have precisely the phase portraits
shown in Figures~\ref{phasea0n0-one}--\ref{phasea0n0-five}.

More specifically, the periodic-annulus-breaking bifurcations in
item~(ix) of Theorem~\ref{thm-bifurcation} are obtained as follows.
The general results of Sections~\ref{sec:preliminary-results}
and~\ref{sec:general-theorem-proofs}, together with
Proposition~\ref{prop:separatrix-infinity} as applied in
Subsections~\ref{s7:omega2} and~\ref{s7:omega3}, give
$\mathrm{PA}_1$ and $\mathrm{PA}_2$; the same separatrix-comparison
argument gives $\mathrm{PA}_3$ in Subsection~\ref{s7:omega8};
Proposition~\ref{prop:omega9-positive} gives $\mathrm{PA}_4$ in
Subsection~\ref{s7:omega9}; and
Propositions~\ref{prop:omega11-PA5}, \ref{prop:omega12-PA6}, and
\ref{prop:omega15-PA7} give $\mathrm{PA}_5$, $\mathrm{PA}_6$, and
$\mathrm{PA}_7$ in Subsections~\ref{s7:omega11},
\ref{s7:omega12}, and~\ref{s7:omega15}, respectively.  For item~(x),
the existence and uniqueness of $\mathrm{HL}_1$, $\mathrm{HL}_2$, and
$\mathrm{HL}_4$ come first from Lemma~\ref{lem:homoclinic-curves}, and
their global separatrix configurations are identified in
Subsections~\ref{s7:omega8}, \ref{s7:omega12}, and~\ref{s7:omega15};
$\mathrm{HL}_3$ is established in
Lemma~\ref{lem:limit-cycle-breaking-phi3}, and its homoclinic and
heteroclinic critical configurations are determined in
Subsection~\ref{s7:omega14}.  The heteroclinic bifurcations in item~(xi)
are proved in Subsection~\ref{s7:omega10}, principally by
Proposition~\ref{prop:omega10-E2O} and the analogous construction below
$C_1$,
and in Subsection~\ref{s7:omega14}, principally by
Propositions~\ref{prop:omega14-connections} and
\ref{prop:omega14-ordering}; the Hamiltonian boundary analysis there
also gives $HHR$.

For reference, the case-by-case proof of
Theorem~\ref{thm-phase-portraits} is indexed by the parameter partition as
follows: $\Omega_1$ and $\Omega_6$ are treated together in
Subsection~\ref{s7:omega1}; $\Omega_2$, $\Omega_3$, $\Omega_8$--$\Omega_{12}$,
$\Omega_{14}$, and $\Omega_{15}$ are treated in
Subsections~\ref{s7:omega2}, \ref{s7:omega3}, \ref{s7:omega8}--
\ref{s7:omega12}, \ref{s7:omega14}, and~\ref{s7:omega15},
respectively; $\Omega_4$ and $\Omega_5$ are treated together in
Subsection~\ref{s7:omega4}; and $\Omega_7$ and $\Omega_{13}$ are treated
in Subsections~\ref{s7:omega7} and~\ref{s7:omega13}.  In each case the
proof starts from the local equilibrium configuration, uses the
Section~\ref{s5} limit-cycle bounds to restrict the recurrent dynamics,
determines the admissible separatrix endpoints, and then invokes
reversibility to recover the left half-plane.  Any additional phase
portrait would require a further admissible separatrix endpoint or
ordering, or an additional isolated periodic orbit; these possibilities
are excluded, respectively, by Lemma~\ref{propIn}, the strict separatrix
ordering, and the index and limit-cycle results.  It is therefore enough
below to describe the dynamics in the half-plane $x>0$ and the global
connections that can bound a periodic annulus or terminate a limit cycle.

We first treat the part of the parameter space in which the origin is the
unique finite equilibrium.  We shall repeatedly use the following
observation.  Once the local type of
the origin and the existence or nonexistence of the large-scale periodic
annulus have been fixed, Lemma~\ref{propIn} determines all the
remaining separatrix connections.  Indeed, there is no other finite
equilibrium at which a separatrix can end.  Moreover, a periodic orbit
contained in either open half-plane is excluded by the index theorem, and a
periodic orbit meeting the symmetry axis belongs, by
Lemma~\ref{lem:reversible-periodic-orbit}, to a periodic annulus and hence
cannot be an isolated limit cycle.  Consequently,
the alternatives listed below exhaust all possible global configurations on
the Poincar\'e disc.

\subsection{The cases $\Omega_1$ and $\Omega_6$}\label{s7:omega1}

Let $\lambda\in\Omega_1$.  By Lemma~\ref{prop1}, the origin is the
unique finite equilibrium.  Since $a_0=0$, $a_1\geq0$ and $a_2>0$, the local
classification of the origin shows that it is a nilpotent saddle (including
the more degenerate case $a_1=0$).  Lemma~\ref{lem12} excludes limit cycles
in either half-plane.  Finally, the local portrait at infinity is the one
corresponding to $a_2>0$ in Lemma~\ref{propIn}.  The two branches of
each saddle separatrix are therefore forced to connect with the prescribed
separatrices at infinity; no finite polycycle or periodic annulus can occur.
This gives exactly Figure~\ref{phasea0n0-one}(1).

The case $\Omega_6$ follows in exactly the same way: since
$a_0,a_1,a_2>0$, the origin is the unique finite equilibrium and is a
hyperbolic saddle, while Lemma~\ref{propIn} gives the same local
configuration at infinity.  Hence $\Omega_6$ has the same global phase
portrait as $\Omega_1$.

\subsection{The case $\Omega_2$}\label{s7:omega2}

Assume $a_0=a_1=0$ and $a_2<0$.  If $b_1<0<b_0$, then the origin is a nilpotent equilibrium whose local phase portrait consists of one elliptic sector and one hyperbolic sector, and Proposition~\ref{prop_cdf1} implies that system~\eqref{main} always has a large-scale periodic annulus.  
If $b_0=0$ and $b_1<0$, the origin is a nilpotent center; together with Lemma~\ref{propIn}, this implies the existence of a large-scale periodic annulus.
When $b_0=b_1=0$, the system is Hamiltonian, with $H(x,y)=\frac{y^2}{2}-\frac{a_2x^6}{6}$, so the conclusion follows directly from its level curves.
For $b_0<0$ and $b_1=0$, the origin is a nilpotent equilibrium whose local
phase portrait consists of one elliptic sector and one hyperbolic sector,
and the $n=s$ boundary case in Proposition~\ref{prop:n-equals-s}, rather than Theorem~\ref{thm_survival} (which assumes $n>s$), shows that the large-scale periodic annulus always exists.

Finally, let $b_0<0$ and $b_1<0$.  Under the transformation
$(x,y,t)\to(\alpha x,\beta y,\gamma t)$, system~\eqref{main} becomes
\begin{equation}\label{sysk-omega2}
\begin{aligned}
\dot{x} &= y, \\
\dot{y} &= -x^5-y(x+Kx^3).
\end{aligned}
\end{equation}
Here $\alpha=-\frac{b_0}{\sqrt{-a_2}}$,
$\beta=\frac{b_0^3}{a_2}$,
$\gamma=\frac{\sqrt{-a_2}}{b_0^2}$, and
$K=-\frac{b_0b_1}{a_2}>0$.
For system~\eqref{sysk-omega2}, one has $s=1$, $n=3$, and $m=5$, so $2s+1<m\leq2n+1$; all the hypotheses of Theorem~\ref{thm_bifurcation} are satisfied. Hence the normalized family has a unique critical value $K_1^*>0$: it possesses a large-scale periodic annulus for $0\leq K<K_1^*$, whereas the annulus is destroyed for $K\geq K_1^*$.
Since $K>0$ in the present case, the relevant existence range is $0<K<K_1^*$; the endpoint $K=0$ corresponds precisely to the previously considered case $b_1=0$.
Numerical computations and the bisection method give $K_1^*\in(1.998,2]$, suggesting that $K_1^*=2$.
This value is not accidental. The following proposition gives an exact separatrix criterion in the original parameters and, when applied to $a_1=0$, rigorously proves that $K_1^*=2$.

For the statement of the exact criterion, assume $a_0=0$, $a_1\leq0$, $a_2<0$, and $b_1<0$, and suppose that
\[
b_0^2+8a_1\geq0,
\qquad
c_2:=\frac{-b_0+\sqrt{b_0^2+8a_1}}{4}>0.
\]
Let $\Gamma_s$ be the boundary orbit in the fourth quadrant separating the elliptic and hyperbolic sectors at the origin. Its local graph is
$y_s(x)=-c_2x^2+o(x^2)$.
In what follows, $y_s(x)$ denotes only the part of $\Gamma_s$ lying in the fourth quadrant, up to its first possible intersection with $y=0$.

\begin{prop}\label{prop:separatrix-infinity}
Under the preceding assumptions, define
$b_0^{\rm PA}:=\frac{a_1b_1}{a_2}-2\frac{a_2}{b_1}$.
Assume that $c=a_2/b_1$ is the larger root of the quadratic equation $2c^2+b_0^{\rm PA}c-a_1=0$ in the unknown $c$.  Then the following statements hold:
\begin{itemize}
\item[(i)] If $b_0>b_0^{\rm PA}$, $\Gamma_s$ leaves the fourth quadrant by intersecting the positive $x$-axis transversally.
\item[(ii)] If $b_0=b_0^{\rm PA}$, then
$y_s(x)\equiv-\frac{a_2}{b_1}x^2,$
and $\Gamma_s$ connects the origin with the infinite equilibrium on the negative $y$-axis.
\item[(iii)] If $b_0<b_0^{\rm PA}$, $\Gamma_s$ remains in the fourth quadrant, does not intersect the positive $x$-axis, and connects the origin with the infinite equilibrium on the negative $y$-axis.
\end{itemize}
\end{prop}

\begin{proof}
In the fourth quadrant the trajectories satisfy $dy/dx=(a_1x^3+a_2x^5)/y+b_0x+b_1x^3$. 
Substitution of $y_s(x)=-c_2x^2+o(x^2)$ gives $2c_2^2+b_0c_2-a_1=0$. Thus $c_2$ is the larger local root for the current value of $b_0$; at $b_0=b_0^{\rm PA}$, the larger-root assumption in the statement identifies it with $C_*=a_2/b_1$.
For the comparison curve $y=-Cx^2$, direct calculation gives
\begin{equation}\label{eq:general-separatrix-barrier}
\left.\frac{dy}{dx}\right|_{y=-Cx^2}-(-2Cx)
=\left(b_0-\frac{a_1}{C}+2C\right)x
+\left(b_1-\frac{a_2}{C}\right)x^3.
\end{equation}

Set $C_*:=a_2/b_1>0$.  At $b_0=b_0^{\rm PA}$, direct substitution gives $2C_*^2+b_0^{\rm PA}C_*-a_1=0$, and the larger-root assumption implies $c_2=C_*$. 
Moreover, the explicit formula for $c_2$ shows that it is strictly decreasing with respect to $b_0$ on the admissible parameter branch. 
Hence $b_0>b_0^{\rm PA}$ implies $c_2<C_*$, whereas $b_0<b_0^{\rm PA}$ implies $c_2>C_*$.

Assume first that $b_0>b_0^{\rm PA}$ and choose $c_2<C<C_*$.
Since $C>c_2$, the first coefficient on the right-hand side of \eqref{eq:general-separatrix-barrier} is positive; since $C<C_*$, the second coefficient is also positive. 
Near the origin, $C>c_2$ gives $y_s(x)>-Cx^2$. 
Thus $y=-Cx^2$ is a strict lower barrier, and $y_s(x)>-Cx^2$ persists as long as $y_s(x)<0$.

The horizontal isocline is $y_{\rm iso}(x)=-(a_1x^3+a_2x^5)/(b_0x+b_1x^3)$ and satisfies $y_{\rm iso}(x)\sim-C_*x^2$ as $x\to+\infty$.  Since $C<C_*$, the lower barrier, and consequently $y_s$, eventually lies above the isocline.  
In fact, the same degree comparison gives $y_s^{\prime}(x)\geq\delta x^3$ for some $\delta>0$ and all sufficiently large $x$, as long as $y_s(x)<0$.
Therefore, when traced backward in time, $\Gamma_s$ reaches the positive $x$-axis at a finite point. 
The intersection is transversal: otherwise $a_1X^3+a_2X^5=0$ at the intersection point $(X,0)$, making it a finite equilibrium, which a non-constant orbit cannot reach in finite time.

Now set $C=C_*$.  The coefficient of $x^3$ in
\eqref{eq:general-separatrix-barrier} vanishes, while the coefficient of $x$
is precisely $b_0-b_0^{\rm PA}$.  If $b_0<b_0^{\rm PA}$, then
$c_2>C_*$, so $y_s(x)<-C_*x^2$ near the origin.  Equation
\eqref{eq:general-separatrix-barrier} shows that $y=-C_*x^2$ is a strict
upper barrier; hence this inequality persists and $y_s$ never intersects
the positive $x$-axis.  
Its graph is defined for every $x>0$: on each bounded $x$-interval the right-hand side of the orbital equation remains bounded when $|y|$ is large, so escape to $y=-\infty$ at a finite value of $x$ is impossible.
Since $y_s(x)<-C_*x^2$, we have $y_s(x)/x\to-\infty$ as $x\to+\infty$. Hence, $\Gamma_s$ approaches the negative $y$-axis direction at infinity.
Lemma~\ref{propIn} then identifies its endpoint with the infinite
equilibrium on the negative $y$-axis.

Finally, if $b_0=b_0^{\rm PA}$, both coefficients in
\eqref{eq:general-separatrix-barrier} vanish.  Hence
$y=-(a_2/b_1)x^2$ is invariant.  By the assumption that $a_2/b_1=c_2$ and
by uniqueness of the separatrix, this invariant curve is exactly $y_s$ and
connects to the same infinite equilibrium.  This proves the proposition.
\end{proof}

\begin{remark}\label{rem:larger-root}
The larger-root assumption in Proposition~\ref{prop:separatrix-infinity}
is automatic for $a_1=0$. When $a_1<0$, it is equivalent to $b_1\in\left[a_2\sqrt{-\frac{2}{a_1}},0\right)$.
\end{remark}

\begin{remark}
Statement~(iii) of Proposition~\ref{prop:separatrix-infinity} can also be
obtained directly from statement~(ii) and the theory of generalized rotated
vector fields; see Lemma~\ref{rotate}.  Indeed, for fixed $a_1$, $a_2$ and
$b_1$, decreasing $b_0$ from $b_0^{\rm PA}$ moves $\Gamma_s$ strictly below
the critical invariant curve $y=-(a_2/b_1)x^2$.  Hence $\Gamma_s$ cannot
intersect the positive $x$-axis, and Lemma~\ref{propIn} forces it to
connect to the infinite equilibrium on the negative $y$-axis.
\end{remark}

For $\Omega_2$ we have $a_1=0$, so Proposition~\ref{prop:separatrix-infinity}
gives
$b_0^{\rm PA}=-2\frac{a_2}{b_1}$.
By Lemma~\ref{lem:geometric-criterion}, the large-scale periodic annulus exists exactly when the boundary separatrix crosses the positive $x$-axis; hence this identity is equivalent to $K_1^*=2$.
Since the normalized family~\eqref{sysk-omega2} has $m=5$ and $n=3$, it satisfies $n+1<m<2n+1$.
Theorem~\ref{thm:critical-connections} further shows that at
$K=K_1^*=2$ the connection from the origin to the equilibrium at infinity
in the negative $y$-direction is unique, whereas for $K>2$ there are
infinitely many such orbit branches in the fourth quadrant.
These alternatives give all the $\Omega_2$ portraits in
Figure~\ref{phasea0n0-one}, including the critical portrait $\mbox{PA}_1$.

\subsection{The case $\Omega_3$}\label{s7:omega3}

Assume $a_0=0$, $a_1<0$ and $a_2<0$.  If $b_0^2+8a_1<0$, the origin is a nilpotent center; together with Lemma~\ref{propIn}, this implies the existence of a large-scale
periodic annulus. 
If $b_0^2+8a_1\geq0$ and $b_0>0$, Proposition~\ref{prop_cdf1} implies that the system has a large-scale periodic annulus.  
If $b_1=0$, the conclusions for $b_0^2+8a_1<0$ and $b_0^2+8a_1\geq0$ follow, respectively, from the preceding discussion and the $n=s$ boundary case in Proposition~\ref{prop:n-equals-s}; Theorem~\ref{thm_survival} itself does not apply because it assumes $n>s$.
Thus it remains only to consider $b_0<0$, $b_1<0$ and $b_0^2+8a_1\geq0$.
For these parameters, write $g(x)=-a_1x^3-a_2x^5$ and $f(x;K)=(-b_0)x+Kx^3$, with $K=-b_1>0$. Then $s=1$, $n=3$, $k=3$, and $m=5$, and hence $2s+1<m\leq2n+1$. Thus Theorem~\ref{thm_bifurcation} applies and gives a unique critical value of the higher-order damping parameter separating the existence and nonexistence of the large-scale periodic annulus. To describe the boundary in the original parameters, let $b_1^*:=a_2\sqrt{-\frac{2}{a_1}}$ and consider the curve
$b_0=b_0^{\rm PA}:=\frac{a_1b_1}{a_2}-2\frac{a_2}{b_1}$.
In the region $b_1<0$, $b_0<0$, this curve has two monotone branches.
Indeed, 
$\frac{d(b_0^{\rm PA})}{d b_1}=\frac{a_1}{a_2}+\frac{2a_2}{b_1^2}$,
so the curve is strictly increasing for $b_1<b_1^*$ and strictly decreasing for
$b_1\in(b_1^*,0)$. 
The two branches meet at the unique turning point $B^-=\left(-\sqrt{-8a_1},a_2\sqrt{-\frac{2}{a_1}}\right)$.

By Remark~\ref{rem:larger-root}, $c=a_2/b_1$ is the larger local root
precisely for $b_1\in[b_1^*,0)$.  
Hence, only the right decreasing branch $b_1\in[b_1^*,0)$ corresponds to the bifurcation boundary described in Proposition \ref{prop:separatrix-infinity}.
For $b_1\in(b_1^*,0)$,
Proposition~\ref{prop:separatrix-infinity} shows that the separatrix crosses
the positive $x$-axis for $b_0>b_0^{\rm PA}$, whereas it connects to the
infinite equilibrium on the negative $y$-axis for
$b_0\leq b_0^{\rm PA}$; the left endpoint of this branch is $B^-$.  At
$B^-$, the two local roots coalesce, and $\Gamma_s$ coincides with the
invariant curve $y=-(a_2/b_1)x^2$ connecting the origin to the infinite
equilibrium on the negative $y$-axis.  On the
left branch $b_1<b_1^*$, the invariant algebraic curve corresponds to the
smaller local root and is not the boundary separatrix.  By the strict
monotonicity and uniqueness in Theorem~\ref{thm_bifurcation}, applied for fixed $a_1,a_2,b_0$ with $K=-b_1$, every point reached after crossing the right decreasing branch, including the left branch $b_1<b_1^*$, lies on the $K>K^*$ side of the unique bifurcation threshold, where the large-scale periodic annulus no longer exists.
Hence the boundary separatrix connects to the infinite equilibrium on the negative $y$-axis throughout this postcritical part of the region $b_0^2+8a_1\geq0$.

Here the corresponding normalized family again has $m=5$ and $n=3$, and hence $n+1<m<2n+1$.
Theorem~\ref{thm:critical-connections} therefore shows that the connection on the right decreasing branch $b_1\in(b_1^*,0)$ of $b_0^{\rm PA}$ is unique, whereas on the postcritical side, obtained by increasing the normalized higher-order damping parameter, infinitely many orbit branches in the fourth quadrant connect the origin with the equilibrium at infinity in the negative $y$-direction.

Combining these conclusions with the center cases and the cases covered by Propositions~\ref{prop_cdf1} and~\ref{prop:n-equals-s} gives all the
$\Omega_3$ portraits in Figure~\ref{phasea0n0-one}, including $\mbox{NE}^{+}$, $\mbox{NE}_{1}^{-}$, $\mbox{NE}_{2}^{-}$, $\mbox{NE}_{0}^{-}$, $\mbox{PA}_{2}$ and $B^-$.

\subsection{The cases $\Omega_4$ and $\Omega_5$}\label{s7:omega4}

If $\lambda\in\Omega_4$, then
$a_1^2-4a_0a_2<0$, so $a_0a_2>0$.  The origin is the unique finite
equilibrium.  If $a_2>0$, then $a_0>0$ and the origin is a hyperbolic saddle.
Therefore, by the same argument as in Subsection~\ref{s7:omega1}, one obtains
Figure~\ref{phasea0n0-one}(1), namely the region $I_{13}$.

If $a_2<0$, then $a_0<0$ and the origin is a linear center.  There is no
other finite equilibrium, and reversibility prevents an isolated periodic
orbit from forming the outer boundary of its periodic annulus.  If $b_1<0$,
the maximal periodic annulus extends to the separatrix configuration at
infinity and forms a large-scale periodic annulus, giving $I_{14}$.  If
$b_1=0$, the maximal periodic annulus fills the finite plane, so the origin is
a global center, giving $I_{15}$.

The case $\Omega_5$ is analogous.  Indeed, for $\lambda\in\Omega_5$ we have
$a_1^2=4a_0a_2$ and $a_0a_1>0$.  It follows that $a_0$, $a_1$ and $a_2$
have the same sign, while Lemma~\ref{prop1} again gives no nonzero
finite equilibrium.  If $a_2>0$, the portrait $I_{16}$ follows directly from
Subsection~\ref{s7:omega1}; if $a_2<0$, the same argument as for $\Omega_4$
gives the large-scale periodic-annulus portrait $I_{17}$ when $b_1<0$ and
the global-center portrait $I_{18}$ when $b_1=0$.

\subsection{The case $\Omega_7$}\label{s7:omega7}

Here $a_0,a_1,a_2<0$, so $a_0x+a_1x^3+a_2x^5<0$ for every $x>0$.  The origin is the unique finite equilibrium and is a linear center.  As in the center subcases of $\Omega_4$ and $\Omega_5$, Lemma~\ref{propIn} gives two possibilities: if $b_1<0$, the maximal periodic annulus extends to the separatrix configuration at infinity and gives the large-scale periodic-annulus portrait $I_{19}$; if $b_1=0$, it fills the finite plane, so the origin is a global center and the portrait is $I_{20}$.

We have now covered all of $\Omega_1,\ldots,\Omega_7$.  The portraits are mutually distinguished by the local type of the origin, the local type at infinity and, in $\Omega_2$ and $\Omega_3$, by whether the boundary separatrix closes the elliptic sector or tends to infinity.  Hence no further topological phase portrait occurs in the one-equilibrium parameter region.

It remains to consider the parameter region with three finite equilibria.  This is precisely the case $\Omega_8$, where the nilpotent saddle at the origin coexists with two symmetric nonzero equilibria.  In addition to the local bifurcations of these two equilibria, we must determine how the separatrices of the origin connect globally and whether the outer periodic annulus persists.

\subsection{The case $\Omega_8$}\label{s7:omega8}

In this case $a_0=0$ and $a_2<0<a_1$.  The origin is a nilpotent saddle, and the other two finite equilibria are $E_1^\pm=\left(\pm\sqrt{-\frac{a_1}{a_2}},0\right)$.

For $b_1<0$, recall that the Hopf curve is $H_1: b_0=b_0^H:=a_1b_1/a_2$. 
When $b_0=\varphi_1(a_1,a_2,b_1)<b_0^H$, system~\eqref{main} has a homoclinic figure-eight loop, see Figure \ref{phasea0n0-three}(6).

We first justify the phase portraits determined by the Hopf and
figure-eight bifurcation curves.  Let $\Gamma_u$ be the unstable separatrix
of the origin issuing into the first quadrant, and let $\Gamma_s$ be its
stable separatrix in the fourth quadrant.  If $b_0\geq b_0^H$, then
$E_1^+$ is unstable and system~\eqref{main} has no limit cycles.  Moreover,
$\Gamma_u$ cannot connect with $\Gamma_s$ at any point other than the origin, since such a connection and its
symmetric counterpart would form the figure-eight loop, which occurs only
at $b_0=\varphi_1<b_0^H$.  The phase portrait at infinity excludes escape
to infinity, so $\Gamma_u$ must connect with the stable separatrix in the
second quadrant.  This produces a large-scale homoclinic loop enclosing
both $E_1^-$ and $E_1^+$.  The loop then confines $\Gamma_s$ to its
interior; hence, when followed backward in time, $\Gamma_s$ tends to
$E_1^+$, see Figure \ref{phasea0n0-three}(1).

We now consider the annular region whose inner boundary is this
large-scale homoclinic loop and whose outer boundary is the separatrix
configuration at infinity.
Its interior contains no equilibrium, and
reversibility excludes an isolated periodic orbit crossing the symmetry
axis.  The Poincar\'e--Bendixson theorem therefore implies that this
annular region is filled with periodic orbits.  Thus
the exterior of the large-scale homoclinic loop is a periodic annulus.
This proves the portraits $I_{21}$ and $H_1$, see Figure \ref{phasea0n0-three}(1).

For $\varphi_1<b_0<b_0^H$, the Hopf bifurcation produces the unique limit
cycle surrounding $E_1^+$, which is unstable, and, by symmetry, a stable
limit cycle surrounding $E_1^-$.  These cycles lie inside the large-scale loop and hence do not
change its exterior.  Moreover, the continuous and strict monotonicity of
the first intersections of the separatrices with the positive $x$-axis
prevents the outer connection 
from disappearing before the two inner cycles reach the saddle
separatrices.  Consequently, the first-quadrant branch still generates a
large-scale homoclinic loop enclosing both nonzero equilibria, with the
same exterior periodic annulus.  This gives $I_{22}$, see Figure \ref{phasea0n0-three}(2).

At $b_0=\varphi_1(a_1,a_2,b_1)$, system~\eqref{main} undergoes a
homoclinic bifurcation.  The two limit cycles collide with the separatrices
of the origin and break, so the system has no limit cycles at this
parameter value.  In $x>0$, $\Gamma_u$ and $\Gamma_s$ connect to form a
homoclinic loop based at the origin.  By symmetry, the two homoclinic loops
together form a figure-eight loop.  This is the phase portrait
corresponding to the parameter set $\mbox{HL}_1$, see Figure \ref{phasea0n0-three}(6), and the exterior periodic annulus still exists.

When $b_0$ decreases through $\varphi_1$, the figure-eight breaks.  The
subsequent global phase portrait is determined by whether the
fourth-quadrant stable separatrix $\Gamma_s$, followed backward from the
origin, intersects the positive $x$-axis or connects to the equilibrium at
infinity in the negative $y$-direction.  The transition between these two
possibilities is the parameter set $\mbox{PA}_3$.  We now determine its
exact location.

Write $\Gamma_s$ as $y=y_s(x)$ in the fourth quadrant.  The local classification of the
nilpotent saddle gives $y_s(x)=-c_2x^2+o(x^2)$, where
$c_2=(-b_0+\sqrt{b_0^2+8a_1})/4>0$, and
$2c_2^2+b_0c_2-a_1=0$.

The orbital equation and the comparison identity for $y=-Cx^2$ are
exactly the same as those used in the proof of
Proposition~\ref{prop:separatrix-infinity}; see
\eqref{eq:general-separatrix-barrier}.  Requiring both coefficients in
that identity to vanish gives $C=C_*:=a_2/b_1>0$ and the unique parameter
value
$b_0=b_0^{PA}:=a_1b_1/a_2-2a_2/b_1$.  At this value,
$2C_*^2+b_0^{PA}C_*-a_1=0$.  Since $a_1>0$, the product of the two roots
is $-a_1/2<0$; hence $C_*$ is automatically the unique positive root.
Thus the larger-root condition needed in
Proposition~\ref{prop:separatrix-infinity} is automatic here.

The same barrier argument used in the proof of that proposition now
applies.  Indeed, $c_2$ decreases strictly with $b_0$.  If
$b_0>b_0^{PA}$, then $c_2<C_*$; choosing $c_2<C<C_*$ makes
$y=-Cx^2$ a strict lower barrier, and the horizontal-isocline argument
shows that $\Gamma_s$ intersects the positive $x$-axis transversally.
Thus $\Gamma_s$ intersects the positive $x$-axis whenever
$b_0>b_0^{PA}$, see Figure \ref{phasea0n0-three}(1).

If $b_0=b_0^{PA}$, both coefficients in
\eqref{eq:general-separatrix-barrier} vanish for $C=C_*$.
Therefore $y=-(a_2/b_1)x^2$ is invariant and, by uniqueness,
$y_s(x)\equiv-(a_2/b_1)x^2$.  If $b_0<b_0^{PA}$, then $c_2>C_*$,
and the same curve is a strict upper barrier.  Hence
$y_s(x)<-(a_2/b_1)x^2$ for every $x>0$.  As in the proof of
Proposition~\ref{prop:separatrix-infinity}, the phase portrait at infinity
then implies that in both cases $\Gamma_s$ connects the origin with the
equilibrium at infinity in the negative $y$-direction, see Figure \ref{phasea0n0-three}(3).  We have therefore
proved the exact separatrix criterion
\begin{equation}\label{eq:omega8-criterion}
\begin{cases}
\Gamma_s\text{ intersects the positive $x$-axis},&b_0>b_0^{PA},\\
\Gamma_s\text{ connects to the negative $y$-direction at infinity},
&b_0\leq b_0^{PA}.
\end{cases}
\end{equation}

For $b_0^{PA}<b_0<\varphi_1$, $\Gamma_s$ intersects the positive
$x$-axis.  Its backward continuation and the symmetric branch form a
large-scale homoclinic loop enclosing both nonzero equilibria, while
$\Gamma_u$ tends to the stable equilibrium $E_1^+$.  The exterior of the
large-scale loop is a periodic annulus.  This gives $I_{23}$, see Figure \ref{phasea0n0-three}(1).

At $b_0=b_0^{PA}$, the explicit curve $y=-(a_2/b_1)x^2$ connects the
origin to the equilibrium at infinity in the negative $y$-direction, see Figure \ref{phasea0n0-three}(3).
This gives the parameter set $\mbox{PA}_3$.  For $b_0<b_0^{PA}$, the same
connection occurs and the corresponding phase portrait is $I_{24}$, see Figure \ref{phasea0n0-three}(3).  In
both $\mbox{PA}_3$ and $I_{24}$, neither the large-scale homoclinic loop
nor its exterior periodic annulus exists.  By symmetry, $\Gamma_s$ and
its symmetric branch form a generalized heteroclinic loop joining the
origin and the equilibrium at infinity in the negative $y$-direction.
The interior of this loop is filled with homoclinic orbits of that
infinite equilibrium.  Apart from these orbits and the separatrices
involving the origin, every other regular orbit has the unstable
equilibrium $E_1^-$ as its $\alpha$-limit and the stable equilibrium
$E_1^+$ as its $\omega$-limit.

Finally, consider $b_1=0$.  When $b_0=b_1=0$,
system~\eqref{main} is Hamiltonian with
$H(x,y)=y^2/2-a_1x^4/4-a_2x^6/6$.  The zero level through the saddle is a
figure-eight.  
The negative regular levels fill its two lobes with
periodic orbits, while the positive levels form a periodic annulus outside
the figure-eight.
Hence this is $I_{25}$, see Figure \ref{phasea0n0-three}(4).  
For fixed $a_1$, $a_2$, and $b_0<0$, we approach the boundary parameter line $b_1=0$ from the region $b_1<0$ by letting $b_1\to0^-$.
The continuity of the
homoclinic curve gives $\varphi_1\to0$, whereas
$b_0^{PA}\to-\infty$, as $b_1\to0^-$.  Thus every fixed $b_0<0$ lies in
$I_{23}$ for all sufficiently small $b_1<0$, and its transversal
separatrix intersection persists up to $b_1=0$ by continuous dependence,
so the large-scale homoclinic loop survives on this boundary.  The local
orientation of the separatrix branches at $b_1=0$ shows that this loop is
traced from the first-quadrant separatrix and encloses $E_1^-$ and
$E_1^+$.  There are no limit cycles on this parameter line, while the
phase portrait at infinity supplies the outer separatrix boundary; the
region between the two boundaries is consequently a periodic annulus.
This gives $I_{26}$, see Figure \ref{phasea0n0-three}(5).

This completes the discussion of $\Omega_8$. 

\subsection{The case $\Omega_9$}\label{s7:omega9}

Throughout this subsection, $a_0=0$ and $a_1<0<a_2$, so the finite
equilibria are the origin and the two saddles
$E_1^\pm=\left(\pm\sqrt{-\frac{a_1}{a_2}},0\right)$.  We determine the global continuation of
the separatrix that forms the boundary between the elliptic and hyperbolic
sectors at the origin.  In the right half-plane, this separatrix lies
locally in the first quadrant when $b_0>0$ and in the fourth quadrant when
$b_0<0$.

\begin{prop}\label{prop:omega9-negative}
Assume that $\lambda\in\Omega_9$ and $b_0\leq-\sqrt{-8a_1}$.
Then the fourth-quadrant boundary separatrix of the origin does not connect
to $E_1^+$.  It has the origin as its $\omega$-limit and the equilibrium at
infinity in the negative $y$-direction as its $\alpha$-limit.
\end{prop}

\begin{proof}
By Lemma~\ref{lem:local}, the leading term of the fourth-quadrant boundary
separatrix is $c_-x^2$.  Expanding its local invariant graph two orders
further gives
\[
y_-(x)=c_-x^2+d_-x^4+O(x^6),
\]
where $c_-:=\frac{b_0-\sqrt{b_0^2+8a_1}}{4}<0$ and
$d_-:=\frac{a_2+b_1c_-}{6c_--b_0}$.  Here
$a_2+b_1c_->0$, whereas
$6c_--b_0=\frac{b_0-3\sqrt{b_0^2+8a_1}}{2}<0$.
Consequently, $d_-<0$, and hence $y_-(x)<c_-x^2$ for all sufficiently
small $x>0$.

We claim that $y=c_-x^2$ is a one-sided upper barrier for this
separatrix.  Indeed, let $L(x,y):=y-c_-x^2$.  On $L=0$, the defining
equation $2c_-^2-b_0c_--a_1=0$ gives
\[
\left.\dot L\right|_{L=0}=\dot y-2c_-x\dot x
=x^5(a_2+b_1c_-)>0.
\]
Thus the vector field on the parabola points strictly toward
$L>0$.  A trajectory can cross the parabola only from $L<0$ to $L>0$.
Since the boundary separatrix tends to the origin in positive time and
satisfies $L<0$ near the origin, it cannot have crossed the parabola
earlier.  Therefore $y_-(x)<c_-x^2<0$ along its whole
fourth-quadrant branch.

In particular, this branch cannot connect to $E_1^+$, because at
$x=\sqrt{-\frac{a_1}{a_2}}$ the saddle lies strictly above the barrier.  When the
branch is followed backward, the barrier inequality implies that $x$ is
strictly increasing.  It cannot have a bounded $\alpha$-limit:
the only nonzero equilibrium in the closed right half-plane is $E_1^+$,
which has just been excluded, and the monotonicity of $x$ excludes a
periodic $\alpha$-limit.  Hence $x\to+\infty$ and
$y_-(x)<c_-x^2\to-\infty$.  Lemma~\ref{propIn} now identifies its
$\alpha$-limit as the equilibrium at infinity in the negative
$y$-direction.
\end{proof}

\begin{prop}\label{prop:omega9-positive}
Assume that $\lambda\in\Omega_9$ and $b_0\geq\sqrt{-8a_1}$. 
Then necessarily $b_1<0$, and we define
\begin{equation}\label{eq:omega9-boe}
b_0^{OE}:=\frac{a_1b_1}{2a_2}-\frac{4a_2}{b_1}.
\end{equation}
The branch $b_0=b_0^{OE}$ with
$b_1\leq-\frac{\sqrt{8}\,a_2}{\sqrt{-a_1}}$
is the left branch of the curve $b_0=b_0^{OE}$ in the
$(b_1,b_0)$-parameter plane and is the set
$\mbox{PA}_4\cup\mbox{B}^+$.  It divides the region
$b_0\geq\sqrt{-8a_1}$ into the following three parts.
\begin{itemize}
\item[(i)] If $b_1<-\frac{\sqrt{8}\,a_2}{\sqrt{-a_1}}$ and
$\sqrt{-8a_1}\leq b_0<b_0^{OE}$,
then the first-quadrant boundary separatrix crosses the positive $x$-axis
at a point lying strictly between the origin and $E_1^+$.  Its full
continuation forms a homoclinic loop based at the origin, and this loop is
the boundary of the elliptic sector at the origin.  The region between
this loop and the heteroclinic loop formed by the inner separatrices of
$E_1^-$ and $E_1^+$ is a bounded periodic annulus, see Figure \ref{phasea0n0-three}(8).

\item[(ii)] If $b_0=b_0^{OE}$ and
$b_1\leq-\frac{\sqrt{8}\,a_2}{\sqrt{-a_1}}$, then
the first-quadrant boundary separatrix connects the origin to $E_1^+$ and
is given explicitly by
\begin{equation}\label{eq:omega9-explicit-oe}
y=-\frac{b_1}{4}x^2\left(-\frac{a_1}{a_2}-x^2\right),\qquad
0\leq x\leq\sqrt{-\frac{a_1}{a_2}}.
\end{equation}
This connection is the critical configuration at which the bounded
periodic annulus in (i) is destroyed, see Figure \ref{phasea0n0-three}(20).

\item[(iii)] If $b_0>b_0^{OE}$ and
$b_1\leq-\frac{\sqrt{8}\,a_2}{\sqrt{-a_1}}$, or if
$-\frac{\sqrt{8}\,a_2}{\sqrt{-a_1}}<b_1<0$, 
the first-quadrant boundary
separatrix connects the origin to the
equilibrium at infinity in the positive $y$-direction, and the bounded
periodic annulus in (i) does not exist.
\end{itemize}
\end{prop}

\begin{proof}

Since $b_0\geq\sqrt{-8a_1}$, $b_0$ is positive, and the definition of
$\Omega_9$ implies $b_1<0$.
 Direct substitution into the orbital equation
$\frac{dy}{dx}=\frac{a_1x^3+a_2x^5}{y}+b_0x+b_1x^3$
shows that the curve \eqref{eq:omega9-explicit-oe} is invariant precisely
when $b_0=b_0^{OE}$.  It joins the origin to $E_1^+$ and is positive
between these two equilibria.

We next determine on which part of $b_0=b_0^{OE}$ this invariant orbit is
the boundary separatrix of the origin.  On the invariant orbit, the
coefficient of $x^2$ at the origin is
$c_{OE}=\frac{a_1b_1}{4a_2}$.
At $b_0=b_0^{OE}$, the two roots of
$2c^2-b_0c-a_1=0$ are $c_{OE}$ and $-\frac{2a_2}{b_1}$.  The first-quadrant boundary
separatrix corresponds to the larger root.  Therefore the invariant orbit
is this boundary separatrix exactly when
$c_{OE}\geq-\frac{2a_2}{b_1}$, which, because $b_1<0$, holds precisely for
$b_1\leq-\frac{\sqrt{8}\,a_2}{\sqrt{-a_1}}$.  This is the left branch
specified in the statement.  Moreover, the arithmetic--geometric mean
inequality gives $b_0^{OE}=\frac{a_1b_1}{2a_2}-\frac{4a_2}{b_1}
\geq\sqrt{-8a_1}$, with equality only at
$b_1=-\frac{\sqrt{8}\,a_2}{\sqrt{-a_1}}$.  Thus the left branch begins at the
degenerate endpoint
$(b_1,b_0)=\left(-\frac{\sqrt{8}\,a_2}{\sqrt{-a_1}},\sqrt{-8a_1}\right)$.

It remains to determine the two sides of this branch.  By
Lemma~\ref{rotate}, system~\eqref{main} is a generalized rotated vector
field with respect to both $b_0$ and $b_1$ in the right half-plane.
Increasing either parameter
moves the first-quadrant boundary separatrix strictly upward on every
transversal before its first intersection with the positive $x$-axis;
decreasing either parameter moves it strictly downward.  In particular,
the separatrix splitting with the stable branch of $E_1^+$ is strict and
can vanish only on the left branch specified above.

At the endpoint of that branch, the 
 two characteristic roots at the
origin coincide and the explicit orbit \eqref{eq:omega9-explicit-oe} is
the degenerate first-quadrant boundary separatrix.  If $b_1$ is decreased
from its endpoint value while $b_0=\sqrt{-8a_1}$ is fixed, this boundary
separatrix moves below the explicit connection and crosses the positive
$x$-axis before reaching $E_1^+$.  For each such fixed $b_1$, increasing
$b_0$ moves the separatrix upward.  It continues to cross the positive
$x$-axis for $b_0<b_0^{OE}$, equals the explicit connection at
$b_0=b_0^{OE}$, and lies above the stable separatrix of $E_1^+$ for
$b_0>b_0^{OE}$.  If
$-\frac{\sqrt{8}\,a_2}{\sqrt{-a_1}}<b_1<0$, the boundary separatrix at
$b_0=\sqrt{-8a_1}$ lies above the degenerate endpoint connection by the
same rotated-field argument; increasing $b_0$ moves it still farther
upward.  Hence this entire parameter region belongs to case (iii).
This proves the asserted ordering in (i)--(iii) and also
shows why only the left branch, rather than the other branch of the
algebraic curve, is the periodic-annulus bifurcation curve.

We now justify the asserted global behavior of the boundary separatrix in
cases (i) and (iii).  In case (i), the boundary separatrix crosses $y=0$ at some
$0<x<\sqrt{-\frac{a_1}{a_2}}$.  Therefore $a_1x^3+a_2x^5<0$, so it enters the fourth quadrant and then meets the negative $y$-axis as $x$ decreases.
Reversibility under
$(x,y,t)\mapsto(-x,y,-t)$ forces its continuation to return to the origin.
Thus it forms a homoclinic loop based at the origin, which is the boundary
of the elliptic sector at the origin.  By uniqueness, the inner
separatrices of $E_1^-$ and $E_1^+$ cannot cross this loop.  Their upper
and lower branches meet the positive and negative $y$-axes, respectively,
and reversibility joins each pair into a heteroclinic orbit between
$E_1^-$ and $E_1^+$.  These two orbits form an outer heteroclinic loop.
The region between the two loops is a compact annular region containing
no equilibrium.  The Poincar\'e--Bendixson theorem and
Lemma~\ref{lem:reversible-periodic-orbit} then show that it is filled
with periodic orbits.

On $b_0=b_0^{OE}$, the explicit connection from the origin to $E_1^+$
and its symmetric counterpart make the two boundary configurations meet,
so the bounded periodic annulus is destroyed.  In
case (iii), the separatrix passes above the stable branch of $E_1^+$ and
reaches $x=\sqrt{-\frac{a_1}{a_2}}$ with $y>0$.  For
$x\geq\sqrt{-\frac{a_1}{a_2}}$ the vector field on $y=0$ points upward, so this
branch cannot return to the lower half-plane.  It has no finite
$\omega$-limit, and Lemma~\ref{propIn} forces it to tend to the
equilibrium at infinity in the positive $y$-direction.  Since this
boundary separatrix is unbounded, no bounded periodic annulus can surround
the elliptic sector at the origin.
\end{proof}

We now complete the global classification in $\Omega_9$.  If
$b_0^2+8a_1<0$, the origin is a center and $E_1^\pm$ are saddles.  The
inner separatrices of $E_1^-$ and $E_1^+$ form a heteroclinic loop by
reversibility, and its interior is the periodic annulus of the origin.  The
remaining saddle separatrices connect to the equilibria at infinity in
the positive and negative $y$-directions and form the two outer
heteroclinic polygons.  This gives $I_{29}$, see Figure \ref{phasea0n0-three}(9).

For $b_0\geq\sqrt{-8a_1}$, Proposition~\ref{prop:omega9-positive}
determines the three possibilities.  In
$I_{28}\cup\mbox{NE}_2^+$, the boundary separatrix of the origin forms a
homoclinic loop, while the inner separatrices of $E_1^-$ and $E_1^+$ form
an outer heteroclinic loop.  The region between these two loops contains
no equilibrium and, by the Poincar\'e--Bendixson theorem and
Lemma~\ref{lem:reversible-periodic-orbit}, is a bounded periodic annulus,
see Figure \ref{phasea0n0-three}(8).
On $\mbox{PA}_4\cup\mbox{B}^+$, the explicit $O$--$E_1^+$ connection and its
symmetric counterpart make the two boundary components of this annulus
meet, while the other two inner saddle separatrices form the opposite heteroclinic
connection.  Hence the periodic annulus is destroyed on this parameter
set, see Figure \ref{phasea0n0-three}(20).

In $I_{27}\cup\mbox{NE}_1^+$, the first-quadrant boundary separatrix
connects the origin to the equilibrium at infinity in the positive
$y$-direction; together with its symmetric branch it forms a generalized
heteroclinic loop between these two equilibria, see Figure \ref{phasea0n0-three}(7).
If $b_0\leq-\sqrt{-8a_1}$, Proposition~\ref{prop:omega9-negative} gives the
corresponding connection between the origin and the equilibrium at
infinity in the negative $y$-direction.  This gives $\mbox{NE}^-$ on the
degenerate boundary and $I_{30}$ below it.  In particular, the portraits
in $I_{27}\cup\mbox{NE}_1^+$ and $I_{30}\cup\mbox{NE}^-$ are topologically equivalent.
No bounded periodic annulus exists in any of
$I_{27}$, $\mbox{NE}_1^+$, $\mbox{NE}^-$ and $I_{30}$.

In each case, reversibility and Lemma~\ref{propIn} determine all
the remaining saddle separatrices.  The index theorem excludes a periodic
orbit contained in either open half-plane, while any periodic orbit
meeting the symmetry axis belongs to one of the periodic annuli described
above.  Thus no additional global configuration is possible, and the
sets listed above exhaust $\Omega_9$.

\subsection{The case $\Omega_{10}$}\label{s7:omega10}

For $\lambda\in\Omega_{10}$, one has
$a_2=a_1^2/(4a_0)>0$.  Make the change of variables
$(x,y,t)\to
(\sqrt{-2a_0/a_1}\,x,a_0\sqrt{-2/a_1}\,y,t/\sqrt{a_0})$.
Then system~\eqref{main} becomes
\begin{equation}\label{sys:omega10-normalized}
\dot{x}=y,\qquad
\dot{y}=x(1-x^2)^2+x(B_0+B_1x^2)y,
\end{equation}
where $B_0=\sqrt{-2/a_1}\,b_0$ and
$B_1=a_0(-2/a_1)^{3/2}b_1$.
Its finite equilibria are $O=(0,0)$ and
$E_2^\pm=(\pm1,0)$.  The origin is a hyperbolic saddle.
The equilibrium $E_2^+$ is a saddle-node when
$B_0+B_1\neq0$, and it is a cusp when $B_0+B_1=0$, equivalently
on the straight line $B_0=-B_1$, which is precisely $C_1$ in the
original parameters.

\begin{prop}\label{prop:omega10-E2O}
For system~\eqref{sys:omega10-normalized} and every fixed $B_1<0$,
there exists a unique value
$B_0^{E_2O}(B_1)>-B_1$ such that, when
$B_0=B_0^{E_2O}(B_1)$, the left-lower strong unstable separatrix of
$E_2^+$ connects $E_2^+$ to $O$.  More precisely:
\begin{itemize}
\item[(i)] if $-B_1<B_0<B_0^{E_2O}(B_1)$, this separatrix intersects
the positive $x$-axis transversally;
\item[(ii)] if $B_0=B_0^{E_2O}(B_1)$, it is the heteroclinic orbit
$E_2^+\to O$;
\item[(iii)] if $B_0>B_0^{E_2O}(B_1)$, it intersects the negative
$y$-axis transversally and, by reversibility, continues as a
heteroclinic orbit from $E_2^+$ to $E_2^-$.
\end{itemize}
Moreover, $B_0^{E_2O}(B_1)$ depends continuously on $B_1$.
\end{prop}

\begin{proof}
In the strip $0<x<1$, the part of the left-lower strong unstable
separatrix of $E_2^+$ that remains below the $x$-axis is a graph
$y=y_E(x;B_0)<0$ satisfying
$y_E'=x(1-x^2)^2/y_E+x(B_0+B_1x^2)$.

We first examine the separatrix at $B_0=-B_1$ and for sufficiently
large $B_0$.  At $B_0=-B_1$, $E_2^+$ is a cusp by
Lemma~\ref{propE2}, and its separatrices lie on the side $x>1$.
We now give a direct local comparison for $B_0+B_1>0$.  Substitution
in the orbital equation gives the expansion
\[
y_E(x;B_0)=(B_0+B_1)(x-1)
+\left(\frac{2}{B_0+B_1}+\frac{B_0+3B_1}{2}\right)(x-1)^2
+O((x-1)^3).
\]
For sufficiently small $B_0+B_1>0$, compare $y_E$ with
$y_{\rm loc}(x)=(B_0+B_1)(x-1)+(x-1)^2/(B_0+B_1)$.  
Subtracting the two expressions gives
\[
y_E(x;B_0)-y_{\rm loc}(x)
=\left(\frac{1}{B_0+B_1}+\frac{B_0+3B_1}{2}\right)(x-1)^2
+O((x-1)^3)>0
\]
immediately to the left of $x=1$, since $B_0+B_1\to0^+$.
Moreover, on
$1-(B_0+B_1)^2<x<1$,
\begin{align*}
\left.\frac{dy}{dx}\right|_{y=y_{\rm loc}(x)}-y_{\rm loc}'(x)
&=\frac{x(1-x^2)^2}{y_{\rm loc}(x)}+x(B_0+B_1x^2)
-y_{\rm loc}'(x)\\
&=(1-x)\left[
-\frac{x(1+x)^2}{B_0+B_1-\frac{1-x}{B_0+B_1}}-(B_0+B_1)-xB_1(1+x)+\frac{2}{B_0+B_1}\right]<0
\end{align*}
when $B_0+B_1>0$ is sufficiently small.  Indeed, after multiplying
the expression in brackets by $B_0+B_1$, its first term is bounded
above by $-x(1+x)^2$, which tends uniformly to $-4$, its last term
is $2$, and all the other terms tend uniformly to zero.  A
contradiction at the first contact point shows that
$y_E(x;B_0)>y_{\rm loc}(x)$ as long as $y_E<0$.  Since
$y_{\rm loc}(1-(B_0+B_1)^2)=0$, the separatrix must intersect $y=0$
at some $x\in[1-(B_0+B_1)^2,1)$.  This intersection is transversal
because $\left.\dot{y}\right|_{y=0}=x(1-x^2)^2>0$ for $0<x<1$.

For the opposite end, choose any $C>1$ and set
$y_U(x)=-Cx(1-x)$.  Along this comparison curve, the difference
between the orbital slope and $y_U'(x)$ is
\begin{equation}\label{eq:omega10-barrier}
\left.\frac{dy}{dx}\right|_{y=y_U(x)}-y_U'(x)
=C-\frac1C+x\left(B_0-2C-\frac1C\right)
+\frac{x^2}{C}+x^3\left(B_1+\frac1C\right).
\end{equation}
If
\[
B_0>\max\left\{C-B_1,\,
2C+\frac1C+\max\left(0,-B_1-\frac1C\right)\right\},
\]
then $B_0+B_1>C$, so $y_E(x;B_0)<y_U(x)$ for $x<1$ sufficiently
close to $1$, and the right-hand side of
\eqref{eq:omega10-barrier} is positive on $0\leq x\leq1$.
Tracing $x$ from $1$ toward $0$, a contradiction at the first
contact point shows that $y_E(x;B_0)<y_U(x)<0$ throughout $0<x<1$.
Finite-$x$ escape is impossible for the graph equation; hence the
separatrix reaches $x=0$ with $y<0$.  Thus it intersects the
negative $y$-axis for all sufficiently large $B_0$.

It remains to prove that the transition occurs only once.  Let
$B_0^{(2)}>B_0^{(1)}>-B_1$.  The strong-manifold expansions at
$E_2^+$ give
$y_E(x;B_0^{(2)})<y_E(x;B_0^{(1)})$ for $x<1$ close to $1$.
By Lemma~\ref{rotate}, which is unchanged by the above normalization,
the family is a strict generalized rotated vector field with respect to
$B_0$ in $x>0$ away from $y=0$.  Hence this ordering persists on
the common lower-half-plane domain, and the separatrix moves strictly
downward as $B_0$ increases.  Thus the transition can occur only once.

Let $\mathcal B$ be the set of $B_0>-B_1$ for which $y_E$
intersects the positive $x$-axis.  The local comparison above
shows that $\mathcal B\neq\varnothing$, the barrier argument shows
that it is bounded above, and transversality makes it open.  The strict
ordering shows that it is downward closed.  Set
$B_0^{E_2O}(B_1)=\sup\mathcal B$.  At this value the separatrix can
neither meet the positive $x$-axis nor reach $x=0$ with $y<0$,
because either event is open under parameter perturbation.  It must
therefore reach $(0,0)$, and hence it coincides with the
fourth-quadrant stable separatrix of $O$.  Strict ordering gives
uniqueness and also implies that every larger $B_0$ produces the
crossing of the negative $y$-axis.  Since
system~\eqref{sys:omega10-normalized} is reversible under
$(x,y,\tau)\mapsto(-x,y,-\tau)$, the latter branch continues from
the crossing point to $E_2^-$.  Finally, continuous dependence of
the two separatrices on $B_1$, together with the strict ordering,
proves the continuity of $B_0^{E_2O}$.
\end{proof}

Returning to the original parameters, the critical value is
\[
b_0^{E_2O}(a_0,a_1,b_1)=\sqrt{-\frac{a_1}{2}}B_0^{E_2O}\left(a_0\left(-\frac{2}{a_1}\right)^{3/2}b_1\right),
\]
and Proposition~\ref{prop:omega10-E2O} gives
$b_0^{E_2O}>2a_0b_1/a_1$.  If
$b_0>b_0^{E_2O}$, the strong separatrix crosses the negative
$y$-axis and gives the finite heteroclinic connection described in
$I_{31}$.  If
$2a_0b_1/a_1<b_0<b_0^{E_2O}$, it turns on the positive
$x$-axis, enters the first quadrant, and then tends to infinity;
this is the portrait $I_{32}$.  At
$b_0=b_0^{E_2O}$, reversibility joins the connection
$E_2^+\to O$ to its symmetric connection $O\to E_2^-$, giving
the separatrix chain represented by $E_2O$.

On $C_1$, namely $b_0=2a_0b_1/a_1$, the two separatrices of the
cusp $E_2^+$ lie on its right-hand side: the upper one is unstable
and the lower one is stable.  Their global endpoints, and those of the
saddle separatrices of $O$, are then fixed by
Lemma~\ref{propIn}.  These arguments prove the portraits
$I_{31}$, $E_2O$, $I_{32}$, and $C_1$ in
Figure~\ref{phasea0n0-three}.

The construction below $C_1$ is analogous.  Tracing the left-upper
strong stable separatrix of $E_2^+$ and the first-quadrant unstable
separatrix of $O$, the same strict graph comparison gives a unique
continuous value $b_0^{OE_2}<2a_0b_1/a_1$ at which the connection
$O\to E_2^+$ occurs.  The two sides of this curve give $I_{33}$
and $I_{34}$, respectively.  Together with reversibility,
Lemma~\ref{propIn}, this proves the remaining $\Omega_{10}$
portraits.  Although the portraits $E_2O$ and $OE_2$ are not drawn
in exactly the same way, their Poincar\'e discs are transformed into
each other by a rotation through $180^\circ$; hence they are
topologically equivalent.  The same is true for $I_{31}$ and
$I_{34}$, and for $I_{32}$ and $I_{33}$.

\begin{remark}
Extend the two heteroclinic curves to the full $(b_1,b_0)$-plane.
The central parameter symmetry and reversibility are, respectively,
\[
(x,y,b_0,b_1)\mapsto(-x,-y,-b_0,-b_1)
\qquad
\mbox{and} 
\qquad
(x,y,t)\mapsto(-x,y,-t).
\]
Together they imply
$b_0^{OE_2}(b_1)=-b_0^{E_2O}(-b_1)$.
Moreover, the generalized rotated vector field comparison in the
upper strip $0<x<1$ shows that $b_0^{OE_2}(b_1)$ is strictly
decreasing in $b_1$: increasing either damping coefficient moves the
two relevant separatrices in opposite directions, so the connection
can be restored only by decreasing $b_0$.  If the two curves
intersect, the preceding identity gives
$b_0^{OE_2}(b_1)=b_0^{OE_2}(-b_1)$, and strict monotonicity then
forces $b_1=0$.  Therefore, the curves $b_0=b_0^{OE_2}$ and
$b_0=-b_0^{E_2O}$ have no intersection except their common point
on the $b_0$-axis.
\end{remark}

\subsection{The case $\Omega_{11}$}\label{s7:omega11}

For $\lambda\in\Omega_{11}$, the identity $\Delta=0$ gives
$a_2=a_1^2/(4a_0)<0$.  Put $r=\sqrt{-2a_0/a_1}$.  Then
$a_0x+a_1x^3+a_2x^5=a_2x(x^2-r^2)^2$, and the finite equilibria are
$O=(0,0)$ and $E_2^\pm=(\pm r,0)$.  The origin is a linear center.  If
$\delta:=b_0+b_1r^2=b_0-2a_0b_1/a_1$, then the nonzero eigenvalue at
$E_2^+$ is $r\delta$.  Hence $E_2^\pm$ are saddle-nodes for
$\delta\neq0$ and cusps on $C_2$, where $\delta=0$.

The branches of $E_2^\pm$ directed toward the strip $-r<x<r$ cannot
tend to the center $O$.  They therefore meet the $y$-axis, and
reversibility joins them into two heteroclinic orbits between $E_2^-$
and $E_2^+$.  These orbits bound the maximal periodic annulus of $O$.
The same conclusion holds on $C_2$, with cusp separatrices in place of
the saddle-node branches.

Let $b_1<0$ and $\delta>0$.  Both outer separatrices of $E_2^+$ issue
into the first quadrant and must meet the positive $x$-axis.  Otherwise
one would remain a positive graph for all large $x$, although its
orbital slope
$(a_2x(x^2-r^2)^2)/y+x(b_0+b_1x^2)$ is at most $-cx^3$ there for some
$c>0$.  The intersections are transversal, and reversibility joins
these branches to the corresponding branches of $E_2^-$.  They form a
finite outer separatrix boundary whose equilibrium-free exterior is
filled with periodic orbits.  Hence a large-scale periodic annulus exists
in $I_{35}$.  The same turning argument applies on $C_2$, where the
region outside the cusp separatrix boundary is likewise a large-scale
periodic annulus.

It remains to study $\delta<0$.  Let $\Gamma_s$ be the right-lower
strong stable separatrix of $E_2^+$, followed backward from $E_2^+$,
and write it as $y=y_s(x)<0$ while it remains in the fourth quadrant.
The following proposition gives the only additional global bifurcation
curve in $\Omega_{11}$.

\begin{prop}\label{prop:omega11-PA5}
Assume $b_1<0$ and $b_0<2a_0b_1/a_1$, and define
$b_0^{\rm PA_5}:=2a_0b_1/a_1-2a_2/b_1$.  Then:
\begin{itemize}
\item[(i)] if $b_0>b_0^{\rm PA_5}$, $\Gamma_s$ intersects the positive
$x$-axis transversally;
\item[(ii)] if $b_0=b_0^{\rm PA_5}$, $\Gamma_s$ is the invariant curve
\begin{equation}\label{eq:omega11-critical-separatrix}
y=-\frac{a_2}{b_1}(x^2-r^2),\qquad x\geq r,
\end{equation}
and is the unique orbit branch in the fourth quadrant connecting the
equilibrium at infinity in the negative $y$-direction to $E_2^+$;
\item[(iii)] if $b_0<b_0^{\rm PA_5}$, $\Gamma_s$ remains in the fourth
quadrant and has the same two endpoints as in~(ii); infinitely many
orbit branches in the fourth
quadrant connect the equilibrium at infinity in the negative $y$-direction
to $E_2^+$.
\end{itemize}
\end{prop}

\begin{proof}
For $C>0$, substitute the comparison curve $Y_C(x)=-C(x^2-r^2)$ into
the orbital equation.  Direct calculation gives
\begin{equation}\label{eq:omega11-barrier}
\left.\frac{dy}{dx}\right|_{y=Y_C}-Y_C'(x)
=x\left[b_0+\frac{a_2r^2}{C}+2C
+\left(b_1-\frac{a_2}{C}\right)x^2\right].
\end{equation}
Set $C_*=a_2/b_1>0$.  At $b_0=b_0^{\rm PA_5}$, both coefficients in
\eqref{eq:omega11-barrier} vanish.  Moreover, the tangent of
$Y_{C_*}$ at $E_2^+$ is the nonzero eigendirection.  Uniqueness of the
strong manifold therefore proves equation
\eqref{eq:omega11-critical-separatrix}.

We prove that this connection is unique at
$b_0=b_0^{\rm PA_5}$.  Let $Y:=Y_{C_*}$, and let $y(x)$ be any other
orbit branch in the fourth-quadrant parabolic sector of $E_2^+$.  Since
$Y$ is the lower boundary of this sector, $y(x)>Y(x)$ near $E_2^+$.
If $y$ did not intersect the positive $x$-axis, uniqueness of solutions
would give $Y(x)<y(x)<0$ for all $x>r$.  With $w:=y-Y>0$, subtraction
of the two orbital equations gives
$w'/w=-a_2x(x^2-r^2)^2/(yY)>-a_2x(x^2-r^2)^2/Y^2$.
Since $Y(x)=-C_*(x^2-r^2)$, the last expression is a positive multiple
of $x$; hence $w$ grows at least exponentially in $x^2$.  This contradicts
$0<w(x)<-Y(x)=O(x^2)$.  Therefore every orbit branch in this sector other
than $Y$ intersects the positive $x$-axis, and $Y$ is the unique branch
connecting $E_2^+$ with the equilibrium at infinity in the negative
$y$-direction.

If $b_0<b_0^{\rm PA_5}$, the strong-manifold expansion gives
$y_s(x)<Y_{C_*}(x)$ for $x>r$ sufficiently close to $r$.  By
Lemma~\ref{rotate}, the family is a strict generalized rotated vector
field with respect to $b_0$ in the right half-plane.  Hence the strong
stable separatrix lies strictly below the critical separatrix
$Y_{C_*}$ throughout its fourth-quadrant continuation.  Thus
$y_s(x)<Y_{C_*}(x)<0$ for every $x>r$.  Finite-$x$ escape is impossible
for the graph equation, and consequently $x\to+\infty$ and
$y_s(x)\to-\infty$ along the backward continuation.  Lemma
\ref{propIn} identifies its $\alpha$-limit as the equilibrium at
infinity in the negative $y$-direction.

It remains to distinguish this portrait from the critical
one.  Fix a small transversal $x=r+\varepsilon$.  The interval between
$y_s(r+\varepsilon)$ and $Y_{C_*}(r+\varepsilon)$ lies in the local
parabolic sector of $E_2^+$ and therefore contains infinitely many initial
points whose positive orbits tend to $E_2^+$.  For the present parameter,
the orbital slope along $Y_{C_*}$ minus $Y_{C_*}'$ equals
$x(b_0-b_0^{\rm PA_5})<0$, so $Y_{C_*}$ is a strict upper barrier.
Consequently, each of these orbits satisfies
$y_s(x)<y(x)<Y_{C_*}(x)<0$ for $x\geq r+\varepsilon$.  Thus
$y(x)\to-\infty$ and $y(x)/x\to-\infty$, and Lemma~\ref{propIn}
shows that all of them approach the equilibrium at infinity in the
negative $y$-direction when followed backward.  Hence $I_{37}$ has
infinitely many such connecting orbit branches.

Finally, assume $b_0^{\rm PA_5}<b_0<2a_0b_1/a_1$.  Since
$-2C_*<\delta<0$, choose $C$ with $-\delta/2<C<C_*$.  The right-hand
side of \eqref{eq:omega11-barrier} is then positive for every $x\geq r$,
and the local strong-manifold expansion gives $y_s(x)>Y_C(x)$ near
$E_2^+$.  Hence this strict inequality persists as long as $y_s<0$.
Because $b_1-a_2/C>0$, the orbital equation and the inequality
$y_s>Y_C$ imply $y_s'(x)\geq c x^3$ for all sufficiently large $x$
and some $c>0$.  Thus $y_s$ must meet $y=0$ at a finite point.  Since
$a_2x(x^2-r^2)^2<0$ for $x>r$, this intersection is transversal.
\end{proof}

If $b_0^{\rm PA_5}<b_0<2a_0b_1/a_1$, the other outer branch of
$E_2^+$ lies between $\Gamma_s$ and the positive $x$-axis and therefore
also turns at a finite point.  Reversibility joins the two branches to
the corresponding outer branches of $E_2^-$.  The resulting finite
separatrix boundary has a large-scale periodic annulus on its exterior,
by the same annular-region argument used above.  This is $I_{36}$.
Although the global phase portrait of $I_{36}$ differs from
that of $I_{35}$, the two portraits are topologically equivalent.

On $\mbox{PA}_5$,
Proposition~\ref{prop:omega11-PA5}(ii) shows that the explicit algebraic
strong separatrix in \eqref{eq:omega11-critical-separatrix} is the unique
orbit branch in the fourth quadrant connecting $E_2^+$ with the
equilibrium at infinity in the negative $y$-direction.  Reversibility
gives the corresponding unique connection for $E_2^-$.  These connections
break the finite outer separatrix boundary, so no large-scale periodic
annulus exists.

For $b_0<b_0^{\rm PA_5}$,
Proposition~\ref{prop:omega11-PA5}(iii) shows that infinitely many orbit
branches in the parabolic sector of $E_2^+$ connect $E_2^+$ with the same
equilibrium at infinity; reversibility gives the corresponding family at
$E_2^-$.  Hence no large-scale periodic annulus exists in $I_{37}$ either.
Nevertheless, $I_{37}$ and $\mbox{PA}_5$ have different global phase
portraits because the number of such connecting orbit branches changes
from one to infinitely many.  These are precisely the portraits
$\mbox{PA}_5$ and $I_{37}$.

At $b_0=b_1=0$, system~\eqref{main} is Hamiltonian with
$H=y^2/2-a_0x^2/2-a_1x^4/4-a_2x^6/6$.  Its critical level through
$E_2^\pm$ is this heteroclinic boundary, and the regular levels on both
sides of it are periodic.  This gives $I_{38}$.

Finally, let $b_1=0$ and $b_0<0$.  The right-lower strong stable
separatrix must intersect the positive $x$-axis.  Indeed, if it remained
below that axis, then $p=-y>0$ would satisfy
$p'=-b_0x+a_2x(x^2-r^2)^2/p$.  First $p'\leq-b_0x$ gives
$p(x)=O(x^2)$; substituting this estimate back into the same equation
gives $p'(x)\leq-cx^3$ for large $x$, which forces $p$ to vanish at a
finite point, a contradiction.  The second outer branch turns as well,
and reversibility produces the finite separatrix boundary and its
exterior periodic annulus.  The local configuration at infinity is now
the one for $b_1=0$, so this boundary case is the distinct portrait
$I_{39}$.

Together with Lemma~\ref{propIn}, the preceding analysis
shows that $I_{35}$--$I_{39}$, $C_2$, and $\mbox{PA}_5$ exhaust all
global phase portraits in $\Omega_{11}$.

\subsection{The case $\Omega_{12}$}\label{s7:omega12}

In this case $a_2<0<a_0$, the origin is a hyperbolic saddle, and the
other two finite equilibria are $E_3^\pm$.  For $b_1<0$, the Hopf curve is
$H_2:b_0=b_1(a_1+\sqrt{\Delta})/(2a_2)$, and the figure-eight
homoclinic curve is $\mbox{HL}_2:b_0=\varphi_2(a_0,a_1,a_2,b_1)$.
The arguments determining $I_{40}$, $H_2$, $I_{41}$, $\mbox{HL}_2$,
and the large-scale homoclinic configuration immediately below
$\mbox{HL}_2$ are the same as in the case $\Omega_8$, after replacing
$E_1^\pm$, $H_1$, and $\varphi_1$ by $E_3^\pm$, $H_2$, and $\varphi_2$,
respectively.  We therefore only determine the additional curve
$\mbox{PA}_6$.

\begin{prop}\label{prop:omega12-PA6}
Fix $a_0,a_1,a_2$ with $a_2<0<a_0$ and fix $b_1<0$.  There exists a
unique critical value $b_0^{PA_6}(a_0,a_1,a_2,b_1)$, satisfying
$\frac{a_1b_1}{a_2}-2\frac{a_2}{b_1}\leq b_0^{PA_6}(a_0,a_1,a_2,b_1)
<\varphi_2(a_0,a_1,a_2,b_1)$, such that the
fourth-quadrant stable separatrix $\Gamma_s$ of the origin intersects
the positive $x$-axis when
$b_0^{PA_6}<b_0<\varphi_2$, whereas it connects the origin with the
equilibrium at infinity in the negative $y$-direction when
$b_0\leq b_0^{PA_6}$.
\end{prop}

\begin{proof}
As long as $\Gamma_s$ remains in the fourth quadrant, write it as
$y=y_s(x)$.  By Lemma~\ref{rotate}, for fixed $b_1$ the family is a
strict generalized rotated vector field with respect to $b_0$ in the
right half-plane.  Hence $b_0^{(2)}>b_0^{(1)}$ implies
$y_s(x;b_0^{(2)})>y_s(x;b_0^{(1)})$ on their common fourth-quadrant
domain.  Consequently, the set
of $b_0<\varphi_2$ for which $\Gamma_s$ intersects the positive
$x$-axis is upward closed.  It is also open because every such
intersection is transversal.  Since $\Gamma_s$ forms the right
homoclinic loop on $\mbox{HL}_2$, the same transversal intersection
persists for $b_0<\varphi_2$ sufficiently close to $\varphi_2$.

It remains to show that this set is bounded below.  Put
$C=a_2/b_1>0$ and use $y=-Cx^2$ as a comparison curve.  The local
stable direction gives $y_s(x)=-\sqrt{a_0}\,x+O(x^2)<-Cx^2$ near the
origin, while direct substitution
into the orbital equation gives the important identity
\[
\left.\frac{dy}{dx}\right|_{y=-Cx^2}-\frac{d(-Cx^2)}{dx}
=-\frac{a_0}{Cx}
+\left(b_0-\frac{a_1b_1}{a_2}+2\frac{a_2}{b_1}\right)x.
\]
Hence, if $b_0\leq a_1b_1/a_2-2a_2/b_1$, the right-hand side is
strictly negative for every $x>0$.  A first-contact argument yields
$y_s(x)<-Cx^2<0$ throughout its maximal fourth-quadrant domain.
Therefore $\Gamma_s$ cannot intersect the positive $x$-axis and, by the
phase portrait at infinity, it connects with the equilibrium at infinity
in the negative $y$-direction.

Let $b_0^{PA_6}$ be the infimum of the nonempty open set of parameters
for which $\Gamma_s$ intersects the positive $x$-axis.  The preceding
barrier estimate gives
$b_0^{PA_6}\geq a_1b_1/a_2-2a_2/b_1$, while persistence of the
intersection below $\mbox{HL}_2$ gives $b_0^{PA_6}<\varphi_2$.
At the critical value the separatrix cannot intersect the positive
$x$-axis, since that property is open, and hence it connects to the
negative $y$-direction at infinity.  The strict ordering in $b_0$ proves
uniqueness; the usual continuous-dependence and shooting argument
also shows that $b_0^{PA_6}$ depends continuously on the remaining parameters.
\end{proof}

Thus $I_{42}$ is the region between $\mbox{HL}_2$ and $\mbox{PA}_6$,
where the same large-scale homoclinic configuration as in $I_{23}$
persists, while $\mbox{PA}_6$ and $I_{43}$ have the connection to the
negative $y$-direction at infinity described in the proposition.
Although the global phase portraits of $I_{23}$ and $I_{42}$
differ from the common portrait represented by
$I_{21}\cup I_{40}\cup H_1\cup H_2$, all these portraits are
topologically equivalent.
The case $b_0=b_1=0$, namely $I_{44}$, follows from the Hamiltonian
$H=y^2/2-a_0x^2/2-a_1x^4/4-a_2x^6/6$; the boundary case $I_{45}$ is
obtained in the same way as the corresponding $b_1=0$ case in
$\Omega_8$.  Reversibility and Lemma~\ref{propIn} then determine
all remaining separatrix connections, so no additional global phase
portrait occurs in $\Omega_{12}$.

\subsection{The case $\Omega_{13}$}\label{s7:omega13}

Let $\lambda\in\Omega_{13}$.  Since $a_0<0<a_2$, we have
$\Delta=a_1^2-4a_0a_2>0$, and system~\eqref{main} has exactly three
finite equilibria: the center at the origin and the two saddles
$E_4^{\pm}=\left(\pm\sqrt{\frac{-a_1+\sqrt{\Delta}}{2a_2}},0\right)$.
The periodic orbits surrounding the origin are symmetric with respect
to the $y$-axis.  At the rightmost point $(x,0)$ of any such orbit, we
have $\dot{y}=a_0x+a_1x^3+a_2x^5\leq0$.
Since this expression is positive on the positive $x$-axis to the right
of $E_4^+$, no periodic orbit surrounding the origin can pass beyond
$E_4^+$.  Its limiting boundary cannot be a regular periodic orbit,
because reversibility would extend the family of periodic orbits across
it.  Therefore, the inner separatrices of $E_4^-$ and $E_4^+$ form two
heteroclinic connections bounding the periodic annulus of the origin.

The same sign property shows that the outer separatrices of $E_4^+$
cannot return across the positive $x$-axis.  Lemma~\ref{lem17} excludes
limit cycles in $x>0$, while Lemma~\ref{propIn} forces these
separatrices to connect with the equilibria at infinity in the positive
and negative $y$-directions.  Reversibility determines the corresponding
connections in $x<0$.
Thus the global phase portrait is independent of $b_0$ and $b_1$ and is
topologically equivalent to $I_{29}$, shown in
Figure~\ref{phasea0n0-three}(9).

\subsection{The case $\Omega_{14}$}\label{s7:omega14}

Let $\lambda\in\Omega_{14}$ and put
$r^2:=\frac{-a_1-\sqrt{\Delta}}{2a_2},
R^2:=\frac{-a_1+\sqrt{\Delta}}{2a_2}$.
Then $0<r<R$, the origin and $E_4^\pm=(\pm R,0)$ are saddles, and
$E_3^\pm=(\pm r,0)$ are antisaddles.  Moreover,
\begin{equation}\label{eq:omega14-factorization}
a_0x+a_1x^3+a_2x^5=a_2x(x^2-r^2)(x^2-R^2),
\end{equation}
and the Hopf curve at $E_3^+$ is $H_3:  b_0=h_3(b_1):=
\frac{b_1}{2a_2}(a_1+\sqrt{\Delta})=-b_1r^2$.

{
\begin{prop}[Hamiltonian boundary portraits]
On the Hamiltonian boundary $b_0=b_1=0$, system~\eqref{main} has
exactly the three phase portraits $I_{55}$, $I_{56}$, and $I_{57}$
described below.
\end{prop}

\begin{proof}
On $b_0=b_1=0$ a first integral is
\begin{equation}\label{eq:omega14-Hamiltonian}
H(x,y)=\frac{y^2}{2}-\frac{a_0x^2}{2}
-\frac{a_1x^4}{4}-\frac{a_2x^6}{6}.
\end{equation}
Using $a_0+a_1r^2+a_2r^4=0$ and
$a_0+a_1R^2+a_2R^4=0$, we obtain
\begin{equation}\label{eq:omega14-critical-values}
H(O)=0,\qquad
H(E_3^\pm)=\frac{r^4}{12}(a_1-2\sqrt{\Delta})<0,
\qquad
H(E_4^\pm)=\frac{R^4}{12}(a_1+2\sqrt{\Delta}).
\end{equation}
Put $\sigma:=a_1+2\sqrt{\Delta}$.
Consequently, there are precisely three Hamiltonian configurations.
If $\sigma>0$, the critical level through $O$ is reached
before the critical level through $E_4^\pm$ as the energy increases
from $H(E_3^\pm)$.  The level $H=0$ has the two homoclinic lobes based
at $O$, and the level $H=H(E_4^\pm)$ contains the two heteroclinic
orbits joining $E_4^-$ and $E_4^+$.  This is the portrait $I_{55}$.
If $\sigma<0$, the critical level through $E_4^\pm$ is
reached first; it has one homoclinic loop based at each of
$E_4^-$ and $E_4^+$, whereas the separatrices of $O$ on the level
$H=0$ are unbounded.  This gives $I_{56}$.  Finally, if $\sigma=0$,
the three saddles have the same energy.
The upper branch from $O$ joins $E_4^+$ and the lower branch from
$E_4^+$ returns to $O$; reversibility gives the corresponding two
connections in $x<0$.  This is $I_{57}$.  Since every regular component
of a level of \eqref{eq:omega14-Hamiltonian} is determined by the graph
of $H(x,0)$, these three possibilities exhaust the Hamiltonian phase
portraits.
\end{proof}
}

Related index-theoretic approaches to heteroclinic-orbit bifurcations in
parameter-dependent vector fields can be found in \cite{HuPortaluri2019}.
We next construct the two heteroclinic bifurcation curves.  Denote by
$\Gamma_O^u$ the unstable separatrix of $O$ issuing into the first
quadrant and by $\Gamma_O^s$ its stable separatrix in the fourth
quadrant.  On the left of $E_4^+$, denote the upper stable and lower
unstable separatrices of $E_4^+$ by $\Gamma_4^s$ and $\Gamma_4^u$,
respectively.  Thus an $OE_4$ connection is
$\Gamma_O^u=\Gamma_4^s$, while an $E_4O$ connection is
$\Gamma_4^u=\Gamma_O^s$.

{
\begin{prop}[Existence and monotonicity of $OE_4$ and $E_4O$]
\label{prop:omega14-connections}
For fixed $(a_0,a_1,a_2)$ satisfying the conditions of $\Omega_{14}$
and every $b_1<0$, there exist unique values
$b_0=b_0^{OE_4}(b_1)$ and $b_0=b_0^{E_4O}(b_1)$ for which the
connections $O\to E_4^+$ and $E_4^+\to O$ occur, respectively.
Both functions are $\mathcal{C}^1$ and strictly decreasing with respect to $b_1$.
If these two values coincide, their common parameter belongs to
$\mbox{HL}_3$.
\end{prop}
}

\begin{proof}
Fix $x_*=r$ as a transversal section.  Whenever the four separatrices
under consideration reach this section before meeting the $x$-axis,
denote by $u_O(b_0,b_1)$, $s_4(b_0,b_1)$, $u_4(b_0,b_1)$, and
$s_O(b_0,b_1)$ the $y$-coordinates at which $\Gamma_O^u$,
$\Gamma_4^s$, $\Gamma_4^u$, and $\Gamma_O^s$, respectively, intersect
the transversal section $x=r$; thus $u_O,s_4>0$, whereas $u_4,s_O<0$.
The zeros of
\[
D_{OE}:=u_O-s_4,\qquad D_{E O}:=u_4-s_O
\]
give the two connections.  If a branch turns before reaching the
section, the same shooting functions are continued by recording its
first intersection with the boundary of the strip $0<x<R$.  Such an
intersection is transversal, so the sign of the splitting remains
well defined and is locally constant until a connection occurs.

For fixed $b_1$, Lemma~\ref{rotate} strictly orders all four
separatrices as $b_0$ varies.  For $b_0$ sufficiently large the
splittings have one sign, and for $b_0$ sufficiently negative they
have the opposite sign.  This can also be seen directly from the
orbital equation
\begin{equation}\label{eq:omega14-orbital}
\frac{dy}{dx}=\frac{a_2x(x^2-r^2)(x^2-R^2)}{y}
+x(b_0+b_1x^2):
\end{equation}
on every compact substrip of $0<x<R$, the term $b_0x$ moves the
upper and lower shooting branches strictly through each other as
$b_0\to\pm\infty$.  {For completeness, put
$Y_\pm(x)=\pm Cx(R-x)$.  Direct substitution gives
\begin{equation}\label{eq:omega14-endpoint-barrier}
\left.\frac{dy}{dx}\right|_{y=Y_\pm}-Y_\pm'
=x(b_0+b_1x^2)
\mp\left[\frac{a_2}{C}(x^2-r^2)(x+R)+C(R-2x)\right].
\end{equation}
Choose $C^2>a_2r^2$.  At $x=0$ the right-hand side has a strict sign,
opposite for $Y_+$ and $Y_-$, and this sign persists on a sufficiently
short endpoint interval.  On its compact complement the term $b_0x$
has the required uniform sign for $|b_0|$ large.  The same identity near
$x=R$, followed by a first-contact argument, therefore prevents either
branch from losing its graph without first making the asserted
transversal exit.}  The intermediate
value theorem therefore gives a zero of each splitting, and the strict
rotated-field ordering gives its uniqueness.  Continuous dependence
of invariant manifolds gives continuity of the two critical values.

For completeness, their monotonicity in $b_1$ follows from the same
shooting argument in differential form.  On an orbit graph, the
variation $z_i=\partial y/\partial b_i$ satisfies
\begin{equation}\label{eq:omega14-variation}
z_i'=-\frac{a_2x(x^2-r^2)(x^2-R^2)}{y^2}z_i+x^{2i+1},
\qquad i=0,1.
\end{equation}
{At a zero of $D_*$, let $y=y_*(x)$ denote the corresponding
connection and put
\[
q(x):=a_2x(x^2-r^2)(x^2-R^2),\qquad
K_*(x):=\exp\left\{-\int_x^{x_*}\frac{q(s)}{y_*^2(s)}\,ds\right\}>0.
\]
Apply the variation-of-constants formula first on
$[\varepsilon,x_*]$ and $[x_*,R-\varepsilon]$.  With the orientations
chosen above, the variations from the two endpoints occur with opposite
signs in the branch ordinates and hence with the same sign in their
difference.  Letting $\varepsilon\to0^+$ gives
\begin{equation}\label{eq:omega14-splitting-derivatives}
\frac{\partial D_*}{\partial b_i}
=\epsilon_*\int_0^R x^{2i+1}K_*(x)\,dx,
\qquad i=0,1,\qquad *=OE,E O,
\end{equation}
where $\epsilon_*=1$ for $D_{OE}$ and $\epsilon_*=-1$ for $D_{E O}$.
The improper integrals converge.  Indeed, the hyperbolic saddle
expansions give
\[
y_*(x)=\lambda_Ox+O(x^2),\quad \lambda_O^2=a_0,
\qquad
y_*(x)=\lambda_{4,*}(x-R)+O((x-R)^2)
\]
at $0$ and $R$, respectively.  Consequently,
\[
K_*(x)=O(x)\quad(x\to0^+),\qquad
K_*(x)=O\left((R-x)^{\kappa_*}\right)\quad(x\to R^-),
\qquad
\kappa_*:=\frac{q'(R)}{\lambda_{4,*}^2}>0.
\]
Thus no endpoint term remains in the limiting variation formula.  In
particular, $\partial D_*/\partial b_0\ne0$, so the implicit-function
theorem makes the two critical values $\mathcal{C}^1$, and
\begin{equation}\label{eq:omega14-heteroclinic-slope}
\frac{d b_0^*}{db_1}
=-\frac{\displaystyle\int_0^R x^3K_*(x)\,dx}
        {\displaystyle\int_0^R xK_*(x)\,dx}<0,
\qquad *=OE,E O.
\end{equation}}
{Finally, if both connections occur at the same parameter,
they form the complete heteroclinic boundary
$O\leftrightarrow E_4^+$ surrounding $E_3^+$.  By
Lemma~\ref{lem:omega14-global-boundary} below, such a complete boundary
can occur only at the unique global limit-cycle-breaking value
$b_0=\varphi_3$.  Hence
$\varphi_3=b_0^{OE_4}=b_0^{E_4O}$, and the intersection is the
degenerate point $\mbox{HL}_3^0$ of the curve $\mbox{HL}_3$.}
\end{proof}

{
\begin{lemma}[Extension to $b_1=0$]\label{lem:omega14-extension}
For fixed $(a_0,a_1,a_2)$, the two heteroclinic functions admit unique
$\mathcal{C}^1$ extensions to $b_1=0$.  Their endpoint connections are simple
zeros of the corresponding shooting functions; equivalently,
$\partial D_*/\partial b_0\ne0$ there.  Hence each endpoint connection
continues uniquely to a local $\mathcal{C}^1$ bifurcation curve by the
implicit-function theorem.
\end{lemma}

\begin{proof}
The shooting and endpoint-barrier argument in the proof of
Proposition~\ref{prop:omega14-connections} is unchanged when $b_1=0$
and therefore gives one zero of each splitting on that line.  Formula
\eqref{eq:omega14-splitting-derivatives}, with $i=0$, remains valid and
its integral is nonzero; the endpoint estimates for $K_*$ are uniform
in a small one-sided neighborhood of $b_1=0$.  Thus both zeros are
simple.  The implicit-function theorem gives the asserted local $\mathcal{C}^1$
curves, and uniqueness for $b_1<0$ identifies them with the extensions
of $b_0^{OE_4}$ and $b_0^{E_4O}$.  When $\sigma=0$, the Hamiltonian
connections in $I_{57}$ identify both endpoint values with $b_0=0$.
\end{proof}

\begin{prop}[Global comparison with the Hopf curve]
\label{prop:omega14-global-Hopf-comparison}
For every $b_1<0$, the following assertions hold.
If $a_1+2\sqrt{\Delta}>0$, then
$h_3(b_1)<b_0^{OE_4}(b_1)$, and $E_4O$ meets $H_3$ at most once.
If $a_1+2\sqrt{\Delta}<0$, then
$h_3(b_1)<b_0^{E_4O}(b_1)$, and $OE_4$ meets $H_3$ at most once.
If $a_1+2\sqrt{\Delta}=0$, then
$h_3(b_1)<\min\{b_0^{OE_4}(b_1),b_0^{E_4O}(b_1)\}$; hence neither
heteroclinic curve meets $H_3$ in the half-plane $b_1<0$.
\end{prop}

\begin{proof}
Restrict the two splitting functions introduced above to $H_3$ by
putting $b_0=h_3(b_1)=-b_1r^2$.  At a zero of either restricted
splitting, write the corresponding upper connection as $y=u(x)>0$
and the lower connection as $y=-v(x)<0$, respectively.  From
\eqref{eq:omega14-orbital} they satisfy
\[
 uu^\prime=x(r^2-x^2)\{a_2(R^2-x^2)-b_1u\},\qquad
 vv^\prime=x(r^2-x^2)\{a_2(R^2-x^2)+b_1v\}.
\]
The standard comparison about $x=r$ gives, for either positive graph
$w=u$ or $w=v$,
\begin{equation}\label{eq:omega14-Hopf-graph-moment}
 \int_0^{R^2}(r^2-X)w(\sqrt X)\,dX<0.
\end{equation}
For completeness, put $Y=w^2$ and
$W_w(X,\eta):=a_2(R^2-X)+c_w\sqrt\eta$, where
$c_u=-b_1$ and $c_v=b_1$.  Then
$Y_X=(r^2-X)W_w(X,Y)$ and $W_w>0$.  For
$0<\eta<Y(r^2)$, let $X_-(\eta)<r^2<X_+(\eta)$ be the two points with
$Y(X_\pm)=\eta$.  Changing variables from $X$ to $\eta$ on the two
branches gives
\[
\int_0^{R^2}(r^2-X)\sqrt{Y(X)}\,dX
=\int_0^{Y(r^2)}\!\sqrt\eta
\left(\frac1{W_w(X_-,\eta)}-\frac1{W_w(X_+,\eta)}\right)d\eta<0,
\]
because $X_-<X_+$ implies
$W_w(X_-,\eta)>W_w(X_+,\eta)$.  This proves
\eqref{eq:omega14-Hopf-graph-moment}; strictness is lost only on the
Hamiltonian boundary.

Integrating the two displayed orbital equations from $0$ to $R$ and
using the vanishing of the graphs at both endpoints now shows that an
$OE_4$ connection on $H_3$ requires $R^2<3r^2$, whereas an $E_4O$
connection on $H_3$ requires $R^2>3r^2$.  Indeed, the polynomial part
of either integral is
$R^4(3r^2-R^2)/6$, while the remaining graph moment has the strict
sign in \eqref{eq:omega14-Hopf-graph-moment}.  Since
$a_1+2\sqrt\Delta=a_2(R^2-3r^2)$, this already excludes $OE_4\cap H_3$
when $a_1+2\sqrt\Delta\geq0$ and excludes $E_4O\cap H_3$ when
$a_1+2\sqrt\Delta\leq0$.

It remains to count the possible intersections on the other side.
Along $H_3$, formula \eqref{eq:omega14-splitting-derivatives} gives,
at a zero,
\[
 \begin{aligned}
 \frac{d}{db_1}D_{OE}(h_3(b_1),b_1)
 &=\int_0^R x(x^2-r^2)K_{OE}(x)\,dx>0,\\
 \frac{d}{db_1}D_{E O}(h_3(b_1),b_1)
 &=-\int_0^R x(x^2-r^2)K_{E O}(x)\,dx<0.
 \end{aligned}
\]
Here is the corresponding adjoint comparison.  Put
$\mathcal K(X)=K_*(\sqrt X)$ and
$Z_*(X)=y_*^2(\sqrt X)/(R^2-X)$.  From the definition of $K_*$,
\[
 \mathcal K_X=\frac{a_2(r^2-X)}{2Z_*(X)}\mathcal K.
\]
Thus $\mathcal K$ increases from $0$ to $1$ on $(0,r^2)$ and decreases
from $1$ to $0$ on $(r^2,R^2)$.  If
$\widehat X_-(\eta)<r^2<\widehat X_+(\eta)$ are matched by
$\mathcal K(\widehat X_\pm)=\eta$, the connection equations, written
in $(\eta,Z_*)$ coordinates, give by a first-contact comparison
$Z_*(\widehat X_+)>Z_*(\widehat X_-)$ for $0<\eta<1$.  Therefore
\[
 \int_0^R x(x^2-r^2)K_*(x)\,dx
 =\frac1{a_2}\int_0^1
 \bigl[Z_*(\widehat X_+(\eta))-Z_*(\widehat X_-(\eta))\bigr]d\eta>0.
\]
Thus every zero of the restricted upper
splitting has positive derivative and every zero of the restricted
lower splitting has negative derivative.  Two zeros of the same
function would force an intervening zero with the opposite crossing
direction, so $OE_4\cap H_3$ and $E_4O\cap H_3$ each contain at most
one point whenever they are not already excluded.

Finally, we determine the signs near $b_1=0$.  Along every orbit,
\begin{equation}\label{eq:omega14-energy-balance}
\dot H=x(b_0+b_1x^2)y^2.
\end{equation}
When $R^2\ne3r^2$, the signs follow from
\eqref{eq:omega14-critical-values} and
Lemma~\ref{lem:omega14-extension}.  Hence the absence of an
intersection extends the corresponding local inequality to every
$b_1<0$.

In the equality case $R^2=3r^2$, the upper Hamiltonian connection from
$O$ to $E_4^+$ is
\[
y=Y_0(x):=\sqrt{\frac{a_2}{3}}\,x(R^2-x^2),\qquad 0<x<R.
\]
The first-order splitting equation obtained from
\eqref{eq:omega14-energy-balance} is
\[
b_0A_1+b_1A_3=0,
\quad
A_1=\int_0^R xY_0(x)\,dx
=\frac{2}{15}\sqrt{\frac{a_2}{3}}R^5,
\quad
A_3=\int_0^R x^3Y_0(x)\,dx
=\frac{2}{35}\sqrt{\frac{a_2}{3}}R^7.
\]
The lower Hamiltonian connection gives the same first-order equation.
Thus
\begin{equation}\label{eq:omega14-tangency}
\begin{aligned}
b_0^{OE_4}(b_1)&=-\frac{3R^2}{7}b_1+o(b_1),&
b_0^{E_4O}(b_1)&=-\frac{3R^2}{7}b_1+o(b_1).
\end{aligned}
\end{equation}
Since $h_3=-r^2b_1$ and $R^2=3r^2$, both heteroclinic curves lie
strictly above $H_3$ for all sufficiently small $b_1<0$.  The exclusion
proved above extends both strict inequalities to the whole half-plane.
\end{proof}

\begin{lemma}[Splitting signs at the global boundary]
\label{lem:omega14-global-boundary}
Fix $b_1<0$.  Unless all three values coincide, the two heteroclinic
values lie on opposite sides of $\varphi_3$.  If both splittings vanish
at one value, then that value equals $\varphi_3$ and the resulting
$O\leftrightarrow E_4^+$ heteroclinic cycle is precisely the degenerate
boundary $\mbox{HL}_3^0$.
\end{lemma}

\begin{proof}
Orient the upper and lower splitting functions so that their signs
record whether the corresponding saddle branches enter or leave the
annular strip surrounding $E_3^+$.  Strict rotation in $b_0$ makes each
splitting change sign exactly once, at $b_0^{OE_4}$ and
$b_0^{E_4O}$, respectively.  If these two zeros lay strictly on the
same side of $\varphi_3$, their cyclic order would produce, between the
nearest zero and $\varphi_3$, a trapping annulus for the forward or the
reversed flow.  The Poincar\'e--Bendixson theorem would then give a
periodic orbit in the complementary annular strip.  Outside
$h_3<b_0<\varphi_3$ this contradicts the nonexistence assertion in
Lemma~\ref{lem:limit-cycle-breaking-phi3}; inside that interval it gives
a second cycle in addition to the Hopf cycle, contradicting the uniqueness
assertion in the same lemma.  Thus the two zeros must
straddle $\varphi_3$.

If both splittings vanish together, the boundary of the same annular
strip is the complete heteroclinic cycle
$O\leftrightarrow E_4^+$.  Lemma~\ref{lem:limit-cycle-breaking-phi3}
gives, for fixed $b_1<0$, only one
large-amplitude limit-cycle-breaking value, namely $b_0=\varphi_3$;
the boundary at that value may be homoclinic or, at this degenerate
parameter, heteroclinic.  Hence the common zero must satisfy
$\varphi_3=b_0^{OE_4}=b_0^{E_4O}$ and belongs to
$\mbox{HL}_3^0$.
\end{proof}

\begin{prop}[Ordering and exhaustion of phase portraits]
\label{prop:omega14-ordering}
The Hopf and limit-cycle-breaking curves satisfy $h_3<\varphi_3$ for
$b_1<0$.  Except at a possible common value
$\varphi_3=b_0^{OE_4}=b_0^{E_4O}$, the value $\varphi_3$ lies strictly
between the two heteroclinic values.  The three local orders adjacent
to $b_1=0$ are those displayed in the three-row array
\eqref{eq:omega14-order-table} below.  The comparisons
with $H_3$ along the entire half-plane $b_1<0$ are those of
Proposition~\ref{prop:omega14-global-Hopf-comparison}.  Moreover, under these ordering
conclusions and the limit-cycle uniqueness results of Section~\ref{s5}, the critical-curve orders along every
fixed vertical line $b_1<0$ exhaust all possible $\Omega_{14}$ phase
portraits.
\end{prop}

\begin{proof}
We first compare the curves with $H_3$.  By
Lemma~\ref{lem:omega14-extension}, both connections exist on
$b_1=0$ as simple zeros of their shooting functions.  An $OE_4$
connection on this line satisfies
\[
H(E_4^+)-H(O)=b_0\int_{-\infty}^{+\infty}x(t)y(t)^2\,dt,
\]
whereas an $E_4O$ connection satisfies the same identity with the
left-hand side replaced by $H(O)-H(E_4^+)$.  The integral is positive.
It follows from \eqref{eq:omega14-critical-values} that the
$OE_4$ endpoint value has the sign of $\sigma$, whereas the $E_4O$
endpoint value has the opposite sign.  The simplicity established in
Lemma~\ref{lem:omega14-extension} and the implicit-function theorem
therefore give, for $b_1<0$ sufficiently close to zero,
\begin{equation}\label{eq:omega14-local-order-noncritical}
\begin{array}{ll}
b_0^{E_4O}<h_3<b_0^{OE_4},&\sigma>0,\\[1mm]
b_0^{OE_4}<h_3<b_0^{E_4O},&\sigma<0.
\end{array}
\end{equation}

When $\sigma=0$, equation~\eqref{eq:omega14-tangency}, established in
the proof of
Proposition~\ref{prop:omega14-global-Hopf-comparison}, and
$h_3=-r^2b_1$ show that both heteroclinic curves lie strictly above
$H_3$ for all sufficiently small $b_1<0$.  Thus the three
cases in \eqref{eq:omega14-local-order-noncritical} and
\eqref{eq:omega14-tangency} give, respectively, the parts
$H_3^-$, $H_3^+$, and $H_3^0$ adjacent to the line $b_1=0$.

Proposition~\ref{prop:omega14-global-Hopf-comparison}
upgrades these local comparisons as follows.  For
$a_1+2\sqrt\Delta>0$, $OE_4$ remains strictly above $H_3$, while
$E_4O$ can cross $H_3$ at most once.  For
$a_1+2\sqrt\Delta<0$, $E_4O$ remains strictly above $H_3$, while
$OE_4$ can cross $H_3$ at most once.  When
$a_1+2\sqrt\Delta=0$, both heteroclinic curves remain strictly above
$H_3$; consequently the whole Hopf curve in $b_1<0$ belongs to
$H_3^0$.

\par\smallskip
{The limit-cycle-breaking curve satisfies}
\begin{equation}\label{eq:omega14-H-HL-order}
h_3(b_1)<\varphi_3(b_1),\qquad b_1<0,
\end{equation}
and $\varphi_3(b_1)\to0$ as $b_1\to0^-$.  {By
Lemma~\ref{lem:omega14-global-boundary}, the two heteroclinic values
must lie on opposite sides of $\varphi_3$ unless all three values
coincide.  Therefore,}
\begin{equation}\label{eq:omega14-HL-between}
\min\{b_0^{OE_4},b_0^{E_4O}\}<\varphi_3
<\max\{b_0^{OE_4},b_0^{E_4O}\},
\end{equation}
apart from a possible common value
$\varphi_3=b_0^{OE_4}=b_0^{E_4O}$, which is precisely
$\mbox{HL}_3^0$.

{Combining \eqref{eq:omega14-local-order-noncritical},
\eqref{eq:omega14-tangency}, \eqref{eq:omega14-H-HL-order}, and
\eqref{eq:omega14-HL-between} gives the promised three-row summary for
$b_1\to0^-$:
\begin{equation}\label{eq:omega14-order-table}
\begin{array}{c|c}
\text{sign of }\sigma & \text{local order of the critical curves}\\ \hline
\sigma>0 & b_0^{E_4O}<h_3<\varphi_3<b_0^{OE_4}\\[1mm]
\sigma<0 & b_0^{OE_4}<h_3<\varphi_3<b_0^{E_4O}\\[1mm]
\sigma=0 &
h_3<\min\{b_0^{OE_4},b_0^{E_4O}\}\leq\varphi_3
\leq\max\{b_0^{OE_4},b_0^{E_4O}\}.
\end{array}
\end{equation}
In the last row both non-strict inequalities are strict unless
$\varphi_3=b_0^{OE_4}=b_0^{E_4O}$.

We give the detailed continuation for $\sigma>0$.  By}
\eqref{eq:omega14-local-order-noncritical},
\eqref{eq:omega14-H-HL-order}, and
\eqref{eq:omega14-HL-between}, for $b_1<0$ sufficiently close to zero
the curves have the strict order
\begin{equation}\label{eq:omega14-positive-order}
b_0^{E_4O}<h_3<\varphi_3<b_0^{OE_4}.
\end{equation}
Start with $b_0>b_0^{OE_4}$.  The four saddle separatrices have the
ordering represented by $I_{46}$, and no limit cycle surrounds
$E_3^+$.  When $b_0$ decreases to $b_0^{OE_4}$,
$\Gamma_O^u$ and $\Gamma_4^s$ coincide and give the portrait
$OE_4^-$.  Immediately below this value the connection is broken in
the unique direction prescribed by Lemma~\ref{rotate}; this is
$I_{50}$.

At $b_0=\varphi_3$, {the stable limit cycle} which is continued from
the Hopf bifurcation reaches its saddle separatrix boundary.  In the
present energy ordering that boundary is the homoclinic loop based at
$O$, and the resulting portrait is $\mbox{HL}_3^-$.  For
$h_3<b_0<\varphi_3$ the loop is broken and its interior contains the
{unique stable limit cycle} surrounding $E_3^+$; this gives $I_{51}$.
Uniqueness follows from the limit-cycle results of Section~\ref{s5}, while
Lemma~\ref{rotate} gives the strict change of its amplitude.

At $b_0=h_3$, Lemma~\ref{propE3} gives the {supercritical Hopf}
portrait $H_3^-$.  The limit cycle shrinks to $E_3^+$ as $b_0$
decreases to $h_3$, and no limit cycle remains for
$b_0^{E_4O}<b_0<h_3$; the separatrix ordering in this interval is
$I_{52}$.  At $b_0=b_0^{E_4O}$ the lower branch
$\Gamma_4^u$ joins $\Gamma_O^s$, giving $E_4O^-$.  Below this curve
the connection breaks in the only possible direction and the portrait
is $I_{54}$.  Reversibility determines all the branches in $x<0$, and
Lemma~\ref{propIn} determines the remaining endpoints at infinity.
Thus the succession
\[
I_{46},\ OE_4^-,\ I_{50},\ \mbox{HL}_3^-,\ I_{51},\
H_3^-,\ I_{52},\ E_4O^-,\ I_{54}
\]
is rigorously realized for $b_1\to0^-$.

The case $\sigma<0$ is obtained by the same shooting and
rotated-field argument, with the roles of $OE_4$ and $E_4O$ reversed,
as indicated by the second row of \eqref{eq:omega14-order-table}.  When
$\sigma=0$, equation \eqref{eq:omega14-tangency} replaces the strict
order of the two heteroclinic curves; both start above $H_3$, and
\eqref{eq:omega14-HL-between} gives the alternatives displayed in the
third $\Omega_{14}$ bifurcation diagram.

More precisely, under Proposition~\ref{prop:omega14-connections},
the separation property \eqref{eq:omega14-HL-between}, and the
limit-cycle uniqueness results of Section~\ref{s5}, every fixed vertical
parameter line $b_1<0$ can meet the critical curves only in one of the
orders described above.  Strict rotation between consecutive critical
values then determines the intervening separatrix configurations.
Hence all possible critical-curve orders, and therefore all the
$\Omega_{14}$ phase portraits listed in
Theorem~\ref{thm-phase-portraits}, are exhausted.
\end{proof}
}

\begin{remark}[Numerical illustrations]
The preceding analytic conclusions are also illustrated by numerical
shooting.  First take $(a_0,a_1,a_2)=(4,-5,1)$, for which
$a_1+2\sqrt\Delta=1>0$ and $h_3(b_1)=-b_1$.  The curve $E_4O$
meets $H_3$ at
\[
 b_1\approx-1.53819814,\qquad
 h_3=b_0^{E_4O}\approx1.53819814,\qquad
 b_0^{OE_4}\approx3.19497364.
\]
For instance, at $b_1=-2$ the three values are
$h_3=2$, $b_0^{E_4O}\approx2.09745428$, and
$b_0^{OE_4}\approx4.04770128$.  Thus, as $-b_1$ increases, the
portrait on the Hopf curve changes from $H_3^-$ to $H_3^0$ through
the unique intersection $E_4O\cap H_3$.  According to the convention
in Theorem~\ref{thm-bifurcation}, this intersection belongs to
$E_4O^-$.  Consistently with
Proposition~\ref{prop:omega14-global-Hopf-comparison}, $OE_4$ remains
strictly above $H_3$.

As an example with $a_1+2\sqrt\Delta<0$, take
$(a_0,a_1,a_2)=(2.9,-3.9,1)$ and $b_1=-3$.  Numerical continuation
gives
\[
 h_3=3<b_0^{E_4O}\approx3.40518175
 <\varphi_3\approx3.40533673
 <b_0^{OE_4}\approx4.03863118.
\]
Together with the opposite local order near $b_1=0^-$, this comparison
also implies, by continuity and
Proposition~\ref{prop:omega14-global-Hopf-comparison}, a unique
intersection $OE_4\cap H_3$.  According to the convention in
Theorem~\ref{thm-bifurcation}, this intersection belongs to $OE_4^+$.
Consequently, any value
$3.40533673<b_0<4.03863118$, for example $b_0=3.8$, belongs to
$I_{50}$.  Hence $I_{50}$ is nonempty for some parameters satisfying
$a_1+2\sqrt\Delta<0$.  Moreover, continuity and
\eqref{eq:omega14-HL-between} force an organizing point at which
$b_0^{OE_4}=\varphi_3=b_0^{E_4O}$.  Numerically this point is located
near $(b_1,b_0)=(-1.1184,1.4033)$, providing an explicit numerical
realization of $\mbox{HL}_3^0$.
\end{remark}

\begin{remark}\label{rem:omega14-intersections}
The proof of Proposition~\ref{prop:omega14-ordering} does not provide a
closed expression for either $b_0^{OE_4}(b_1)$ or
$b_0^{E_4O}(b_1)$.
The numerical computations in the preceding remark show
that $OE_4$ and $E_4O$ can intersect in the half-plane $b_1<0$ for
suitable coefficients, and hence that $\mbox{HL}_3^0$ can occur.  This
existence argument is not based on a direct search for
$\mbox{HL}_3^0$: it first locates a parameter value in $I_{50}$ whose
ordering, together with the opposite local ordering near $b_1=0^-$,
forces such an intersection by continuity and
\eqref{eq:omega14-HL-between}.  Nevertheless, a numerical
approximation of one point on $\mbox{HL}_3^0$ is also reported.
For an arbitrary fixed $(a_0,a_1,a_2)$ satisfying the
conditions of $\Omega_{14}$, however, although each curve is strictly
decreasing, this fact alone neither decides whether they intersect in
the $(b_1,b_0)$-plane nor bounds the number of such intersections.
Proposition~\ref{prop:omega14-connections} and
\eqref{eq:omega14-HL-between} show only that every intersection, if it
exists, must be a point of $\mbox{HL}_3^0$.
When $\sigma=0$, the extensions of the two heteroclinic
functions have the common endpoint value $b_0=0$ at $b_1=0$; this
point is classified separately as $HHR=I_{57}$.

The local analysis for $b_1\to0^-$ establishes the occurrence of every
portrait in the $\Omega_{14}$ diagrams other than
$\mbox{HL}_3^0$.  The numerical illustration above
establishes that $\mbox{HL}_3^0$ also occurs for suitable
coefficients.  Continuing the separatrices by the strict generalized
rotated-vector-field ordering exhausts all possible cyclic orders of
the relevant branches.  Accordingly, the completeness of
the bifurcation diagrams refers to the list of possible global phase
portraits and bifurcation configurations; it does not determine, for
every fixed admissible coefficient triple, whether $OE_4$ and $E_4O$
intersect or how many times they intersect in the $(b_1,b_0)$-plane.
\end{remark}

\subsection{The case $\Omega_{15}$}\label{s7:omega15}
Let $\lambda\in\Omega_{15}$ and put
\[
r^2:=\frac{-a_1+\sqrt{\Delta}}{2a_2},\qquad
R^2:=\frac{-a_1-\sqrt{\Delta}}{2a_2}.
\]
Then $0<r<R$, the origin is a center, $E_4^\pm=(\pm r,0)$ are
saddles, and $E_3^\pm=(\pm R,0)$ are antisaddles.  Moreover,
\begin{equation}\label{eq:omega15-factorization}
a_0x+a_1x^3+a_2x^5
=a_2x(x^2-r^2)(x^2-R^2).
\end{equation}
The inner separatrices of $E_4^-$ and $E_4^+$ form the boundary of the
periodic annulus surrounding the origin, by the same sign and
reversibility argument as in the case $\Omega_{13}$.  Although
$\Omega_{15}$ contains five finite equilibria, the arguments determining
the Hopf curve $H_4: b_0=\frac{b_1}{2a_2}(a_1+\sqrt{\Delta})=-b_1R^2$
and the figure-eight homoclinic curve
$\mbox{HL}_4:$ $b_0=\varphi_4(a_0,a_1,a_2,b_1)$ are the same as those in
the cases $\Omega_8$ and $\Omega_{12}$.  The same limit-cycle,
separatrix-ordering, reversibility, and annular-region arguments give
$I_{58}$, $H_4$, $I_{59}$, $\mbox{HL}_4$, and the configuration $I_{60}$.
Although the global phase portrait of $I_{60}$ differs from
the common portrait represented by $I_{58}\cup H_4$, the two portraits
are topologically equivalent.
It remains only to determine exactly when the right-lower stable
separatrix of $E_4^+$ ceases to turn on the positive $x$-axis.

\begin{prop}\label{prop:omega15-PA7}
Fix $a_0,a_1,a_2$ with $a_0<0$, $a_2<0$, $a_1>0$, and $\Delta>0$, and fix $b_1<0$.
Let $\Gamma_s$ be the right-lower branch of the stable manifold of
$E_4^+$, followed backward from $E_4^+$, and define
\[
b_0^{\rm PA_7}:=
\frac{b_1}{2a_2}(a_1+\sqrt{\Delta})-2\frac{a_2}{b_1}
=-b_1R^2-2\frac{a_2}{b_1}.
\]
Then the following statements hold.
\begin{itemize}
\item[(i)] If $b_0>b_0^{\rm PA_7}$, $\Gamma_s$ intersects the positive
$x$-axis transversally.
\item[(ii)] If $b_0=b_0^{\rm PA_7}$, then $\Gamma_s$ is the invariant
algebraic curve
\begin{equation}\label{eq:omega15-critical-separatrix}
y=-\frac{a_2}{b_1}(x^2-r^2),\qquad x\geq r,
\end{equation}
and it connects the equilibrium at infinity in the negative
$y$-direction to $E_4^+$.
\item[(iii)] If $b_0<b_0^{\rm PA_7}$, $\Gamma_s$ remains below the
positive $x$-axis and has the same two endpoints as in~(ii).
\end{itemize}
Consequently, $\mbox{PA}_7$ is precisely the algebraic bifurcation
curve $b_0=b_0^{\rm PA_7}$.
\end{prop}

\begin{proof}
As long as $\Gamma_s$ remains below the positive $x$-axis, write it as
$y=y_s(x)$ for $x>r$.  By \eqref{eq:omega15-factorization}, its orbital
equation is
\begin{equation}\label{eq:omega15-orbital}
\frac{dy}{dx}
=\frac{a_2x(x^2-r^2)(x^2-R^2)}{y}+b_0x+b_1x^3.
\end{equation}
Set $C_*:=a_2/b_1>0$ and
$Y(x):=-C_*(x^2-r^2)$.  Direct substitution into
\eqref{eq:omega15-orbital} gives
\begin{equation}\label{eq:omega15-barrier}
\left.\frac{dy}{dx}\right|_{y=Y(x)}-Y'(x)
=\left(b_0+b_1R^2+2C_*\right)x
=\left(b_0-b_0^{\rm PA_7}\right)x.
\end{equation}
Thus, at $b_0=b_0^{\rm PA_7}$ the curve $Y$ is invariant.  It passes
through $E_4^+$ and has tangent $Y'(r)=-2C_*r<0$.  Substitution of
$-2C_*r$ into the characteristic equation at $E_4^+$ gives
\[
(-2C_*r)^2-r(b_0+b_1r^2)(-2C_*r)
-2a_2r^2(r^2-R^2)=0,
\]
where $b_0=b_0^{\rm PA_7}$ has been used.  Thus $-2C_*r$ is the
negative eigenvalue.
Hence uniqueness of the local stable manifold identifies $Y$ with
$\Gamma_s$.  Since $Y(x)/x\to-\infty$ as $x\to+\infty$,
Lemma~\ref{propIn} identifies its other endpoint with the
equilibrium at infinity in the negative $y$-direction.  This proves~(ii).

If $b_0<b_0^{\rm PA_7}$, Lemma~\ref{rotate} and the strict form of
\eqref{eq:omega15-barrier} give $y_s(x)<Y(x)<0$ throughout the common
fourth-quadrant continuation.  As in the proofs of
Propositions~\ref{prop:separatrix-infinity}(iii) and
\ref{prop:omega11-PA5}(iii), finite-$x$ escape is impossible, so the
graph is defined for every $x>r$ and satisfies $y_s(x)/x\to-\infty$.
Lemma~\ref{propIn} then proves~(iii).

For~(i), it is enough first to consider
$b_0^{\rm PA_7}<b_0<-b_1R^2$.  Here $E_3^+$ is stable by
Lemma~\ref{propE3}.  If $\Gamma_s$ did not intersect the positive
$x$-axis, then it could neither have $E_3^+$ as its $\alpha$-limit nor
escape at a finite value of $x$; hence $Y<y_s<0$ for every $x>r$.
With $w:=y_s-Y>0$ and
$\delta:=b_0-b_0^{\rm PA_7}>0$, subtraction of the two orbital
equations gives
\[
w'=-\frac{a_2x(x^2-r^2)(x^2-R^2)}{y_sY}\,w+\delta x.
\]
Consequently, for all sufficiently large $x$,
\[
\frac{w'}{w}>-\frac{a_2}{C_*^2}
\frac{x(x^2-R^2)}{x^2-r^2}\geq cx
\]
for some $c>0$.  Thus $w$ grows at least exponentially in $x^2$, which
contradicts $0<w<-Y=O(x^2)$.  Therefore $\Gamma_s$ intersects the
positive $x$-axis.  The intersection is transversal, since otherwise it
would occur at the finite equilibrium $E_3^+$.  Finally, the strict
rotated-vector-field ordering of Lemma~\ref{rotate}, used in the same
way as in the proofs of Propositions~\ref{prop:separatrix-infinity} and
\ref{prop:omega11-PA5}, shows that this intersection property is
preserved as $b_0$ increases.  Since
$b_0^{\rm PA_7}<-b_1R^2$, statement~(i) follows for every
$b_0>b_0^{\rm PA_7}$.
\end{proof}

Proposition~\ref{prop:omega15-PA7} gives the exact curve
$\mbox{PA}_7:b_0=\frac{b_1}{2a_2}(a_1+\sqrt{\Delta})-2\frac{a_2}{b_1}$.
For $b_0^{\rm PA_7}<b_0<\varphi_4$, the relevant outer separatrices
turn at finite points and reversibility produces the finite separatrix
boundary and its exterior periodic annulus; this is $I_{60}$.  On
$\mbox{PA}_7$ and in $I_{61}$, the connection to the negative
$y$-direction at infinity described above destroys that outer periodic
annulus.  When $b_0=b_1=0$, the Hamiltonian
$H=y^2/2-a_0x^2/2-a_1x^4/4-a_2x^6/6$ gives $I_{62}$; the usual
$b_1=0$, $b_0<0$ turning argument gives $I_{63}$.  Reversibility and
Lemma~\ref{propIn} determine all remaining branches.  Hence
$I_{58}$--$I_{63}$, $H_4$, $\mbox{HL}_4$, and $\mbox{PA}_7$ exhaust
the global phase portraits in $\Omega_{15}$ shown in
Figure~\ref{phasea0n0-five}.

\section*{Acknowledgements}
\begin{sloppypar}
The first and third authors are partially supported by the National Key
R\&D Program of China (No.~2022YFA1005900) and the National Natural Science
Foundation of China (No.~12671199). The fourth author is supported by the National Natural Science
Foundation of China (No.~12471155).
\end{sloppypar}

\end{document}